\documentclass[12pt]{article}
\usepackage[utf8]{inputenc}
\usepackage[T1]{fontenc}
\usepackage{amsmath}
\usepackage{amsfonts}
\usepackage{amssymb}
\usepackage[version=4]{mhchem}
\usepackage{stmaryrd}
\usepackage{bbold}
\usepackage{amstext}
\usepackage{amsthm}
\usepackage{colortbl}
 
\usepackage{geometry}

\usepackage{comment}

\usepackage[colorlinks=true,citecolor=blue,urlcolor=blue]{hyperref}
\usepackage[numbers]{natbib}
\newtheorem{theorem}{Theorem}

\newtheorem{corollary}[theorem]{Corollary}

\newtheorem{definition}[theorem]{Definition}

\newtheorem{lemma}[theorem]{Lemma}

\newtheorem{proposition}[theorem]{Proposition}

\DeclareMathOperator{\var}{var}
\DeclareMathOperator{\cov}{cov}

\newcommand{\numberthis}{\addtocounter{equation}{1}\tag{\theequation}}
\newcommand{\cX}{\mathcal{X}}
\newcommand{\cZ}{\mathcal{Z}}

\newcommand{\cT}{\mathcal{T}}

\newcommand{\cB}{\mathcal{B}}
\newcommand{\cH}{\mathcal{H}}
\newcommand{\R}{\mathbb{R}}

\newcommand{\cF}{\mathcal{F}}
\newcommand{\PP}{\mathbb{P}}
\newcommand{\NN}{\mathbb{N}}
\newcommand{\EE}{\mathbb{E}}
\newcommand{\norm}[1]{\left\| #1 \right\|}

\title{Concentration of the bootstrap empirical process, with applications to statistical inference}

\author{Guillaume Maillard, Adrien Saumard\\
\large \href{mailto:guillaume.maillard@ensai.fr}{guillaume.maillard@ensai.fr} ; \href{mailto:adrien.saumard@ensai.fr}{adrien.saumard@ensai.fr} \\
\large University of Rennes, ENSAI, CNRS, CREST-UMR 9194, F-35000 Rennes, France}

\begin{document}

\maketitle

\begin{abstract}
Considering a general framework of bootstrap with exchangeable weights, we show some concentration inequalities for the supremum of the bootstrap empirical process.  On the one hand, we discuss the concentration of the bootstrap empirical process around its conditional expectation with respect to the original data, and on the other hand, the concentration of the latter quantity around its mean. For the concentration conditional on data, we build on Chatterjee's exchangeable pairs approach to concentration. To attain optimal concentration rates, we develop some refined arguments for the convergence of transposition walks on the symmetric group. The conditional expectation of the bootstrap empirical process is proved to be self-bounding, thus extending a well-known property for conditional Rademacher averages. To illustrate the interest of these concentration inequalities, we provide some new results pertaining to confidence regions for the estimation of a mean vector, as well as non-asymptotic bounds for the two-sample permutation test.
\end{abstract}


\tableofcontents

\section{Introduction}

Bootstrap techniques are central tools for statistical inference, since they provide generic, data-dependent ways of approximating the sampling distribution of a statistic. Such information indeed allows one to build confidence intervals or to calibrate statistical tests, especially when the limiting sampling distribution of the statistic of interest is unknown or intractable. 

The asymptotic validity of non-parametric bootstrap schemes has been addressed in wide generality in the context of the classical empirical process theory, where the ambiant dimension is fixed. More precisely, considering Efron's empirical bootstrap (\cite{Efron79}), Giné and Zinn \cite{GinZinn:90} characterized the functional central limit theorem and the uniform law of large numbers for the bootstrap empirical process indexed by a general class of functions, conditional on the data sample (both almost surely and in probability). Pr{\ae}stgaard and Wellner \cite{PraestWellner:93} -- see also \cite{vandervaartWellner:23} -- also proved conditional (again, almost surely and in probability) Donsker-type theorems and uniform laws of large numbers for the exchangeable bootstrap, under general moment conditions. 

Recently, efforts have been made to extend the bootstrap theory to high-dimensional situations, where the dimension grows with the sample size. In a series of papers,  Chernozhukov, Chetverikov, Kato and Koike notably obtained central limit theorems and non-asymptotic approximation rates for the wild and empirical bootstraps, with logarithmic dependence in the dimension (\cite{CCK13,CCK15,CCK16,CCK17,CCK23}, see also \cite{CCKK2023} for a recent survey with further references). The proofs developed by these authors are based on Gaussian approximations of suprema of empirical processes, Gaussian anti-concentration and comparison inequalities. Applications include multiple testing, confidence regions for high-dimensional vectors, post-selection inference, to name but a few. Further statistical outcomes are discussed \cite{CCKK2023}.

In this article, we prove concentration inequalities for the supremum of the bootstrap empirical process, for both Efron's empirical bootstrap and the exchangeable bootstrap. To our best knowledge, such concentration inequalities are new. Indeed, as further explained in Section \ref{sec_general_concen}, this amounts to proving concentration inequalities for suprema of weighted sums of random variables, where the weights have the essential property of being exchangeable. Considering a weighted sum of exchangeable variables, Foygel Barber \cite{Foy24} recently obtained Hoeffding and Bernstein-type concentration inequalities. In terms of exchangeable bootstrap, the framework considered in \cite{Foy24} corresponds to analyzing the concentration conditionnally on data. 
But the fundamental difference with our problem is that there is only one sum that is considered in \cite{Foy24}, corresponding to the degenerate situation where the supremum of the bootstrap empirical process is taken over a singleton. Nonetheless, we recover in a bounded setting a sub-Gaussian behavior (Hoeffding-type concentration) for the supremum of the exchangeable bootstrap empirical process, as can be seen in Theorem \ref{thm_intro_concen_exchange} below.


In order to give a clear idea of the concentration inequalities obtained in this article, let us state here two of our main results. We start with the concentration of the expectation of the supremum of the exchangeable weighted empirical process, conditioned on data. 

When the weights are i.i.d. Rademacher variables, it is well-known that such a statistic, usually termed the conditional Rademacher average, is a self-bounding function \cite{BoucLugMas00, BouLugMas_SB_09} (see also \cite{McDiaReed06,Maurer06, BoucheronLugosiMassart:2013}). Consequently, it satisfies a Poisson-type concentration inequality, the variance proxy for the sub-Gaussian regime being proportional to the expectation. For a definition of the self-bounding property, see Section \ref{sec_general_concen} below, following \cite{BoucheronLugosiMassart:2013}. Considering exchangeable weights having a finite first moment, we extend the self-bounding property to the conditional expectation of the supremum of the exchangeable bootstrap process.

More precisely, let $(\xi_i)_{1\leq i \leq n}$ be exchangeable weights such that $\EE[\vert \xi_1 \vert]=\kappa<+\infty$ and
\[ \sum_{i = 1}^n \xi_i = 0\quad a.s. \]
Let $(E,\cB)$ be a measure space and let $\cT$ be a set of measurable functions from $E$ to $[-1,1]$. We define the function $\overline{g}$ on the set $E^n$ by
\[ \overline{g}((x_i)_{1 \leq i \leq n}) = \mathbb{E} \left[ \sup_{t \in \cT} \sum_{i = 1}^n \xi_i t(x_{i}) \right]. \]
Note that this expectation is indeed well-defined.
If $(x_{i})_{1 \leq i \leq n}$ is a realization of $n$ i.i.d. random variables, then $\overline{g}((x_{i})_{1 \leq i \leq n})$ represents the bootstrap estimate of the expected supremum.

\begin{theorem} \label{thm_intro_cond_exp}
    If $X = (X_{i})_{1 \leq i \leq n}$ is a collection of $n$ independent random variables valued in $E$ and such that $\overline{g}(X)$ is measurable, then for any $x > 0$,
    \[ \overline{g}(X) \leq \mathbb{E}\left[ \overline{g}(X) \right] + \sqrt{12\kappa x \mathbb{E}[\overline{g}(X)]} + 5\kappa x \]
    with probability at least $1 - e^{-x}$. Moreover,
    \[ \overline{g}(X) \geq \mathbb{E}\left[ \overline{g}(X) \right] - \sqrt{12\kappa x \mathbb{E}\left[ \overline{g}(X) \right]} \]
    with probability at least $1 - e^{-x}$.
\end{theorem}

We consider now the concentration of the supremum of the bootstrap empirical process, conditionally on the data. Note that for Efron's bootstrap, the bootstrap empirical process is, conditionally on the data, a standard empirical process (see Section \ref{sec_bootep}), the concentration of which is alrady captured by classical results, such as Bousquet's (\cite{Bous02}) and Klein-Rio's inequalities (\cite{KleinRio2005}).
Hence, we focus on the generic exchangeable bootstrap, that necessitates new inequalities. 

First note that exchangeability of the weights $(\xi_i)_{1\leq i \leq n}$ ensures that, for any random permutation $\sigma$ independent of the weights, the distribution of $(\xi_{\sigma(i)})_{1\leq i \leq n}$ is identical to the distribution of $(\xi_i)_{1\leq i \leq n}$. Actually, letting a uniform random permutation act on the weights indices and conditioning on the values of the weights will allow us to prove a sub-Gaussian concentration inequality. For convenience, we rather denote in this context the (fixed) values of weights by $w=(w_i)_{1\leq i \leq n}$.

\begin{theorem}\label{thm_intro_concen_exchange}
    Let $a < 0 < b$ be two real numbers and let $w \in [a;b]^n$ be such that $\sum_{i = 1}^n w_i = 0$. For any $(x_{i})_{1\leq i \leq n} \in E^n$, let
    \[ g_x(\sigma) = \sup_{t \in \cT} \left\{ \sum_{i = 1}^n w_{\sigma(i)} t(x_{i}) \right\}. \]
    Let $\sigma \in \mathfrak{S}_n$ be a uniform random permutation.
    For any $u > 0,$  with probability at least $1-e^{-u}$,
    \begin{equation*}
        g_x(\sigma) - \mathbb{E}[g_x(\sigma)]\leq c \min\left\{(b-a)\sqrt{v_+(x)},\Vert w\Vert_2\right\}\sqrt{u}\;,
    \end{equation*}
   where $c$ is a numerical constant ($c=9$ holds)
    and $v_+(x)$ is the so-called ``weak empirical variance'' linked to the empirical process,
    \[ v_+(x) = \sup_{t \in \cT} \left\{ \sum_{i = 1}^n \left(t(x_{i}) - \frac{1}{n} \sum_{j = 1}^n t(x_{j}) \right)^2 \right\}. \]
\end{theorem}

Compared to classical results pertaining to the concentration of the supremum of the empirical process around its mean, difficulties for proving concentration of its exchangeable bootstrap version are essentially due to the dependencies between the weights. Indeed, as exposed extensively in \cite{BoucheronLugosiMassart:2013}, classical techniques based on modified log-Sobolev inequalities, such as the one used in the proofs of Bousquet's and Klein-Rio's inequalities (\cite{Bous02, KleinRio2005}) for the supremum of the empirical process, heavily rely on independence between random variables in the sums, and on tensorization properties in particular.

As further explained in Section \ref{sec_general_concen}, we prove Theorem \ref{thm_intro_concen_exchange} by making use of a so-called ``exchangeable pair'' technique, which has its origin in the work of Stein (\cite{Stein86,Stein92}, see also \cite{ChenGoldShao11}). In the context of concentration inequalities, the exchangeable pair approach was initiated  by Chatterjee in his pathbreaking PhD thesis (\cite{Chatterjee2005thesis}, see also \cite{Chatterjee07,ChatDey:10}). As emphasized for instance in \cite{MR3727600}, one major strength of the exchangeable pair approach to concentration of measure is that it typically allows one to deal with dependent and complex systems of random variables, such as in statistical mechanics or in random graph theory (\cite{ChatDey:10,Paulin14}). Chatterjee's approach has also been successfully adapted to concentration of random matrices \cite{Mackeyetal14,PaulinMackTropp16}.  

In Theorem \ref{thm_intro_concen_exchange} above, the statistic of interest is a functional of a random permutation. As detailed in Section \ref{sec_cond_X}, we implement the exchangeable pair technique through the use of refined arguments by White \cite{White2019}, who studies the strong mixing time of a random walk on the symmetric group converging to the uniform distribution. Actually, using that strong mixing time would induce an extra logarithmic factor depending on the sample size in our concentration bounds, so we rather exploit a weaker notion, that we call the \textit{strong convergence time}. 



Theorems \ref{thm_intro_cond_exp} and \ref{thm_intro_concen_exchange}  are general devices, that may be helpful in many statistical contexts. As a proof of concept, we use these results to obtain new insights in two applications: multivariate mean confidence intervals and two-sample testing, see Sections \ref{sec_conf_reg} and \ref{sec_two_sample_test} respectively. 


The paper is organized as follows. We present in Section \ref{sec_bootep} the objects of interest and some known results, especially concerning bounds for the mean of the bootstrap empirical process. Then Sections \ref{sec_exch_self_bounding} and \ref{sec_cond_X} focus respectively on the concentration of the conditional mean, with respect to the data, of the supremum of the bootstrap empirical process and the concentration conditional to the data. In particular, we describe in Section \ref{sec_general_concen} some essential steps of our strategy for proving a concentration inequality, conditional to the data, for the supremum of the bootstrap empirical process. In addition, we provide a general bound combining these two aspects in Section \ref{sec_gen_bound}. Finally, in Sections \ref{sec_conf_reg} and \ref{sec_two_sample_test} respectively, we establish new multivariate confidence regions for the mean and new non-asymptotic bounds for the power of various non-parametric two-sample tests. Technical parts of the proofs and further remarks are deferred to the Supplementary Material.

\section{The bootstrap empirical process} \label{sec_bootep}

Let $(E,\cB)$ be a measurable space and let $\cT$ be a bounded class of measurable functions from $E$ to $\mathbb{R}$. Define a ``distance'' between two probability distributions $P,Q$ on $(E,\cB)$ by
\[ d_{\cT}(P,Q) = \sup_{t \in \cT} \left\{ \int t dP - \int t dQ \right\}. \]
Note that if $\cT$ is symmetric and ``large enough'', so that the separation axiom is satisfied for $d_\cT$, then $d_\cT$ is a distance. 
In that case, the distance $d_\cT$ is known as an ``integral probability metric'' and can be used to construct two-sample tests (known as MMD tests \cite{MMD-Gretton2012}, see also Section \ref{sec_MMD_tests} below) as well as \emph{minimum distance estimators} based on minimizing the distance of the empirical measure to a model, such as the \emph{minimum Wasserstein estimator} \cite{mwde2019}. 
Let $X_1,\ldots,X_n$ be i.i.d. random variables belonging to $E$ and denote by $P_n$ the empirical measure.

Consider the quantity
\[ d_\cT(P_n,P) = \sup_{t \in \cT} \left\{ (P_n - P)(t) \right\}\, \]
\textit{i.e.} the supremum of the empirical process based on the class $\cT$, where for any probability measure $Q$ with an integrable function $t$, we set $Q(t)=\int t dQ$.
The quantity $d_\cT(P_n,P)$ plays a key role in controlling the statistical error of the MMD tests and minimum distance estimators based on the distance $d_\cT$. It also controls the generalization error of M-estimators in statistical learning, when $\cT$ is the class of functions
\[  \left\{ (x,y) \mapsto c(s(x),y) : s \in m \right\} \; ,\]
where $c$ is the contrast function and $m$ is some model for the predictor (classifier or regression function). For these applications, it is important to be able to estimate the expectation and/or the quantiles of the random variable $d_\cT(P_n,P)$.

In this article, we consider the use of the bootstrap to approximate the (unknown) distribution of $d_\cT(P_n,P)$. The bootstrap heuristic applied to $d_\cT(P_n,P)$ consists in using the conditional distribution of
\[ c_W d_\cT(P_n^W,P_n) = c_W\sup_{t \in \cT} \left\{ (P_n^W - P_n)(t) \right\}  \]
knowing $P_n$ as a proxy for the unconditional distribution of $d_\cT(P_n,P)$. Here, $P_n^W$ is a \emph{bootstrap empirical measure} based on the empirical measure $P_n$ and auxiliary (independent) randomization $W$, while $c_W$ is a constant depending only on the distribution of $W$. In Efron's original approach \cite{Efron79}, $P_n^W$ is obtained by sampling $n$ times with replacement from $P_n$, leading to a bootstrap empirical measure of the form
\begin{equation} \label{eq_bootstrap_distrib}
   P_n^W = \frac{1}{n} \sum_{i = 1}^n W_i \delta_{X_i}, 
\end{equation}
where $W$ is a multinomial random vector with success probabilities $1/n$. More generally, the \emph{exchangeably weighted bootstrap} \cite{MasonNewton1992, PraestWellner:93} consists in any measure of the form \eqref{eq_bootstrap_distrib} with weights $W$ such that
\begin{itemize}
    \item $\frac{1}{n} W$ is a probability vector,
    \item $W$ is independent from the data,
    \item $W$ is exchangeable.
\end{itemize}

We can therefore rewrite $d_\cT(P_n^W,P_n)$ as
\[ d_{\cT}(P_n^W,P_n) = \frac{1}{n} \sup_{t \in \cT} \left\{ \sum_{i = 1}^n (W_i-1) t(X_i)  \right\} = \frac{1}{n} \sup_{t \in \cT} \left\{ \sum_{i = 1}^n \xi_i t(X_i)  \right\} \]
where $\xi = W-1$ is an exchangeable random vector which sums to $0$. This leads to the following definition.

\begin{definition} \label{def_g_bootstrap}
We denote by    $X = (X_1,\ldots,X_n)$ a sample of $n$ independent random variables valued in $(E,\cB)$. the random vector $\xi \in \mathbb{R}^n$ is exchangeable, independent from $X$, satisfying $\EE[\vert \xi_1 \vert]<+\infty$ and $\sum_{i = 1}^n \xi_i = 0$ almost surely. We also set
    \[ g(X,\xi) = \sup_{t \in \cT} \left\{ \sum_{i = 1}^n \xi_i t(X_i) \right\} . \]
\end{definition}

Note that we do not require $X$ to be an i.i.d. vector, unlike the bootstrap setting described above. This generalization allows to treat the two-sample case, where variables are not i.i.d. under the alternative. Measurability issues seem similar to those that arise for the empirical process, let us only mention the following sufficient condition.

\begin{lemma}\label{lemma_meas}
    If $\cT$ is a set of measurable functions from $E$ to $[-1,1]$ which is separable in the product topology on $\left[-1, 1\right]^E$, then $g$ is measurable from $E^n \times \mathbb{R}^n$ to $\mathbb{R}$ and in particular, $g(X,\xi)$ is a random variable. 
\end{lemma}
For completeness, we give a proof of Lemma \ref{lemma_meas} in Section \ref{sec_proof_meas} of the Supplementary Material. By default, we will always assume in this article that $\cT$ is indeed separable for the product topology in $\left[-1, 1\right]^E$, which ensures the measurability of $g$.

According to the bootstrap heuristic, the distribution of $g(X,\xi)$ knowing the sample $X$ approximates the distribution of $nd_\cT(P_n,P)$ up to a universal constant. In particular, the conditional expectation $\mathbb{E}\left[ g(X,\xi) | X \right]$ should approximate $n \mathbb{E}[d_\cT(P_n,P)]$, the expected supremum of the empirical process, up to a universal constant. Let us introduce notation for these quantities.

\begin{definition} \label{def_gbar_bootstrap}
    For any integer $n$ and distribution $P$ on $(E,\cB)$, let
    \begin{equation}
        M_n(P) = \mathbb{E} \left[ n d_\cT(P_n,P) \right] = \EE\left[\sup_{t\in\cT}\left\{\sum_{i=1}^n t(X_i)-\EE[t(X)]\right\}\right]\;,
    \end{equation} 
    where $P_n$ is the empirical measure based on an i.i.d. sample of size $n$ from $P$. Let also $\overline{g}$ be the function defined on $E^n$ by
    \[ \overline{g}(x) = \mathbb{E} \left[ \sup_{t \in \cT} \left\{ \sum_{i = 1}^n \xi_i t(x_i) \right\} \right],  \]
    where $\xi$ satisfies the assumptions of definition \ref{def_g_bootstrap}. 
\end{definition}

Let us now give a brief overview of the problem and known results.
Note that $\overline{g}(X) = \mathbb{E}[g(X,\xi) | X]$ for $X,\xi,g$ given by Definition \ref{def_g_bootstrap}. In that case, the bootstrap heuristic
states that $c_W \overline{g}(X)$ approximates the constant $M_n(P)$ in distribution. This assertion can be decomposed into two claims: firstly, the expectation $\mathbb{E}[\overline{g}(X)] = \mathbb{E}[g(X,\xi)]$ is approximately $c^{-1}_W M_n(P)$ for some universal constant $c_W$, secondly $\overline{g}(X)$ concentrates around its expectation.  

Concerning the first claim, upper and lower bounds for $\mathbb{E}[g(X,\xi)]$ as a function of $M_n(P)$ have been investigated in the literature. In the i.i.d. setting, Fromont \cite{Fromont2007} establishes the universal lower bound $\mathbb{E}[g(X,\xi)] \geq \mathbb{E}[(\xi_1)_+] M_n(P)$. Arlot \cite{arlot-these} observes that this lower bound can be improved to $\mathbb{E}[g(X,\xi)] \geq \mathbb{E}[|\xi_1|] M_n(P)$ under a symmetry assumption on the class $\mathcal{T}$ but that it is tight in general.   For i.i.d. data,  Han and Wellner \cite{HanWell:19} provide a sharp upper bound that only depends on the tails of the weights and on the supremum of the corresponding Rademacher process, where the original weights are replaced by independent Rademacher variables. Exchangeability of the weights is not actually needed for that bound.  For the sake of completeness, we provide in Section \ref{ssec_mean} of the Supplementary Material, some upper and lower bounds that are instrumental in our applications. 

The second claim, \textit{i.e.} the concentration of $\overline{g}(X)$ around its expectation, is the subject of section \ref{sec_exch_self_bounding} of this article. Together with the results described in the latter paragraph, it validates to some extent the bootstrap heuristic for the approximation of the expected supremum $M_n(P)$ in the non-asymptotic setting.

The bootstrap heuristic applies in principle to the whole distribution of $n d_\cT(P_n,P)$, not just its expectation. Bousquet's inequality (\cite{Bous02}) shows that $n d_\cT(P_n,P)$ concentrates at the right of its expectation $M_n(P)$, so by the bootstrap heuristic, the same ought to be true of the law of $g(X,\xi)$ knowing $X$. In the case of Efron's bootstrap, this is indeed the case and can easily be proved, actually using Bousquet's inequality and the fact that $P_n^W$ is the empirical measure of an i.i.d. sample drawn from $P_n$. More precisely, we have the following result.

\begin{theorem} \label{thm_bootstrap_efron}
Let $x \in (\mathbb{R}^{\cT})^n$ and let $\xi$ be the Efron weights, \textit{i.e.} $\xi_i=W_i-1$ for any $i\in \left\{ 1,\ldots, n\right\}$, where the random vector $W$ follows a multinomial distribution with success probabilities $1/n$.
    Define the ``empirical weak variance'',
    \[ v_+(x) = \sup_{t \in \cT} \left\{ \sum_{i = 1}^n \left(t(x_{i}) - \frac{1}{n} \sum_{j = 1}^n t(x_{j}) \right)^2  \right\}\;. \]
    For any $\lambda > 0$,
    \[ \mathbb{E} \left[ \exp(\lambda (g(x,\xi) - \overline{g}(x))) \right] \leq \exp \left( [2\overline{g}(x) + v_+(x)] \phi(\lambda)  \right), \]
    where $\phi$ is the function $\phi(u) = e^u - u - 1$.
\end{theorem}
Theorem \ref{thm_bootstrap_efron} is a direct application of Bousquet's inequality for the concentration of the empirical process, and we omit its proof.

In the general case, $g(x,\xi)$ is of the form
\[ \sup_{a \in A} \left\{ \sum_{i = 1}^n a_i \xi_i \right\} \]
for some bounded set $A \subset \mathbb{R}^n$. The random vector $\xi$ is exchangeable and sums to $0$, which means that it can be written in the form $\xi = Z_\sigma$, where $\sigma$ is a uniform random permutation and $Z$ is an arbitrary random vector valued in the set 
\[S = \left\{ z \in \mathbb{R}^n : z_1 \leq z_2 \leq ... \leq z_n , \sum_{i = 1}^n z_i = 0 \right\}. \]
Since $Z$ may be arbitrary, concentration of $g(x,\xi)$ around $\overline{g}(x)$ cannot hold under only the assumptions of definition \ref{def_g_bootstrap}. Indeed, if $g(x,\xi)$ concentrates around $\overline{g}(x)$ for some $\xi$, one can replace $\xi$ by $\Lambda \xi$ for some non-negative random variable $\Lambda$. This yields $g(x,\Lambda \xi) = \Lambda g(x,\xi) \approx \Lambda \overline{g}(x)$,
which fails to concentrate around its expectation.

Rather than concentration around $\overline{g}(x)$, we study in Section \ref{sec_cond_X} the concentration of $g(x,Z_\sigma)$ around $\overline{g}_Z(x) = \mathbb{E}[g(x,Z_\sigma) | Z]$, which amounts to setting $\xi = w_\sigma$ for some fixed vector $w\in S$. This is indeed the setting of Theorem \ref{thm_intro_concen_exchange}. 
As $w_\sigma$ is exchangeable and sums to $0$, the same concentration result (Theorem \ref{thm_intro_cond_exp}) applies to $\overline{g}(X)$ and to $\overline{g}_Z(X)$ conditionally on $Z$, yielding concentration of $\overline{g}_Z(X)$ around $\mathbb{E}\left[ g(X,Z_\sigma) | Z \right]$. Moreover, by the same argument, bounds for $\mathbb{E}\left[ g(X, \xi) \right]$, such as Propositions \ref{prop_lower_mean_boot} and \ref{prop_upper_mean_boot}, also apply conditionally on $Z$. Thus, we can to a large extent treat $g(X,\xi)$ as if $\xi$ were of the form $w_\sigma$ for some fixed vector $w$ with sum $0$. This approach is worked out in Section \ref{sec_gen_bound} and yields general deviation upper bounds for $g(X,\xi)$.     

\section{Self-bounding property of the expectation conditioned on data}\label{sec_exch_self_bounding}

Let $g,X,\xi$ and $\overline{g}$ be given by Definitions \ref{def_g_bootstrap} and \ref{def_gbar_bootstrap}. In order to establish the concentration properties of $\overline{g}(X)$ for a vector of independent variables $X = (X_1,\ldots,X_n)$, we will rely on the notion of a \emph{self-bounding function} (see \cite{BoucheronLugosiMassart:2013}).
For completeness, we state the following definition (corresponding more precisely to the notion of a ``strongly $(a,b)-$self-bounding function'' in \cite{BoucheronLugosiMassart:2013}).

\begin{definition} \label{def_self_bounding}
    Let $E$ be a set and $n \geq 1$ be an integer. A function $f: E^n \to \mathbb{R}$ is said to be \emph{$(a,b)$-self-bounding} for $a,b \geq 0$ if for all $i \in \{1,\ldots,n\}$ there exists a function $f_i: E^{n-1} \to \mathbb{R}$ such that for all $x \in E^n,$
    \[ 0 \leq f(x) - f_i(x_{(i)}) \leq 1 \text{ and } \sum_{i = 1}^n f(x) - f_i(x_{(i)}) \leq a f(x) + b \]
    where for all $i$, $x_{(i)} = (x_1,\ldots,x_{i-1},x_{i+1},\ldots,x_n)$.
\end{definition}

Note that this definition makes sense for any function $f$, not necessarily measurable. However, to derive concentration inequalities, measurability of $f$ and $f_i$ is required. 
For this, the following lemma is useful.

\begin{lemma} \label{lem_meas}
    If $(X_i)_{1 \leq i \leq n}$ are independent random variables valued in $(E, \mathcal{B})$ for some sigma-algebra $\mathcal{B}$ and if $f$ is an $(a,b)-$self-bounding function such that $f(X)$ is integrable, where $X = (X_1,\ldots,X_n)$, then there exists measurable functions $\tilde{f}: E^n \to \mathbb{R}$ and $\tilde{f}_i: E^{n-1} \to \mathbb{R}$ such that $f(X) = \tilde{f}(X)$ almost surely,
    \[ 0 \leq \tilde{f}(X) - \tilde{f}_i(X_{(i)}) \leq 1 \text{ and } \sum_{i = 1}^n \tilde{f}(X) - \tilde{f}_i(X_{(i)}) \leq a \tilde{f}(X) + b \;,\]
    where $X_{(i)} = (X_1,\ldots,X_{i-1},X_{i+1},\ldots,X_n)$.
\end{lemma}

We now come to the main theorem of this section.

\begin{theorem} \label{thm_self_bounding}
    Let $\kappa = \mathbb{E}[|\xi_1|]$ and recall that the functions in $\cT$ are valued in $[-1,1]$. The function 
    \[ \frac{\overline{g}}{\kappa \left(3 - \frac{2}{n} \right)}: E^n \to \mathbb{R} \]
    is $(2,0)$-self-bounding. As a consequence, if $X = (X_{i})_{1 \leq i \leq n}$ is a collection of $n$ independent $E$-valued random variables such that $\overline{g}(X)$ is measurable, then for any $x > 0$,
    \[ \overline{g}(X) \leq \mathbb{E}\left[ \overline{g}(X) \right] + \sqrt{12\kappa x \mathbb{E}[\overline{g}(X)]} + 5\kappa x \]
    with probability at least $1 - e^{-x}$. Moreover,
    \[ \overline{g}(X) \geq \mathbb{E}\left[ \overline{g}(X) \right] - \sqrt{12\kappa x \mathbb{E}\left[ \overline{g}(X) \right]} \]
    with probability at least $1 - e^{-x}$.
\end{theorem}

Our principal contribution in Theorem \ref{thm_self_bounding} is the self-bounding property of $\overline{g}$; the rest follows from known properties of self-bounding functions (\cite[Theorems 6.20,6.21]{BoucheronLugosiMassart:2013}). 
Beyond exchangeability, it worth noting that the only constraint on the weights is integrability, which seems minimal and holds for all resampling weights used in practice, as far as we know. 

The proof of Theorem \ref{thm_self_bounding} is deferred to Section \ref{ssec_self_bouning_SM} of the Supplementary Material. The main idea is to build lower proxy functions $\bar{g}_i$ of $\bar{g}$, as in Definition \ref{def_self_bounding}, by acting on the indices of the vector $\xi$ through random permutations $\tau_{i,J}$ and arguing that, by exchangeability, it does not change its distribution.


\section{Concentration conditional on data \label{sec_cond_X}}

In practice, the quantity $\overline{g}(x)$ generally cannot be calculated exactly, but is usually approximated by the Monte Carlo estimate
\[ \hat{g}_B(x) = \frac{1}{B} \sum_{b =1}^B \sup_{t \in \cT} \sum_{i = 1}^n \xi_i^{(b)} t(x_{i})\;, \]
where $(\xi^{(b)})_{b = 1,\ldots,B}$ are i.i.d. draws from the weight vector $\xi$.
In other applications, such as the non-parametric two sample tests discussed in Section \ref{sec_two_sample_test}, the quantity of interest is not the conditional mean $\overline{g}(X)$ but rather a quantile of the conditional distribution of 
\[ g(X,\xi) = \sup_{t \in \cT} \sum_{i = 1}^n \xi_i t(X_{i}) \]
knowing $X$, known as a ``bootstrap quantile''. One way to handle both quantities is to derive concentration inequalities for $g(x,\xi)$ around its mean $\overline{g}(x)$. 

In this section, we focus on weights of the form $\xi = w_\sigma$ where $w \in \mathbb{R}^n$ is a deterministic vector such that $\sum_{i = 1}^n w_i = 0$ and $\sigma \in \mathfrak{S}_n$ is a uniform random permutation. Consequences for general exchangeable weights are derived in section \ref{sec_gen_bound}, following the argument sketched at the end of Section \ref{sec_bootep}.

\subsection{Proof steps for Theorem \ref{thm_intro_concen_exchange}}\label{sec_general_concen}
For the sake of clarity, we briefly describe in this section some techniques that will be instrumental in our proofs. We also explain relations with some classical lemmas pertaining to the literature of concentration inequalities.
\subsubsection{Exponential moments and a decoupling inequality}
\sloppypar{}
The concentration of a random variable is related to the behavior of its moments 
(\cite{BoucheronLugosiMassart:2013,vershynin2018high}).
In the following, we will be interested in random variables with some finite exponential moments. Through the use of the classical Cramér-Chernoff method, it will thus be sufficient for us to control their moment-generating function.

Let us consider a random variable $Z$, defined on the measurable space $(\cZ,\cT)$, and a measurable function $g:(\cZ,\cT) \rightarrow \R$, such that the moment-generating function of $g(Z)$, denoted by $\varphi_{g(Z)}$, is well-defined at the right-neighborhood of the origin: there exists $\theta_0>0$ such that $\varphi_{g(Z)}(\theta_0)=\EE[e^{\theta_0 g(Z)}]<+\infty$. Denote also $g_0(Z)=g(Z)-\EE[g(Z)]$, the expectation of $g(Z)$ being indeed well-defined in this case. Then, for any $\theta \in (0,\theta_0)$, the quantity $\EE[g_0(Z) e^{\theta g_0(Z)}]$ is also well-defined, and it corresponds to the derivative of $\varphi_{g_0(Z)}$ at the point $\theta$. 

In this setting, a quite standard approach to derive concentration inequalities for $g_0(Z)$ (\cite{BoucheronLugosiMassart:2013}) consists in establishing and solving a differential inequality for its moment-generating function.
In this article, we will rather use the following result.
\begin{lemma}
\label{theorem duality entropy}
Let $X$ and $Y$ be two
real random variables such that 
\begin{equation}
\label{hyp_decoupling}
\mathbb{E}\left[Xe^{X}\right]\leq\mathbb{E}\left[Ye^{X}\right]<+\infty,
\end{equation}
then 
\begin{equation*}
\mathbb{E}\left[e^{X}\right]\leq\mathbb{E}\left[e^{Y}\right]\;.\label{bound_laplace}
\end{equation*}
In addition, if Inequality \eqref{hyp_decoupling} is strict, then the conclusion also holds with strict inequality.
\end{lemma}
The above theorem readily implies that if there exists a random variable $Y_\theta$ such that $\EE[g_0(Z) e^{\theta g_0(Z)}]\leq \EE[Y_\theta e^{\theta g_0(Z)}]$, then it holds: $\varphi_{g_0(Z)}(\theta)\leq \EE[e^{\theta Y_\theta}]$. Consequently, one may consider Lemma \ref{theorem duality entropy} a \textit{``decoupling'' inequality}. 


The first part of Lemma \ref{theorem duality entropy} above was proved in \cite{saumard2019weighted}, as a consequence of a duality formula for the entropy. Considering in addition the version with strict inequalities, we further prove in Section \ref{ssec_sm_decoupling} of the Supplementary Material that Lemma \ref{theorem duality entropy} is actually equivalent to that duality formula, that we recall now: for a non-negative random variable $X$ such that $\EE[X\log(X)]<+\infty$, it holds
$${\rm Ent}(X):=\EE[X\log(X)]-\EE[X]\log(\EE[X])=\sup_{U\, \text{s.t.}\, \EE[e^{U}]=1}\EE[UX]\;.$$
In other words, Lemma \ref{theorem duality entropy} is \textit{nothing but another formulation of the duality formula for the entropy}. This remark seems to be new, up to our best knowledge. 

\subsubsection{A covariance inequality through exchangeable pairs}\label{ssec_exch_pair_cov_ineq}
In order to use Lemma \ref{theorem duality entropy} above for a centered function $g_0=g-\EE[g(Z)]$, notice first that
$$\EE[g_0(Z) e^{tg_0(Z)}]=\cov\left(g(Z),e^{tg_0(Z)}\right)\; ,$$
thus casting the problem of bounding the left-hand side quantity into the quite developed framework of \textit{covariance inequalities}.
Covariance inequalities indeed appear to be useful in a variety of domains. Let us simply provide here some pointers to the literature for applications in particle systems (\cite{FKG71,BakryMichel92,Menz-Otto:2013}), convex geometry and log-concavity (\cite{CarlenCordero-ErausquinLieb,Royen:14,SauWel2014}), functional inequalities (\cite{Houdre95,HoudrePerezAbreu95,HoudrePerezAbreuSurgailis98,ArnaudonBonnefontJoulin18,ArrasHoudre23}) and also, more importantly for us, concentration inequalities (\cite{Houdre98,BobGotzeHoudre01,Led01,HoudrePrivault02,HoudreMarchalRB08,saumard2019weighted}). 

Our strategy in Section \ref{ssec_implement_exch_pair} below will be to follow the path initiated by Sourav Chatterjee in his PhD thesis \cite{Chatterjee2005thesis} (see also \cite{Chatterjee07,ChatDey:10}), by taking advantage of an exchangeable pair argument inspired by the pioneering work of Charles Stein related to rates of convergence in the central limit theorem.

The exchangeable pair approach is indeed very well suited for functionals of sums of random variables, especially when there are dependencies among these random variables, which is the case in general in the present article, where we consider exchangeable bootstrap statistics.

\begin{lemma}[\cite{Chatterjee2005thesis,Paulin14}]\label{lem_cov_ineq}
 Consider $t>0$, $g$ a measurable function from $(\cZ,\cT)$ to $\R$, $\left(Z, Z^{\prime}\right)$ an exchangeable pair on $\cZ^2$, and let $F\left(Z, Z^{\prime}\right)$ be a measurable, antisymmetric function such that
$$
\mathbb{E}\left[F\left(Z, Z^{\prime}\right) \mid Z\right]=g(Z)-\mathbb{E}[g(Z)]
$$
and 
$$
\EE\left[\vert F(Z,Z^\prime)\vert e^{tg(Z)}\right]<+\infty\;.
$$
Then, it holds
\begin{equation}\label{cov_ineq_psi}
    \cov\left(g(Z), e^{t g(Z)}\right)
\leq  \mathbb{E}\left[V_+(Z)e^{t g(Z)}\right],
\end{equation}
where 
\begin{equation}\label{def_V_+}
    V_+(Z)=\EE[\left(F\left(Z, Z^{\prime}\right)\right)_{+}(g(Z)-g(Z^\prime))_{+}\mid Z]\;. 
\end{equation}
\end{lemma}
For completeness, the proof of Lemma \ref{lem_cov_ineq} can be found in Section \ref{ssec_sm_cov_ineq} of the Supplementary Material. 
It appears that positive parts in the term $V_+(Z)$ of the covariance inequality, compared to absolute values, can bring substantial improvements when used to address concentration, due for instance to potential asymmetry between right and left tails. They will play an essential role when considering the supremum of the (bootstrap) empirical process. 
 
In light of Lemma \ref{lem_cov_ineq} applied with $\psi=e^{t\cdot}$, Inequality \eqref{hyp_decoupling} of Lemma \ref{theorem duality entropy} is satisfied with $X=tg(Z)$ and $Y=t^2V_+(Z)$. Lemma \ref{theorem duality entropy} then gives
\begin{equation}\label{eq:bound_Laplace_exch}
  \varphi_{g_0(Z)}(t)\leq \EE[\exp(t^2V_+(Z))]\;.
\end{equation}

Hence, the control of the moment generating function of $g_0(Z)$ reduces to the one of $V_+(Z)$, with the parameter taken to the square. In general, there is a major gain in switching to $V_+(Z)$. In particular, whatever $Z$, $V_+(Z)$ is a non-negative random variable. 




\subsubsection{Implementing the exchangeable pair argument}\label{ssec_implement_exch_pair}

Let $w \in \mathbb{R}^n$ be a centered vector (\textit{i.e.} $\sum_{i = 1}^n w_i = 0$) and let $\sigma \in \mathfrak{S}_n$ be a uniform random permutation. The vector $\xi = w_\sigma$ is thus exchangeable and centered.  Let
\[ g_x(\sigma) = g(x,w_\sigma) = \sup_{t \in \cT} \left\{ \sum_{i = 1}^n w_{\sigma(i)} t(x_{i}) \right\}. \]
The only source of randomness here is the random permutation $\sigma$. Hence, we are interested in the concentration properties of certain functions of a random permutation. Using the method of exchangeable pairs, we shall establish general concentration results for such functions which may be of intrinsic interest.

Let $g: \mathfrak{S}_n \to \mathbb{R}$ be a function and let $\sigma \in \mathfrak{S}_n$ be a uniform random permutation. In order to apply Lemma \ref{lem_cov_ineq} to $g(\sigma)$, we need to construct an exchangeable pair $(\sigma, \sigma')$ and an antisymmetric function $F$. 

Consider by analogy the case of an i.i.d. vector $X = (X_1,\ldots,X_n)$, instead of a random permutation $\sigma$. In that case, as first discovered in \cite{Chatterjee2005thesis}, some classical concentration results - according to the assumptions on $g$ -
can be recovered by taking the exchangeable pair $(X,X')$ where $X_i' = X_i$ except for some random index $I$ where $X_I' = X_I^*$ is drawn from an i.i.d. copy of $X$ (denoted $X^*$). 
Indeed, simplifications occur in the computations since in the exchangeable pair $(X,X')$, $X'$ represents a ``small modification'' of $X$, which acts only on a small number of coordinates of $X$ ($1$ coordinate in this case). 

For permutations, the smallest number of coordinates which can be modified from $\sigma$ to $\sigma'$ is $2$, which leads to the pairs $(\sigma, \sigma \tau_{I,J})$ or $(\sigma, \tau_{I,J} \sigma)$ for some random transposition $\tau_{I,J}$ independent from $\sigma$. 
We choose the distribution of $\tau_{I,J}$ to be invariant under permutations of the indices in $\left\{ 1,\ldots, n\right\}$, that is of the form 
\begin{equation} \label{eq_loi_transpo_al}
   \alpha_0 \delta_{\mathrm{Id}} + \frac{2(1-\alpha_0)}{n(n-1)} \sum_{i = 1}^{n-1} \sum_{j = i+1}^n \delta_{\tau_{i,j}}, 
\end{equation}
\textit{i.e.} equal to the identity with probability $\alpha_0$ and uniformly distributed over the non-trivial transpositions.
The pairs $(\sigma, \sigma \tau_{I,J})$ and $(\sigma, \tau_{I,J} \sigma)$ are then exchangeable and equal in distribution since
\[ (\sigma, \sigma \tau_{I,J}) = (\sigma \tau_{I,J} \tau_{I,J}, \sigma \tau_{I,J}) \text{ and } (\sigma, \tau_{I,J} \sigma) = (\sigma, \sigma \tau_{\sigma^{-1}(I), \sigma^{-1}(J)}). \]
By convention, we write $(\sigma, \sigma') = (\sigma, \sigma \tau_{I,J})$. 

Now that the pair is defined, Chatterjee's work \cite{Chatterjee2005thesis} yields a standard method of constructing $F$ and of bounding $(F(\sigma, \sigma'))_+$, which is the great strength of his approach. Let us briefly describe it here in a somewhat informal manner (see Proposition \ref{prop_exist_F} in the Supplementary Material for a rigorous statement). 

For any permutation $\pi$, let $P_\pi$ be the conditional distribution of $\sigma'$ knowing that $\sigma = \pi$, \textit{i.e.} the Markov kernel associated with the pair $(\sigma, \sigma')$. For any integer $k$, let $P_\pi^k$ be the $k$-th iterate of the Markov kernel, \textit{i.e.} the distribution of the $k$-th element of a Markov chain started at $\pi$ and with transition kernel $\pi \mapsto P_\pi(\cdot)$ on $\mathfrak{S}_n$. Here, this Markov chain is the ``transposition random walk'' on the permutation group. Given any two permutations $\pi,\pi'$, the function defined by
\begin{equation} \label{eq_def_F_gen}
   F(\sigma,\sigma') = \sum_{k = 0}^{+ \infty} P_\pi^k g - P_{\pi'}^k g 
\end{equation}
is well defined due to the mixing properties of the transposition random walk (see Proposition \ref{prop_exist_F}). Moreover, it is obviously antisymmetric and satisfies
\[ \mathbb{E} \left[ F(\pi,\pi') | \sigma \right] = \sum_{k = 0}^{+ \infty} P_\sigma^k g - \mathbb{E}\left[P_{\sigma'}^k f | \sigma \right] = \sum_{k = 0}^{+ \infty} P_\sigma^k g - P_\sigma^{k+1} g = g(\sigma) - Pg \;, \]
where $P$ is the uniform distribution on $\mathfrak{S}_n$. This proves that $F$ satisfies the hypotheses of lemma \ref{lem_cov_ineq}. In order to derive useful results from Lemma \ref{lem_cov_ineq}, it remains to bound $(F(\sigma, \sigma'))_+$. A possibility is to use a coupling argument (Chatterjee \cite{Chatterjee2005thesis}), but no Markov coupling can merge faster than $\mathcal{O}(n^2)$ \cite{Bormashenko2011} here. In a work on concentration of Haar measures (\cite{Chatterjee2007Haar}), Chatterjee rather used directly the mixing property of the chain in total variation distance. In this article, we employ a hybrid approach in order to deliver an optimized bound.


Formally, let $(I_j,J_j)_{j \geq 1}$ be an i.i.d. sequence independent from $(I,J)$ and with the same distribution. Let $\pi_0 = \mathrm{Id}$ and for any $k \geq 1$,
\[ \pi_k =  \pi_{k-1} \circ \tau_{I_k,J_k}  = \tau_{I_1,J_1} \tau_{I_{2},J_{2}} ... \tau_{I_{k-1},J_{k-1}} \tau_{I_k,J_k}. \]
This defines a random walk on the symmetric group, called the (right) transposition random walk, started at the identity. We then define
\[ \sigma_k = \sigma \pi_k \text{ and } \sigma_k' = \sigma' \pi_k \]
for all $k \in \mathbb{N}$. Clearly, the two chains never meet (unless $\sigma = \sigma'$), however they both converge in total variation to the uniform distribution on $\mathfrak{S}_n$. In particular, the distributions of the two chains become closer and closer over time. To exploit this property, we introduce the concept of ``strong convergence time'' defined below.

\begin{definition} \label{def_conv_time}
    Let $i,j$ be two distinct indices in $\{1,\ldots,n\}$.
     A non-negative random variable $T_{i,j}$, measurable with respect to $(\pi_k)_{k \geq 1}$ is called a \emph{strong convergence time} for the random walk $(\pi_k)_{k \geq 1}$ and the pair $(i,j)$ if the conditional distributions of $\pi_k$ and $\tau_{i,j} \pi_k$ knowing that $k \geq T_{i,j}$ are equal for any integer $k$. By symmetry, the distribution of $T_{i,j}$ can and will be assumed to not depend on $i,j$. 
\end{definition}

A specific class of ``strong convergence times'' are the \emph{strong stationary times} (also known as \emph{strong uniform times}) which are such that the distribution of $\pi_k$ (and hence of $\tau_{i,j} \pi_k$) is uniform on $\mathfrak{S}_n$ for all $k \geq T$ (knowing that this event has occurred). As is clear from the definition, these ``strong stationary times'' are closely related to the mixing properties of the Markov chain $(\pi_k)_{1 \leq k \leq n}$. Several authors have worked to construct strong stationary times for the transposition random walks \cite{Matthews1988, White2019}; these times have expectation of order $\mathcal{O}(n \log n)$, which is known to be optimal \cite{Diaconis1981, Matthews1988}. In contrast, the property required by definition \ref{def_conv_time} is weaker than full mixing and so it stands to reason that the strong convergence times should be smaller on average than the strong uniform times. In fact, we shall see that strong convergence times can be constructed with expectation of order $\mathcal{O}(n)$.     
First, let us show how strong convergence times can be used to bound $F$. 
\begin{lemma} \label{prop_F_sigma}
    For any collection of strong convergence times $(T_{i,j})_{1 \leq i < j \leq n}$,
    \[ F(\sigma,\sigma \tau_{I,J}) = \mathbb{E} \left[ \sum_{k = 0}^{+\infty} \left( g(\sigma \circ \pi_k) - g(\sigma' \circ \pi_k) \right) \mathbb{I}\{ k < T_{I,J} \} \bigl| \sigma,I,J \right],  \]
    defining $T_{I,J} = 0$ if $I = J$.
\end{lemma}
The proof of Lemma \ref{prop_F_sigma} can be found in Section \ref{ssec_proof_lemma_F_sigma} of the Supplementary Material. Using the covariance inequality (Lemma \ref{lem_cov_ineq}), the decoupling inequality (Lemma \ref{theorem duality entropy}) and Lemma \ref{prop_F_sigma} above yields the following theorem, the proof of which can be found in Section \ref{subsubsection_proof_theorem_concen} of the Supplementary Material.
\begin{theorem} \label{thm_conc_fun_permut}
    Let $g: \mathfrak{S}_n \to \mathbb{R}$ be a function and let $\sigma$ be a uniform random permutation. Let $f = g - \mathbb{E}[g(\sigma)]$. Let $(T_{i,j})_{1\leq i<j\leq n}$ be strong convergence times with common distribution $Q$ and define
    \[ r = r(Q) = r(T_{1,2}) = \sum_{k = 0}^{+\infty} \sqrt{\mathbb{P}(k < T_{1,2})}\;. \]
    Define also the random variable
    \[ V_+(g,\sigma) = \frac{1}{n} \sum_{i = 1}^n \sum_{j = 1}^n (g(\sigma) - g(\sigma \tau_{i,j}))_+^2\;. \]
    For any $\theta \geq 0$,
    \[ \mathbb{E}\left[ \exp \left( \theta f(\sigma) \right) \right] \leq \mathbb{E}\left[ \exp \left( \theta^2 (1-\alpha_0) \frac{r(T_{1,2})}{n-1} V_+(g,\sigma) \right) \right]\;. \]
    Moreover, for $\alpha_0 = \frac{1}{2}$ and any $n \geq 34,$ $T_{1,2}$ can be chosen such that
    \begin{equation} \label{eq_thm_conc_fun_explicit}
       \mathbb{E}\left[ \exp \left( \theta f(\sigma) \right) \right] \leq \mathbb{E}\left[ \exp \left( 9.5 \theta^2 V_+(g,\sigma) \right) \right]\;. 
    \end{equation}
\end{theorem}

This concentration result appears to be new. Works by Bobkov \cite{Bobkov2004} and Tolstikhin \cite{Tolstikhin2017} (see also \cite{tolstikhin2014localized}) established a sub-Gaussian bound with variance parameter proportional to $\norm{V_+(g,\sigma)}_\infty,$ but only for $(m,n)-$symmetric functions, \textit{i.e.} functions $g$ which are invariant under permutation of the first $m$ and last $n-m$ coordinates, for some $m \in \{1,\ldots,n\}$. In the case of the function $g_x,$ this amounts to requiring the vector $w$ to only take two values $a,b$ on the first $m$ and last $n-m$ coordinates, respectively. Another difference is that our bound involves exponential moments of $V_+(g_x,\sigma)$ instead of an almost sure upper bound. The bound of Theorem \ref{thm_conc_fun_permut} depends on the choice of a strong convergence time $T_{i,j}$. If one of the known strong stationary times is used, this yields $r(T)$ of order $n \log n$ and there is an extra logarithmic term compared to the results of Bobkov and Tolstikhin. A bound of this kind was obtained by Chatterjee \cite{Chatterjee2007Haar},  but with $V_+(g,\sigma)$ replaced by
\[ V(g,\sigma) = \frac{1}{n} \sum_{i = 1}^n \sum_{j = 1}^n (g(\sigma) - g(\sigma \tau_{i,j}))^2. \]
Aside from removing the extra $\log n$ term, Theorem \ref{thm_conc_fun_permut} introduces a positive part in the variance proxy $V_+(g,\sigma)$, which makes suprema such as $g_x$ much easier to handle.

To remove the extra logarithm and prove Equation \eqref{eq_thm_conc_fun_explicit}, we now show that the construction of a certain strong stationary time, due to White \cite{White2019}, can be modified to yield a strong convergence time $T_{1,2}$ with $\mathbb{E}[T_{1,2}]$ and $r(T_{1,2})$ of order $n$.

\begin{definition} \label{def_strong_cv_time}
    For any $t \geq 0$, let $\mathcal{P}(t)$ be the partition of $\{1,\ldots,n\}$ defined in \cite{White2019} and let $P_m(t)$ be a block of $\mathcal{P}(t)$ with maximal cardinality (unique in the following equation). For any pair of distinct indices $i,j,$ let
    \[ T_{i,j} = \min \left\{ t \geq 0 : |P_m(t)| > \frac{n}{2} \text{ and } \{i,j\} \subset P_m(t) \right\}. \]
\end{definition}

The construction of $\mathcal{P}(t)$ (``marking scheme D'' in \cite{White2019}) is such that the largest block does not shrink, \textit{i.e.} there always exists $B \in \mathcal{P}(t+1)$ such that $P_m(t) \subset B$. It follows by induction that the maximal blocks $P_m(t)$ form a nested sequence for all $t$ larger than the hitting time of $\left[ \left\lceil\tfrac{n}{2} \right\rceil,+\infty \right)$ by $(|P_m(t)|)_{t  \in \mathbb{N}}$. This guarantees that $\{i,j\} \subset P_m(t)$ remains true for all $t \geq T_{i,j}$: thus,
\[ \{t \geq T_{i,j} \} = \left\{ |P_m(t)| > \frac{n}{2} \text{ and } \{i,j\} \subset P_m(t) \right\}. \]
Moreover, White proved (\cite[Proposition 11]{White2019}) that the law of $\pi_t$ knowing $\mathcal{P}(t)$ is invariant under transpositions $\tau_{k,l}$ whenever $k,l$ belong to the same block of $\mathcal{P}(t)$.
It follows that $T_{i,j}$ is a strong convergence time in the sense of Definition \ref{def_conv_time}.

Based on the work of White \cite{White2019} and our own calculations (see Section \ref{proof_strong_stat_time} of the Supplementary Material), we can prove the following bound for $r(T_{1,2})$.

\begin{proposition} \label{prop_strong_stat_time}
    For the exchangeable pair given by equation \eqref{eq_loi_transpo_al} with $\alpha_0 = \frac{1}{2}$ and the strong convergence time $T_{1,2}$ given by Definition \ref{def_conv_time}, we have that
    \[ r(T_{1,2}) \leq \max\left\{615,18.42 n\right\}. \]
\end{proposition}

We now return to the original problem of finding concentration bounds for the resampled supremum $g(x,\sigma)$ around its expectation $\overline{g}(x)$. By Theorem \ref{thm_conc_fun_permut}, this amounts to finding a majorant for $V_+(g_x,\sigma)$. This is the point of the following Proposition, which proof is detailed in Section \ref{ssec_proof_prop_bound_Vplus_perm} of the Supplementary Material.

\begin{proposition} \label{prop_bound_Vplus_perm}
    For any $x \in E^n$ and $t \in \cT$, let
    \[ \overline{t}_x = \frac{1}{n} \sum_{i = 1}^n t(x_{i}) \]
    and define the weak variance
    \[ v_+(x) = \sup_{t \in \cT} \left\{ \sum_{i = 1}^n (t(x_{i}) - \overline{t}_x)^2 \right\}\;. \]
    Let also $a = \min_{1 \leq i \leq n} w_i$ and $b = \max_{1 \leq i \leq n} w_i$. Then, it holds
    \[ V_+(g_x,\sigma) \leq 2 (b-a)^2 v_+(x)\;. \]
    Moreover, if the functions in $\cT$ are valued in $[-1,1]$, then
    \[ V_+(g_x,\sigma) \leq 8 \norm{w}^2\;, \]
    where $\norm{\cdot}$ is the euclidean norm on $\mathbb{R}^n$
\end{proposition}

 Theorem \ref{thm_conc_fun_permut} and Propositions \ref{prop_strong_stat_time} and \ref{prop_bound_Vplus_perm} directly yield the following bound on the moment generating function of $g_x(\sigma)$.

\begin{theorem}\label{Th_concen_exchange}
    Let $a < 0 < b$ be two real numbers and let $w \in [a;b]^n$ be such that $\sum_{i = 1}^n w_i = 0$. For any  $x \in E^n$, let
    \[ g_x(\sigma) = g(x,\sigma) = \sup_{t \in \cT} \left\{ \sum_{i = 1}^n w_{\sigma(i)} t(x_{i}) \right\}\;. \]
    Let $\sigma \in \mathfrak{S}_n$ be a uniform random permutation.
    Let also $f_x = g_x - \mathbb{E}[g_x(\sigma)]$. Assume that the functions in $\cT$ are valued in $[-1,1]$. Then, for any $\theta > 0$ and any $n \geq 34$,
    \begin{equation} \label{exp_moment_cond_X_v_+}
    \mathbb{E}[\exp(\theta f_x(\sigma))] \leq \exp \left(\theta^2 \min \left\{19 (b-a)^2 v_+(x), 4.2 \norm{w}^2\right\}  \right)\;, 
    \end{equation}
    where $\norm{w}$ is the Euclidean norm of $w$ and  $v_+(x)$ is the weak empirical variance,
    \[ v_+(x) = \sup_{t \in \cT} \left\{ \sum_{i = 1}^n \left(t(x_{i}) - \frac{1}{n} \sum_{j = 1}^n t(x_{j}) \right)^2 \right\}\;. \]
\end{theorem}

Theorem \ref{Th_concen_exchange} shows that $g_x(\sigma)$ is sub-Gaussian with variance factor $\mathcal{O}\left(v_+(x) \right)$. Compared to the bootstrap version of Bousquet's inequality, which applies to the Efron weights (Theorem \ref{thm_bootstrap_efron}), the variance proxy is smaller ($v_+(x)$ instead of $2\overline{g}(x) + v_+(x)$) and the tail behaviour is sub-Gaussian instead of sub-Poisson. This corresponds to the tail behaviour of the weights, which are bounded in Theorem \ref{Th_concen_exchange} and essentially Poisson distributed in Theorem \ref{thm_bootstrap_efron}.  

In the special case where $w \in \left\{- \frac{1}{l}, \frac{1}{k} \right\}^n$ with $k + l = n,$ the conclusion of Theorem \ref{Th_concen_exchange} follows from Tolstikhin's version of Bobkov's inequality \cite{tolstikhin2014localized,Tolstikhin2017} 
Theorem \ref{Th_concen_exchange} is more general since it allows for arbitrary centered $w$. This greater generality does come at the price of a worse constant ($19$ in Inequality \eqref{exp_moment_cond_X_v_+} instead of $2$ in Tolstokhin's inequality), though one should note that the constant $19$ in Inequality \eqref{exp_moment_cond_X_v_+} can likely be improved, either through a more refined analysis of the construction in \cite{White2019} or using a different strong convergence time. In section \ref{proof_Tolstikhin_exch_pair} of the Supplementary Material, we also give a probabilistic proof of Tolstikhin's Theorem that recovers the right constant up to terms of order $1/n$, using the method of exchangeable pairs. This essentially shows that the difference in constants is not due to some inherent defect in our proof strategy.

\section{A general deviation bound}\label{sec_gen_bound}
Let us now combine Theorems \ref{thm_intro_cond_exp} and \ref{thm_intro_concen_exchange} in order to derive a general high probability upper bound comparing the quantity $g(X,\xi)$ -- see Definition \ref{def_g_bootstrap} above -- with the expected supremum of the associated empirical process. Let $\sigma \in \mathfrak{S}_n$ be a random permutation, independent from $\xi$ and define
\[ \overline{g}_\xi(x) = \mathbb{E} \left[ \sup_{t \in \cT} \sum_{i = 1}^n \xi_{\sigma(i)} t(x_i) | \xi \right]\;. \]
First notice that Theorem \ref{thm_intro_cond_exp} applies to $\overline{g}_\xi$ conditionally on $\xi$, yielding concentration of $\overline{g}_\xi(X)$ around
\[ \mathbb{E}[g(X,\xi_\sigma) | \xi] = \mathbb{E} \left[ \sup_{t \in \cT} \sum_{i = 1}^n \xi_{\sigma(i)} t(X_{i}) | \xi  \right]. \]
The latter quantity can be bounded from above and from below by the results recalled in Section \ref{ssec_mean} of the Supplementary Material. For example, according to Proposition \ref{prop_upper_mean_boot},
\[ \mathbb{E} \left[ \sup_{t \in \cT} \sum_{i = 1}^n \xi_{\sigma(i)} t(X_{i}) | \xi  \right] \leq \norm{\xi}_\infty \mathbb{E} \left[ \sup_{t \in \cT} \sum_{i = 1}^n \bigl(t(X_{i}) - \mathbb{E}[t(X_i)] \bigr) \right]. \]
Now, it remains to bound the quantity
\[ g(X,\xi) - \overline{g}_\xi(X) = \sup_{t \in \cT} \sum_{i = 1}^n \xi_{\sigma(i)} t(x_i) - \mathbb{E} \left[ \sup_{t \in \cT} \sum_{i = 1}^n \xi_{\sigma(i)} t(x_i) | \xi \right]. \]
Conditioning on $\xi = w$, this is equivalent to the concentration of
\[ \sup_{t \in \cT} \left\{ \sum_{i = 1}^n w_{\sigma(i)} t(x_{i}) \right\}\;, \]
which results from Theorem \ref{thm_intro_concen_exchange}. This yields the following general probability upper bound on $g(X,\xi)$.


\begin{theorem} \label{thm_hp_ubd_boot}
Let $\cT$ be a set of measurable functions, from $E$ to $\left[-\frac{1}{2}, \frac{1}{2} \right]$, which is separable in the product topology. Let $X=(X_{i})_{1 \leq i \leq n}$ be a collection of $n$ independent random variables valued in $E$. Let 
\[\mu_t = \frac{1}{n}\sum_{i = 1}^n \mathbb{E}[t(X_i)], \ \sigma^2 = \sup_{t \in \cT} \left\{ \sum_{i = 1}^n \mathbb{E}\left[(t(X_{i})-\mu_t)^2 \right] \right\}, \ M_n = \mathbb{E} \left[ \sup_{t \in \cT} \left| \sum_{i = 1}^n t(X_i) - \mathbb{E}[t(X_i)] \right| \right].\]
Define the random variables
\[ \norm{\xi}_\infty = \max_{1 \leq i \leq n} \{ |\xi_i| \}\;, \quad \kappa_\xi = \frac{1}{n} \sum_{i = 1}^n |\xi_i|\;, \quad M_n(\xi) = \mathbb{E} \left[ g(X,\xi_\sigma) | \xi \right]\;, \]
where $\sigma \in \mathfrak{S}_n$ is a uniform random permutation independent from $X,\xi$. 
For any $x > 0$, with probability at least $1 - 2e^{-x}$,
\begin{equation}\label{general_dev_bound}
    g(X,\xi) \leq M_n(\xi) + \sqrt{6\kappa_\xi x M_n(\xi)} + \norm{\xi}_\infty \max \left( 19 \sqrt{x(4M_n +\sigma^2)}, 41 x \right) + \frac{5}{2} \kappa_\xi x\;. 
\end{equation}
\end{theorem}
\noindent The proof of Theorem \ref{thm_hp_ubd_boot} can be found in Section \ref{ssec_proofs_general_bound} of the Supplementary Material.

Inequality \eqref{general_dev_bound} is not a ``true'' concentration inequality since the upper bound depends on $\xi$ and hence is random. In particular, the leading term $M_n(\xi)$ is not the expectation $\mathbb{E}[g(X,\xi)] = \mathbb{E}[g(X,\xi_\sigma)]$ but rather the conditional expectation,
$\mathbb{E} \left[ g(X,\xi_\sigma) | \xi \right]$. It is easy to show that concentration of $g(X,\xi)$ around its expectation cannot hold at this level of generality: for example, replacing $\xi$ by $\Lambda \xi$ for some independent nonnegative random variable $\Lambda$ yields $g(X,\Lambda \xi) = \Lambda g(X,\xi)$, which fails to concentrate.

In the exchangeable bootstrap perspective, the point of using $g(X,\xi)$ is to approximate the distribution of the associated empirical process supremum. In applications, one is typically interested in comparing $g(X,\xi)$ to $M_n$ rather than to $\mathbb{E}[g(X,\xi)]$. If $g(X,\xi)$ is known to concentrate (as is the case for Efron's bootstrap), this still leaves the final step of relating $\mathbb{E}[g(X,\xi)]$ to $M_n,$ using for example Proposition \ref{prop_upper_mean_boot}. Any such bound that is valid for all exchangeable bootstrap weights will apply also to $M_n(\xi)$ since conditionally on $\xi,$ $\xi_\sigma$ is exchangeable with null sum and independent from $X$. For example, using Proposition \ref{prop_upper_mean_boot} immediately yields $M_n(\xi) \leq 2\norm{\xi}_\infty M_n$. Thus, Theorem \ref{thm_hp_ubd_boot} should be as useful as a ``real'' concentration inequality for the purpose of bounding $g(X,\xi)$ by a function of $M_n$ with high probability. 

In fact, if Proposition \ref{prop_upper_mean_boot} is used to compare expectations then the dependence on $\xi$ significantly strengthens the result compared to a deterministic concentration bound, yielding a non-trivial inequality even for unbounded weights. For example, if $M_n \to + \infty$, then Theorem \ref{thm_hp_ubd_boot} and Proposition \ref{prop_upper_mean_boot} together imply that
\[ \mathbb{E}[g(X,\xi)] \leq 2(1+o(1)) \mathbb{E} [\norm{\xi}_\infty] M_n  \]
whereas Proposition \ref{prop_upper_mean_boot} by itself only yields $\mathbb{E}[g(X,\xi)] \leq 2 \norm{\norm{\xi}_\infty}_{L^\infty} M_n$.

The bound of Theorem \ref{thm_hp_ubd_boot} is \emph{not} conditional on $X$, unlike Theorem \ref{Th_concen_exchange}. However, we shall see that it still implies bounds on the upper tail of the bootstrap quantiles. In the rest of this section, let us denote by $q_\alpha(Y)$ the least $1-\alpha$ quantile of a random variable $Y$, \textit{i.e.} 
\[ q_\alpha(Y) = \inf \left\{q \in \mathbb{R}: 1-\alpha \leq \mathbb{P}\left( Y \leq q \right) \right\} = \inf \left\{q \in \mathbb{R}: \mathbb{P}\left( Y > q \right) \leq \alpha \right\}.  \]
Similarly, let $q_\alpha(Y | Z)$ denote the conditional quantile of order $1-\alpha$. In particular, $q_\alpha(g(X,\xi) | X)$ denotes the bootstrap quantile of order $1-\alpha$ (defined up to a negligible event). The key tool is the following simple lemma.

\begin{lemma} \label{lem_quantile_condi}
    Let $Y,X$ be random variables, where $Y$ is real valued. Let $\alpha,\gamma \in (0;1]$. Then
    \[ q_\gamma \left( q_\alpha(Y | X) \right) \leq q_{\gamma \alpha}(Y)\;. \]
\end{lemma}
\noindent The proof of Lemma \ref{lem_quantile_condi} can be found in Section \ref{ssec_proofs_general_bound} of the Supplementary Material.
As a result of this lemma, we can state a universal bound on the bootstrap quantiles. As it turns out, this bound also applies, with no modification, to the empirical bootstrap quantiles based on a random sample $(\xi_{(b)})_{b = 1,\ldots,B}$. For simplicity, we assume that the weights $\xi$ are of the form $1 - W$ or $W-1$, where $\frac{1}{n}W$ is an exchangeable probability vector; that is, the vector $\xi$ is obtained from the weights of the exchangeably weighted bootstrap in the manner described in Section \ref{sec_bootep}.

\begin{corollary} \label{cor_ubd_quant_boot}
    Let $B$ be an integer, $\alpha, \gamma \in (0;1)$ be probabilities and $\xi,(\xi^{(b)})_{b = 1,\ldots,B}$ be i.i.d. vectors, each of which is of the form $\xi = 1 - W$ or $\xi = W-1$, where $W$ is a non-negative, exchangeable vector which sums to $n$. Let $\hat{q}_\alpha^B$ be the empirical bootstrap $1-\alpha$ quantile, \textit{i.e.} the empirical $1-\alpha$ quantile of the sample 
    \[ ( g(\xi^{(b)},X))_{b = 1,\ldots,B}. \]
    For any non-negative $\alpha_1,\alpha_2,\alpha_3$ such that $\sum_{i = 1}^3 \alpha_i = \alpha$, 
    \begin{align*}
        q_\gamma \left( \hat{q}_\alpha^B \right) &\leq q_{\gamma \alpha_2}(M_n(\xi)) + \sqrt{6 q_{\gamma \alpha_2}(M_n(\xi)) (\log(2)-\log(\gamma \alpha_1))} + 5 (\log(2)-\log(\gamma \alpha_1)) \\ 
        &\quad + q_{\gamma \alpha_3}(\norm{\xi}_\infty) \max \left( 19 \sqrt{x(4M_n +\sigma^2)}, 41 x \right).
    \end{align*}
\end{corollary}
\noindent The proof of Corollary \ref{cor_ubd_quant_boot} can be found in Section \ref{ssec_proofs_general_bound} of the Supplementary Material.

According to the bootstrap heuristic, the bootstrap quantile $\hat{q}_\alpha^B$ should approximate the quantile 
\[q_\alpha \left( \sup_{t \in \cT} \sum_{i = 1}^n t(X_{i}) - \mathbb{E}[t(X_{i})] \right),\] 
up to a constant depending on the weights, as $n,B \to + \infty$. By Bousquet's inequality for the empirical process, we therefore expect a bound of the form 
\begin{equation} \label{eq_conc_quant_opt}
   \hat{q}_\alpha^B \leq M_n + R \sqrt{|\log \alpha| (2M_n + \sigma^2)} - R \log \alpha\;, 
\end{equation}
 with $R$ bounded in probability independently of $\alpha$ as $B$ tends to $+\infty$. Corollary \ref{cor_ubd_quant_boot} yields a similar bound, but with $M_n$ replaced by $M_n(\xi)$ and $R$ depending on $\norm{\xi}_\infty$. If the processes $(t(X_{i}))_{t \in \cT}$ are symmetric and \[ \norm{\norm{\xi_i}_\infty}_{L^\infty} \approx \mathbb{E}[|\xi|_1] \approx 1, \]
 then the above optimal bound follows by Proposition \ref{prop_upper_mean_boot}. In general, proving a bound of the form \eqref{eq_conc_quant_opt} with leading constant $1$ is a difficult question, which goes beyond concentration of measure, since $\mathbb{E}[g(X,\xi)]$ is not equal to $M_n$. Significant work has been done to prove the validity of the bootstrap heuristic for suprema on finite index sets $\cT$ \cite{CCK23}, but the bounds depend on the cardinality of $\cT$. Thus, for infinite $\cT$, the convergence rate established by such methods will vary depending on the entropy of $\cT$, as for Donsker and Glivenko-Cantelli theorems. It may well be that the convergence rate of $\mathbb{E}[g(X,\xi)]/M_n$ to $1$ likewise depends on the entropy of $\cT$.  In contrast, the results of this article require no assumptions on $\cT$ other than those which guarantee the measurability of suprema.

\section{Application I: Confidence regions for the mean in high-dimension}\label{sec_conf_reg}

\subsection{Presentation}
In this section, we assume that the space $E$ is a Banach space with norm $\norm{\cdot}$. The sample $(X_i)_{1 \leq i \leq n}$ is made of i.i.d. random vectors valued in $E$, with common expectation $\mu$. The expectation $\mu$ can be estimated by the empirical mean
\[ \overline{X}_n = \frac{1}{n} \sum_{i = 1}^n X_i\;. \]
A norm ball of the form $B(\overline{X}_n,\hat{r}_\alpha)=\left\{x\in E\; /\; \Vert x-\overline{X}_n\Vert \leq \hat{r}_\alpha \right\}$, for some radius $\hat{r}_\alpha>0$, provides a natural confidence region adapted to the geometry of $E$. In order for  $B(\overline{X}_n,\hat{r}_\alpha)$ to have coverage $1-\alpha,$ the data-dependent radius $\hat{r}_\alpha$ must be such that
\begin{equation} \label{eq_def_conf_reg}
    \mathbb{P} \left( \norm{\overline{X}_n - \mu} > \hat{r}_\alpha \right) \leq \alpha.
\end{equation}
For finite-dimensional $E$, an asymptotically valid $\hat{r}_\alpha$ can be constructed using covariance estimation and the central limit theorem.  However, the quality of this approximation is heavily dependent on the space $(E,\norm{\cdot})$, primarily through its dimension, and the distribution $P$. Applications in biology also produce high dimensional datasets where the dimension $d$ of $E$ is much larger than the sample size $n$ \cite{arlot2010resampling}. In such cases, classical asymptotic arguments fail and the empirical covariance does not provide a reliable estimate of the true covariance matrix anymore. 

The bootstrap provides an alternative method of calibrating $\hat{r}_\alpha$, that does not suffer from a curse of dimensionality (as we shall see later). For a general norm $\norm{\cdot}$ and given resampling weights $(\xi_i)_{1 \leq i \leq n}$ satisfying Definition \ref{def_g_bootstrap}, the bootstrap estimate of $\mathbb{E}\left[ \norm{\overline{X}_n - \mu} \right]$ is 
\[\hat{R}_n =   \mathbb{E} \left[\frac{1}{n} \norm{ \sum_{i = 1}^n \xi_i X_i} \Bigr| X_1,\ldots,X_n \right]. \]
The concentration inequalities of this article apply to $\hat{R}_n$ and yield sharper bounds than the existing literature, to the best of our knowledge.

\subsection{New confidence bounds}
First, let us see how this problem relates to our setting. Let $\cB^*$ be the unit ball of the dual of $E$, that is the set of linear functionals $l: E \to \mathbb{R}$ such that $|l(x)| \leq \norm{x}$ for all $x \in E$. By the Hahn-Banach theorem, for all $x \in E$,
\[ \norm{x} = \sup_{l \in \cB^*} l(x). \]
It follows that
\[ \norm{\overline{X}_n - \mu} = \sup_{l \in \cB^*} l\left(\overline{X}_n - \mu \right) = \frac{1}{n} \sup_{l \in \cB^*}   \left\{\sum_{i = 1}^n l(X_i) - l(\mu)  \right\} \;, \]
where by linearity and continuity of $l,$ $\mathbb{E}[l(X_i)] = l(\mu)$. It follows that $\norm{\overline{X}_n - \mu}$ is the supremum of the empirical process associated with the function class $\cB^*$. Likewise, the statistic 
\[n \hat{R}_n=\EE\left[ \sup_{l\in \cB^*} \left\{ \sum_{i=1}^n \xi_il(X_i)\right\}\bigl\vert X\right] \] is of the form of the quantity $\overline{g}$ in Definition \ref{def_gbar_bootstrap}. 
Applying Theorem \ref{thm_intro_cond_exp} yields the following result. 

\begin{theorem} \label{thm_conf_reg_mean}
    Let $\kappa = \mathbb{E}[|\xi_1|]$ and $b = \norm{\xi_1}_\infty$. Assume that $\norm{X_1 - \mu} \leq M$ almost surely for some constant $M$.
    With probability at least $1 - 2e^{-x},$ for all $\theta > 0,$
    \[ \norm{\overline{X}_n - \mu} \leq (1+\theta)^2 \frac{2}{\kappa} \hat{R}_n + \sigma_\cB \sqrt{\frac{2x}{n}} + \left(\frac{7}{\theta} + 19 + 18 \theta + 6 \theta^2 \right) \frac{xM}{n}. \]
    where 
    \[ \sigma_\cB = \sup_{l \in \cB^*} \sqrt{\mathrm{Var}(l(X_1))}.  \]
    Moreover, with probability at least $1 - 2e^{-x},$ for all $\theta \in [0,1],$ 
    \[ \norm{\overline{X}_n - \mu} \geq (1-\theta)^2 \frac{\hat{R}_n}{2b} - \sigma_\cB \sqrt{\frac{2x}{n}} - \left(\frac{3\kappa}{2b} + 1 \right) \frac{xM}{\theta n} - \left(2 - 2 \frac{\kappa}{b} + \theta^2 \frac{\kappa}{2b} \right) \frac{xM}{n} \]
    If $X_1$ is symmetric around the mean, the leading constants $2/\kappa$ and $1/2b$ can be replaced by $1/\kappa$ and $1/b$.
\end{theorem}

Theorem \ref{thm_conf_reg_mean} yields confidence upper and lower bounds on $\norm{\overline{X}_n - \mu}$ based on the bootstrap statistic $\hat{R}_n$, given the constant $M$ and an upper bound on $\sigma_\cB$. A version of Theorem \ref{thm_conf_reg_mean}, that gives slightly more precise bounds for symmetric random variables, can be found in Section \ref{ssec_sm_conf_reg} of the the Supplementary Material (Theorem \ref{thm_conf_reg_mean_gen}), with a detailed proof. 

To the best of our knowledge, the state of the art theory on this bootstrap approach to mean estimation was established in an article by Arlot et al. \cite{arlot2010resampling}.
Their setting is slightly different from the present one since instead of general norms, they consider sub-additive functions bounded by one of the $p$-norms. 
Common to both settings is the case of the $p$-norms which can be used as a point of comparison.  

They consider two assumptions: that $X$ is a Gaussian vector or that $X$ is symmetric about $\mu$ with $\norm{X - \mu}$ almost surely bounded by a constant $M$. In the bounded symmetric case, their confidence bound is based on the bounded difference inequality and controls the deviations through a term of the form $M/\sqrt{n}$. We are able to improve their result in this case,
controlling the deviations through the smaller term $\sqrt{M \hat{R}_n}$, yielding a sharper confidence bound. 

Furthermore, the leading constants $\kappa,b$ in the upper and lower bounds of Theorem $\ref{thm_conf_reg_mean}$ can be made very close to $1$ by an adequate choice of the resampling weights: for example, if $n$ is even, taking $\xi$ to be a random permutation of a fixed vector $w \in \{-1,1\}^n$ having equal numbers of $1$s and $-1$s yields matching constants $\kappa = b = 1$ for symmetric $X_1$ (by Theorem \ref{thm_conf_reg_mean_gen}). Optimizing over $\theta > 0$ yields a remainder term of order
\[ \max \left\{ \sigma_\cB \sqrt{\frac{x}{n}}, \sqrt{\frac{xM}{n} \hat{R}_n}, \frac{xM}{n} \right\}\;. \]
Since $M \geq \max \left\{ \sigma_\cB, \kappa^{-1}\hat{R}_n \right\}$, for fixed $x > 0,$ this bound is indeed sharper than the $M \sqrt{\frac{x}{n}}$ given by \cite[Theorem 2.1]{arlot2010resampling}. The improvement can be quite significant: for example, if $\norm{\cdot}$ is the Euclidean norm and $X_1$ is a random vector of dimension $d_n$ with i.i.d. coordinates $X_{1,j}$ such that $\norm{X_{1,j} - \mu_j}_{L^\infty} = c$ and $\mathrm{Var}(X_{1,j}) = \sigma^2,$ then $M = c d_n,$ $\sigma_\cB = \sigma$ and $\mathbb{E}\left[ \norm{\overline{X}_n - \mu}  \right] \leq \sqrt{\EE[\Vert \overline{X}_n-\mu\Vert^2]}= \sigma \sqrt{d_n/n}$. It follows that the bound of \cite{arlot2010resampling} is never of the correct order when $d_n \to +\infty,$ whereas some standard calculations show that ours remains sharp as long as $d_n = o(n)$.  

More generally, the remainder terms are negligible if $M \ll n \hat{R}_n$ and $\frac{\sigma_\cB}{\sqrt{n}} \ll \hat{R}_n$. Since $\EE[\hat{R}_n] \leq 2b \sigma_\cB/\sqrt{n}$ if $E$ is of dimension $1$ (by using Proposition \ref{prop_upper_mean_boot}) and since, by assuming without loss of generality that $l^*\in \cB^*$ is such that $\var(l^*(X_1))=\sigma^2_\cB$, it holds
\[ \mathbb{E}\left[ \norm{\overline{X}_n - \mu} \right] \geq \mathbb{E} \left[ \left| \frac{1}{n} \sum_{i = 1}^n l^*(X_i-\mu) \right| \right] \underset{n \to +\infty}{\sim} \sqrt{\frac{2}{\pi}} \frac{\sigma_\cB}{\sqrt{n}}\;, \]
the assumption that $\hat{R}_n \gg \sigma_\cB/\sqrt{n}$ can be interpreted as the claim that we are in a high-dimensional setting.
Provided then that $M \ll \sigma_\cB \sqrt{n}$, the bounds of Theorem \ref{thm_conf_reg_mean} are first-order correct.

Note also that Theorem \ref{thm_conf_reg_mean} does not quite yield a practical confidence upper bound, since the quantity $\sigma_\cB$ depends on the unknown distribution of the data. In general, $\sigma_\cB$ can be estimated with
\[ \hat{\sigma}_\cB^2 = \sup_{l \in \cB^*} \hat{\sigma}_l^2 \text{ where } \hat{\sigma}_l^2 = \frac{1}{n} \sum_{i = 1}^n \left(l(X_i) - l(\overline{X}_n) \right)^2\; , \]
which is biased upward by Jensen's inequality and is asymptotically consistent if $\cB^*$ is not too large (\textit{i.e.} if it forms a Donsker class). If $l^* \in \cB^*$ attains the supremum in the definition of $\sigma_\cB$, then $\hat{\sigma}_\cB^2 \geq \hat{\sigma}_{l^*}^2$ where $\hat{\sigma}_{l^*}^2$ is the empirical estimator of $\mathrm{Var}(l^*(X_i)) = \sigma_\cB^2$ based on the sample $\left(l^*(X_i)\right)_{1 \leq i \leq n}$. Thus, $\hat{\sigma}_\cB$ can be substituted for $\hat{\sigma}_{l^*}$ in confidence upper bounds for the (one-dimensional) standard deviation, such as those of \cite{MaurerPontil2009}, yielding a valid non-asymptotic confidence upper bound for $\sigma_\cB$ and thus a valid confidence region for $\mu$ using Theorem \ref{thm_conf_reg_mean}. Alternatively, if $\norm{\cdot}$ is an $\ell^p$-norm on $\mathbb{R}^d$, $\sigma_\cB$ and $\hat{\sigma}_\cB$ can be replaced by $\norm{\sigma}_p, \norm{\hat{\sigma}}_p$, where
\[ \sigma = \left(\sqrt{\mathrm{Var}(\langle e_j, X_{1} \rangle)} \right)_{1 \leq j \leq d},  \hat{\sigma} = \left( \sqrt{\frac{1}{n} \sum_{i = 1}^n \langle e_j, X_i - \overline{X}_n \rangle^2 }  \right)_{1 \leq j \leq d} \]
and $(e_j)_{1 \leq j \leq d}$ is the canonical basis,
as proposed by \cite{arlot2010resampling} for their bound in the Gaussian case. The following lemma shows that $\norm{\sigma}_p, \norm{\hat{\sigma}}_p$ are never smaller than $\sigma_\cB,\hat{\sigma}_\cB$, respectively. Thus, in theory, $\hat{\sigma}_\cB$ provides a better confidence bound than $\norm{\hat{\sigma}}_p$ when $(E,\norm{\cdot}) = (\mathbb{R}^d, \norm{\cdot}_p)$. However, $\norm{\hat{\sigma}}_p$ may be preferable in practice since it is much easier to compute. 

\begin{lemma}\label{lemma_Lp_var}
    Assume that $(E,\norm{}) = (\mathbb{R}^d, \norm{\cdot}_p)$ for some $p \geq 1$. Let $X = (X_{j})_{1 \leq j \leq d}$ be a square integrable random vector. Let $\cB^*$ be the unit ball of the dual $E^*$. For any $j \in \{1,\ldots,d\},$ let $\sigma_j$ be the standard deviation of $X_j$. Define
    \[ \sigma_\cB^2 = \sup_{l \in \cB^*} \mathrm{Var}(l(X)). \]
   Let $\sigma = (\sigma_j)_{1 \leq j \leq d}$. Then
    \[ \sigma_\cB \leq \norm{\sigma}_p = \left( \sum_{j = 1}^d \sigma_j^p \right)^{1/p}. \]
\end{lemma}
\noindent The proof of Lemma \ref{lemma_Lp_var} can be found in Section \ref{ssec_sm_conf_reg} of the Supplementary Material.



\section{Application II: Permutation two-sample testing}\label{sec_two_sample_test}
Let us introduce the non-parametric permutation two-sample test, that will be the statistical problem of interest in this section. 

Assume that we have access to two independent samples $X = (X_1,\ldots,X_n)$ and $Y = (Y_1,\ldots,Y_m)$, of distributions $P^{\otimes n}$ and $Q^{\otimes m}$ respectively. Let also $Z = (X,Y)$ be the concatenation of the two samples. The classical two-sample testing problem is to test the null hypthesis $H_0:P=Q$ against the alternative $H_1:P\neq Q$. 

A generic approach to two-sample testing consists in choosing a class of functions $\cF$ and considering the following test statistic:
\begin{equation} \label{def_stat_test}
    T_{n,m}(\cF) = \sup_{f \in \cF} \left\{ \frac{1}{n} \sum_{i = 1}^n f(X_i) - \frac{1}{m} \sum_{j = 1}^m f(Y_j) \right\}\;.
\end{equation}  
Classical examples, such as Kolmogorov-Smirnov, Kernel Mean Discrepancy tests or testing with respect to the Wasserstein distance, fall into this setting and will be further discussed in Section \ref{sec_examples} below. 

Heuristically, the statistic $T_{n,m}(\cF)$ is expected to be larger under the alternative than under the null hypothesis. The test will thus depend on a threshold, above which the null hypothesis is rejected.

In general, the distribution of the statistic \eqref{def_stat_test} under the null hypothesis is unknown. Hence, a bootstrap method can be used to select the significance threshold for the test. More precisely, let
\[ w = \left(\underbrace{\frac{1}{n},\ldots,\frac{1}{n}}_{n \text{ times}}, \underbrace{\frac{- 1}{m},\ldots,\frac{-1}{m}}_{m \text{ times}} \right) \]
and let $\xi = (w_{\sigma(i)})_{1 \leq i \leq n+m}$ , where $\sigma$ is a uniform random permutation. Remark that $\xi$ is exchangeable and that $\sum_{i = 1}^{n+m} \xi_i = 0$.
Hence, the results of Sections \ref{sec_exch_self_bounding} and \ref{sec_cond_X} above apply for $\xi$.

The so-called permutation test (\cite{Pitman37,Hoeff52}) rejects at level $\alpha$ when
\[ \{T_{n,m}(\cF) \geq \hat{q}_B(\alpha) \}\;,\]
where $\hat{q}_B(\alpha) = \hat{q}_B(Z,\alpha)$ is the empirical $1-\alpha$ quantile based on the sample
\[  \left( \sup_{f \in \cF} \left\{ \sum_{i = 1}^{n+m} \xi_i^{(b)} f(Z_i) \right\} \right)_{b = 0,\ldots,B}, \]
with $\xi^{(0)}=w$ and $(\xi^{(b)})_{1 \leq b \leq B}$ an i.i.d. sample, independent from $(Z,\xi)$, with the same distribution as $\xi$. More precisely, by setting $T_{\xi}(Z)=\sup_{f \in \cF} \left\{ \sum_{i = 1}^{n+m} \xi_i f(Z_i) \right\}$, we have
\begin{equation*}
     \hat{q}_B(\alpha)= \inf\{ q\,/\,\hat{F}^{(\ell)}_B(q)\geq 1-\alpha\}\;,
\end{equation*}
where $\hat{F}^{(\ell)}_B$ is the left-continuous version of the empirical cumulative distribution function of $(T_{\xi^{(0)}}(Z),\ldots,T_{\xi^{(B)}}(Z))$,
\[
\hat{F}^{(\ell)}_B (x) := \frac{1}{B+1} \sum_{b=0}^{B} \mathbb{I}_{\left\{T_{\xi^{(i)}}< x\right\}}\;.
\]
Equivalently, if $T^{[0]}\leq\ldots\leq T^{[B]}$ stands for the order statistics of $(T_{\xi^{(0)}}(Z),\ldots,T_{\xi^{(B)}}(Z))$, then $\hat{q}_B(\alpha)=T^{[m]}$, with $m=\lceil (B+1)(1-\alpha)\rceil$, where $\lceil x\rceil$ stands for the smallest integer larger than or equal to the real number $x$. Using exchangeability, it can be shown (see for instance \cite[Lemma 1]{RomWolf:05}) that this test controls the type-I error at level $\alpha$, for any value of $B$, that is: $P_{H_0}(T_{n,m}(\cF) \geq \hat{q}_B(\alpha))\leq \alpha$, where the notation ``$P_{H_0}$'' indicates that the probability of the event is computed under the assumption that $P=Q$.

The concentration inequalities of Theorems \ref{thm_intro_cond_exp} and \ref{thm_intro_concen_exchange} will allow us to give a non-asymptotic analysis of the power of the permutation two-sample test in terms of the integral probability metric
\[ d_{\cF}(P,Q) := \sup_{f \in \cF} \left\{ \int f dP - \int f d Q \right\}\;. \]

Note that if we assume that $d_{\cF}$ is a distance - as it is the case in the examples of Section \ref{sec_examples} -, then it is symmetric in its arguments. Hence, the supremum over $\cF$ in its definition is equal to the supremum over $\cF \cup -\cF$. We will thus assume that $\cF$ is a symmetric class, in the sense that if the function $f$ belongs to $\cF$ then the function $-f$ also belongs to $\cF$.

Our approach in non-asymptotic, based on Theorems \ref{thm_intro_cond_exp} and \ref{thm_intro_concen_exchange}. We refer to \cite[Section 3.8.1]{vandervaartWellner:23}
for an account on the asymptotic theory of permutation two-sample testing.


\subsection{A general result}\label{ssec_gen_res}
Let us denote, for a distribution $R\in\left\{ P,Q\right\}$, $\sigma^2_R(\cF)=\sup_{f\in\cF}\left\{\var_{R}(f(Z)) \right\}$, $V=\sup_{f\in \cF} \left\{ n \var(f(X_1)) + m\var(f(Y_1)) \right\}$ and, for any positive integer $k$,
$$M_{k}(R)=\EE_{R^{\otimes k}}\left[\sup_{f\in \cF}\left\{\sum_{i=1}^{k} f(Z_i)-\EE[f(Z_i)]\right\}\right].$$
\begin{theorem}\label{th_test_separation_rate}
Grant the notations above, take $\alpha, \delta \in (0,1)$ and define 
\[ \alpha_B = \left(1 + \frac{1}{B} \right) \left( \alpha - \sqrt{\frac{3\alpha}{B} \log \left( \frac{1}{\delta} \right)} - \frac{1}{B+1} \right). \]
Consider the test rejecting the null hypothesis if $T_{n,m}(\cF)>\hat{q}_B(\alpha)$ and assume that $\alpha_B\in (0,1)$. Assume also that all the functions in $\cF$ are valued in $[-1,1]$. Then the test power is at least equal to $3\delta$ if either of the following inequalities holds:
\begin{align*}
  &\left(1 - \frac{2}{\sqrt{n+m-1}} \right)  d_{\cF}(P,Q) \\
  &\quad \geq  \frac{2}{n}(M_n(P)+M_n(Q))+\frac{2}{m}(M_m(P)+M_m(Q))+ \frac{12}{n+m}\log\left(\frac{1}{\delta}\right) \\
  &\quad + \sqrt{\frac{1}{n} + \frac{1}{m}} \left(2\sqrt{2\log \left( \frac{1}{\alpha_B} \right)} + \sqrt{2\log \left( \frac{1}{\delta} \right)} \right) \numberthis \label{eq_thm_sep_rates_hoeff}
\end{align*}
or
\begin{align*}
    &\left(1 - \frac{2}{\sqrt{n+m-1}}-4\sqrt{3\left(\frac{1}{n}+\frac{1}{m}\right)\log\left(\frac{1}{\alpha_B}\right)} \right) d_{\cF}(P,Q)\\
    &\quad \geq \left(\frac{1}{n}\vee \frac{1}{m}\right)\log\left( \frac{1}{\delta}\right)+  \frac{2}{n}(M_n(P)+M_n(Q))+\frac{2}{m}(M_m(P)+M_m(Q))+ \frac{12}{n+m}\log\left(\frac{1}{\delta}\right) \\
    &+ \sqrt{2\left( \frac{\sigma^2_P(\cF)}{n} + \frac{\sigma^2_Q(\cF)}{m} \right)\log\left(\frac{1}{\delta}\right)} +2\left(\frac{1}{n}+\frac{1}{m}\right)\sqrt{\log\left(\frac{1}{\alpha_B}\right)} \times  \\ 
    &\times\sqrt{\left(34 (M_n(P)+M_m(Q))+\frac{2m}{n(n+m)} M^2_n(P)+ \frac{2n}{m(n+m)} M^2_m(Q)+V+4\log\left(\frac{1}{\delta}\right)\right)}\;. \numberthis \label{eq_thm_sep_rates}
\end{align*}
\end{theorem}
The proof of Theorem \ref{th_test_separation_rate} makes use of Theorems \ref{thm_intro_cond_exp} and  \ref{Th_Tolsti} and
is deferred to Section \ref{ssec_two_sample_proof} of the Supplementary Material. Indeed, for the specific weights of the permutation test, Tolstikhin's inequality (Theorem \ref{Th_Tolsti}) achieves better constants than Theorem \ref{thm_intro_concen_exchange}. See also the related discussion in Section \ref{proof_Tolstikhin_exch_pair}.

Let us comment on the separation rates obtained in Theorem \ref{th_test_separation_rate}. Note first that, as expected due to the symmetry between the pairs $(P,n)$ and $(Q,m)$ in the testing problem, the bound are symmetric in these variables.

Another preliminary remark consists in noting that the quantity $M_k(R)/k$ is non-increasing in $k$ for any probability measure $R$ (see Lemma  in Section \ref{sssec_two_sample_tech_lem} of the Supplementary Material). Hence, we have $m/(n(n+m))M^2_n(P)=\mathcal{O}(M_n(P))$ and the same holds true with $P$ replaced by $Q$ and the pair $(n,m)$ switched. 

Assume that the length of the two samples are similar, \textit{i.e.} $n \asymp m$. 
If the class $\cF$ is such that $\sqrt{n} = \mathcal{O}(M_n(R))$ for $R \in \{P,Q\}$ -- which is the case as soon as $\cF$ contains two functions that are not $R$-\textit{a.e.} equal --, then
the separation rate in both \eqref{eq_thm_sep_rates_hoeff} and \eqref{eq_thm_sep_rates} is
\[ \max \left\{ \frac{M_{n}(P)}{n}, \frac{M_{m}(Q)}{m} \right\} \;, \]
\textit{i.e.} the rate of convergence in the uniform law of large numbers for the class $\cF$.  


It is worth noting that the main difference between Inequalities \eqref{eq_thm_sep_rates_hoeff} and \eqref{eq_thm_sep_rates} is that Inequality \eqref{eq_thm_sep_rates} has factors that depend on the variances $\sigma_R^2(\cF)$, $R\in \left\{P,Q\right\}$, and $V$, whereas in comparison, the controls are uniform (\textit{i.e.} numerical constants replace the variance terms) in Inequality \eqref{eq_thm_sep_rates_hoeff}. From a technical viewpoint, this is due to the two possible choices provided by Proposition \ref{prop_bound_Vplus_perm} for bounding the quantity of the form $V_+(g,\sigma)$ appearing in our proof of Theorem \ref{th_test_separation_rate}. If the variance terms are of the order of absolute constants, then Inequalities \eqref{eq_thm_sep_rates_hoeff} and \eqref{eq_thm_sep_rates} are of the same order. 

Finally, note that 
in the case of the Kolmogorov-Smirnoff test and the Wasserstein test in dimension 1, it holds $M_n(R) = \mathcal{O}(\sqrt{n})$ for $R \in \{P,Q\}$, which gives a parametric separation rate, of the order $ \mathcal{O}(1/\sqrt{n})$. See Section \ref{sec_examples} below for more details about these specific tests.

\subsection{Some examples}\label{sec_examples}

 Fix some confidence level $\alpha \in \left(0; \frac{1}{3} \right]$, some desired power level $1 - 3\delta$ and define $\alpha_B$ as in Theorem \ref{th_test_separation_rate}, i.e
\[ \alpha_B = \left(1 + \frac{1}{B} \right) \left( \alpha - \sqrt{\frac{3\alpha}{B} \log \left( \frac{1}{\delta} \right)} - \frac{1}{B+1} \right).  \]

\subsubsection{Kolmogorov-Smirnov two-samples test}

The test statistics $T^{(KS)}_{n,m}$ used in the Kolmogorov-Smirnov two-sample test \cite{smirnov1939estimation} is
\[
T^{(KS)}_{n,m} = \sup_{x \in \mathbb{R}} \left| \frac{1}{n} \sum_{i = 1}^n \mathbb{I}\{X_i \leq x \} - \frac{1}{m} \sum_{j = 1}^m \mathbb{I}\{Y_j \leq x\} \right|\;.
\]
Using the notations of Section \ref{ssec_gen_res} above, it holds $T^{(KS)}_{n,m}=T_{n,m}(\cF)$ for 
\[
    \cF:=\left\{ c\mathbb{I}_{(-\infty,x]}\, / \, c\in\left\{ -1, 1\right\}, \; x\in \R \right\} \;,
\]
the symmetrized class of indicators of right-closed half-lines in $\R$. The corresponding distance $d_{\cF}$ between probability measures is the so-called Kolmogorov distance, given by 
\[
  d_{\cF}(P,Q)= d_{KS}(P,Q) =  \sup_{x\in \R} \vert F(x)-G(x)\vert\;,
\]
where $F$ and $G$ are the cumulative distribution functions of $P$ and $Q$ respectively. The class $\cF$ is bounded, moreover for any distribution $P$,
\[ \sigma^2_P(\cF) = \sup_{f \in \cF} \var(f(X)) = \sup_{x \in \mathbb{R}} \{ F(x) (1-F(x)) \} \leq \frac{1}{4}\;, \]
with equality when the measure $P$ is atomless. This means that there is little to gain by taking the variance into account: thus, we apply Equation \eqref{eq_thm_sep_rates_hoeff} of Theorem \ref{th_test_separation_rate}. By the DKW inequality \cite{DKW1956} with optimal constant \cite{DKW1990}, for any $x > 0$, any $k \in \mathbb{N}$ and any i.i.d. sample $(Z_i)_{1 \leq i \leq k}$ with common distribution function $H$,
\[ \mathbb{P} \left( \sup_{t \in \mathbb{R}} \left| \frac{1}{k} \sum_{i = 1}^k \mathbb{I}\{Z_i \leq t \} - H(t) \right| \geq x \right) \leq 2e^{-2kx^2}\;, \]
which implies that, for all $k \in \mathbb{N}$ and any distribution $R$,
\[ M_k(R) = \mathbb{E} \left[ k \sup_{t \in \mathbb{R}} \left| \frac{1}{k} \sum_{i = 1}^k \mathbb{I}\{Z_i \leq t \} - H(t) \right| \right] \leq 2k \int_0^{+\infty} e^{-2k x^2} dx = \sqrt{\frac{k \pi}{2}}\;. \]
The following result is a corollary of Theorem \ref{th_test_separation_rate} and the above bounds. 

\begin{corollary} \label{cor_KS}
    The permutation test based on the Kolmogorov-Smirnoff test statistic has power at least $1 - 3\delta$ whenever
    \begin{align*}
        \left(1 - \frac{2}{\sqrt{n+m-1}} \right) d_{KS}(P,Q) &\geq \sqrt{\frac{1}{n} + \frac{1}{m}} \left(2\sqrt{2\log \left( \frac{1}{\alpha_B} \right)} + \sqrt{2\log \left( \frac{1}{\delta} \right)} \right) \\
        &\quad + 2\sqrt{2\pi} \left( \frac{1}{\sqrt{n}} + \frac{1}{\sqrt{m}} \right) + \frac{12}{n+m} \log \left( \frac{1}{\delta} \right) \;.
    \end{align*}
\end{corollary}

Usually, the Kolmogorov-Smirnov statistic $T_{n,m}(\cF)$ is used together with a threshold of the form    
\[ c_\alpha \sqrt{\frac{1}{n} + \frac{1}{m}} \approx \sqrt{2 \log \left( \frac{2}{\alpha} \right)} \sqrt{\frac{1}{n} + \frac{1}{m}} \]
based on the limiting distribution of $T_{n,m}(\cF)$ worked out by Smirnov \cite{smirnov1939estimation}. 
This approximation can be poor for finite samples, while the exact critical value is hard to compute \cite{Hodges1958}. Using the DKW inequality \cite{DKW1956} with optimal constant \cite{DKW1990}, the power of the above test can be shown to be at least $1 - \delta$ whenever
\[ d_{KS}(P,Q) \geq c_\alpha \sqrt{\frac{1}{n} + \frac{1}{m}}  + \sqrt{\frac{1}{2} \log \left( \frac{2}{\delta} \right)} \left( \frac{1}{\sqrt{n}} + \frac{1}{\sqrt{m}} \right). \]
Corollary \ref{cor_KS} shows that a similar guarantee holds for the permutation test,
which is exact and does not rely on the knowledge of the asymptotic distribution of $T_{n,m}(\cF)$.

\subsubsection{Testing with respect to the Wasserstein distance}
In this section, we make use of the so-called Wasserstein distance $d_W$ -- or Wasserstein-$1$ distance -- defined as follows: for any two distributions $P$ and $Q$ on a metric space $(\cX, d)$, the Wasserstein distance between $P$ and $Q$ is
\begin{equation*}
    d_W(P,Q)=\sup_{\cF}\left\{Pf-Qf \right\}\;,
\end{equation*}
where $\cF={\rm Lip}_1(\cX)$ is the set of $1-$Lipschitz functions on $(\cX,d)$. Given two samples with empirical distributions $P_n$ and $Q_m,$ the Wasserstein distance $d_W(P_n,Q_m)=T_{n,m}(\cF)$ can be used as a test statistic for the two-sample problem. We refer to Ramdas \cite{RamdasTrillosCuturi:17} for a recent survey on Wasserstein two-sample testing, with a focus on relations to other classical testing problems. Unlike the case of the Kolmogorov-Smirnov statistic, the distribution of $d_W(P_n,Q_m)$ under the null depends on $P$. Ramdas et al. \cite{RamdasTrillosCuturi:17} resolve this problem in dimension $1$ by modifying the test statistic, but their approach is specific to the one-dimensional case. Instead, we consider here the use of the bootstrap to calibrate the threshold. As we discussed previously, this yields an exact test no matter what the distribution of $T_{n,m}(\cF)$ actually is. 


When $\cX$ is compact, the functions belonging to $\cF$ can be centered by their respective mean value, with respect to any probability measure on $\cX$, in order to ensure that they all take values in $[-{\rm diam}(\cX), {\rm diam}(\cX)]$, where ${\rm diam}(\cX)=\sup\left\{ d(x,y):x,y\in \cX\right\}$ is the diameter of the set $\cX$. This comes from the following identity: for any probability measure $P$ on $\cX$, any $f \in \cF$ and any $x\in \                  
                  \cX$, we have
\begin{equation}
    f(x)-Pf=\int (f(x)-f(y))P(dy) \leq \int_{\cX} d(x,y) dP(y) \leq \mathrm{diam}(\cX)\;.
\end{equation}
Moreover, for any $f \in \cF,$
\[ \mathrm{Var}(f(X)) = \frac{1}{2} \mathbb{E} \left[ (f(X) - f(X'))^2 \right] \leq \frac{1}{2} \mathbb{E} \left[ d(X,X')^2 \right]\;,  \]
so we have that
\[ \sigma_P^2(\cF) = \sup_{f \in \cF} \left\{ \mathrm{Var}(f(X)) \right\} \leq \frac{1}{2} \mathbb{E} \left[ d(X,X')^2 \right]. \]
The separation rate of the test based on $ d_W(P_n,Q_m)$ is determined by the rate of growth of $M_{l}(R) = l \mathbb{E}[d_W(R_{l},R)]$ for $R \in \{P,Q\}$ and $l \in \{n,m\}$,  \textit{i.e.} by the rate of convergence of $\mathbb{E}[d_W(R_{l},R)]$ to $0$,  where $R_{l}$ denotes the empirical distribution of an i.i.d. sample of size $l$ from $R$. 

Let us consider the case where $\cX = \mathbb{R}^k$ and $d(x,y) = \norm{x-y}$ for some norm $\norm{\cdot}$.
If the dimension $k = 1, \norm{\cdot} = |\cdot|$, and we have that
\[ d_W(R_{l},R) = \int_{\mathbb{R}} \left| F_{l}(t) - F_R(t) \right| dt\;, \]
where $F_{l}, F_R$ are the cumulative distribution functions of $R_{l}$ and $R$, respectively. It follows by Jensen's inequality that
\[ M_{l}(R) \leq \sqrt{l} \int_{\mathbb{R}} \sqrt{F_R(t)(1-F_R(t))} = \mathcal{O} \left( \sqrt{l} \right) \;, \]
provided that $J_1(R)=\int_{\mathbb{R}} \sqrt{F_R(t)(1-F_R(t))} < +\infty$. As proved in \cite[Section 3]{BobkovLedoux:14}, the finiteness of $J_1(R)$ is in fact necessary and sufficient for a convergence of $\EE[d_W(R_l,R)]$ at the rate $n^{-1/2}$. 

In dimension $d \geq 2$, it follows from the work of Fournier and Guillin \cite{Fournier2015} that for any $q > 2$ and some constant $C$ depending only on $q,d$ and the norm $\Vert \cdot \Vert$,
\begin{equation}
\mathbb{E} \left[ d_W(R_{l},R) \right] \leq C \mathbb{E} \left[ d(X,\mathbb{E} X)^q \right]^{1/q} \frac{\log(1+l)}{\sqrt{l}} \label{ineq_W_dim2}
\end{equation}
when $d = 2$ and
\begin{equation}\label{ineq_W_dimsup3}
\mathbb{E} \left[ d_W(R_{l},R) \right] \leq C \mathbb{E} \left[ d(X,\mathbb{E} X)^q \right]^{1/q} l^{- 1/d} 
\end{equation}
when $d \geq 3$. Fournier and Guillin \cite{Fournier2015} give examples that show that this rate is attained for some distribution $R$ on $\mathbb{R}^d$, up to a $\sqrt{\log l}$ factor when $d = 2$. Unlike the one-dimensional case, this is only a worst-case bound: for example, if $R$ is concentrated on an affine subspace of dimension $k < d$, $d_W(R_{l},R)$ behaves as in dimension $k$. Thus, under the assumption that the distributions $P,Q$ admit a moment of order $q > 2,$ we are able in Corollary \ref{cor_wass_test} below to derive the separation rate for the Wasserstein two-sample test.

To simplify the bounds, we assume that the two samples are of equivalent size, that is $n \leq m \leq \rho n$ for some $\rho > 1$.
Theorem \ref{th_test_separation_rate} then yields the following Corollary.

\begin{corollary}\label{cor_wass_test}
    Let $\cX \subset \mathbb{R}^k$ be a Borel set of diameter $\mathrm{diam}(\cX) \leq 1$.
    Let $d$ be the distance on $\cX$ associated with some norm $\norm{\cdot}$ on $\mathbb{R}^k$. 
    For every distribution $P$ on $\cX$ and every $r \geq 2$, let
    \[ V_r(P) = \mathbb{E} \left[ d(X,\mathbb{E}[X])^r \right]^{1/r}. \]
    Assume that the sample sizes $n$ and $m$ are such that there exists $\rho > 1$ satisfying $n \leq m \leq \rho n$ and moreover, assume that 
    \[  n \geq n_0 = \left\lceil 16 \left(1 + 2\sqrt{6 \log \left( \frac{1}{\alpha_B} \right)} \right)\right\rceil \;. \]
    For any $r > 2$, the permutation test based on the Wasserstein-$1$ distance $d_W$ has power at least $1-3\delta$ whenever 
    \begin{align*}
        d_W(P,Q) &\geq C \left( \frac{L_n}{n^\gamma} + \frac{L_m}{m^\gamma} \right) (V_r(P) + V_r(Q))
        + \sqrt{\frac{1}{n} \log \left( \frac{1}{\alpha_B} \right)} \left(4\sqrt{2(V_2(P) + \rho V_2(Q))} + c_1 \frac{L_n}{n^{\gamma/2}} \right) \\
        &\quad + 2\sqrt{\left( \frac{V_2(P)}{n} + \frac{V_2(Q)}{m} \right) \log \left( \frac{1}{\delta} \right)} + c_2 \frac{L_n}{n} \left( \log \left(\frac{1}{\alpha_B} \right) + \log \left(\frac{1}{\delta} \right) \right),
    \end{align*}
    where $\gamma = \min \left\{1/2, 1/k\right\}$, 
    \[ L_j = \begin{cases}
        & \log (1+j) \text{ if } k = 2 \\
        &1 \text{ else}\;,
    \end{cases} \]
    $C$ is a constant depending only on $r,k,\norm{\cdot}$, while $c_1,c_2$ depend also on $\rho,V_r(P),V_r(Q),V_2(P),V_2(Q)$.
\end{corollary}

Thus, the separation rate of the bootstrap Wasserstein test is of order $\frac{1}{\sqrt{n}}$ in dimension $1$ and almost of order $\frac{1}{\sqrt{n}}$ in dimension $2$ (up to a $\log n$ factor). For dimensions greater than $2$, the worst-case rate worsens: there is a curse of dimensionality. Note however that this applies only to continuous distributions: for singular distributions concentrated on a subspace, the separation rate of the Wasserstein test adapts to the lower dimensionality. For continuous distributions, a possible solution considered in the literature is to replace $T_{n,m}(\cF)$ with a dimensionally reduced version. If this dimensionally reduced Wasserstein distance is an integral probability metric, then it can likewise be analysed using Theorem \ref{th_test_separation_rate}. For example, this is the case of the max-sliced Wasserstein distance \cite{Deshpande2019MaxSlicedWD}.



\subsubsection{Kernel Maximum Mean Discrepancy Tests}\label{sec_MMD_tests}

The Maximum Mean Discrepancy between $P$ and $Q$ is an integral probability metric defined as
\begin{equation}
    {\rm MMD}(\cH,P,Q)=\max_{f\in B^1_\cH}\left\{ \EE_P[f(X)] - \EE_Q[f(Y)]\right\},
\end{equation}
where $\cH$ is a Reproducing Kernel Hilbert Space associated to a (Mercer) kernel $k(\cdot, \cdot)$ on a measure space $\cX$ and $B_\cH=\left\{f\in \cH\, / \,\Vert f \Vert_\cH \leq 1\right\}$ is the unit ball of $\cH$ endowed with its natural norm.
Assuming that the kernel is bounded, that is,
\[ \kappa = \sup_{x \in \cX} \sqrt{k(x,x)} < +\infty, \]
then (by a standard argument) the functions of $\cF = B^1_\cH$ are uniformly bounded by $\kappa$, so the results of this article apply.

Let us however mention that in the context of the MMD test, the use of empirical process techniques can be bypassed through the use of test statistics that corresponds to an empirical version of the following formula,
\begin{equation}
    {\rm MMD}^2(\cH,P,Q)=\EE [k(X,X^\prime)]+\EE [k(Y,Y^\prime)]-2\EE [k(X,Y)],
\end{equation}
where $X^\prime$ is a copy of $X$, similarly for $Y^\prime$ and $Y$, and all the random variables are independent. Nevertheless, it is instructive to compare the results derived from theorem \ref{th_test_separation_rate} with what can be achieved using these alternative techniques.

In the RKHS setting, the general method considered in this article yields the test statistic
\begin{align*}
   T_{n,m}^{\cH} &= {\rm MMD}(\cH,P_n,Q_m) \\
   &= \norm{\frac{1}{n} \sum_{i = 1}^n k(X_i,\cdot) - \frac{1}{m} \sum_{j = 1}^m k(Y_j,\cdot)}_{\cH} \\
   &= \left( \frac{1}{n^2} \sum_{i = 1}^n \sum_{i' = 1}^n k(X_i,X_{i'}) + \frac{1}{m^2} \sum_{j = 1}^m \sum_{j' = 1}^m k(X_j,X_{j'}) - \frac{2}{nm} \sum_{i = 1}^n \sum_{j = 1}^n k(X_i,Y_j) \right)^{1/2} \numberthis \label{eq_kern_MMD}
\end{align*}
where $P_n,Q_m$ denote the empirical distributions associated with the two samples $X = X_1,\ldots,X_n$ and $Y = Y_1,\ldots,Y_m$. This test statistic was one of several considered by Gretton \cite{MMD-Gretton2012}, under the notation ${\rm MMD}_b(\cH,X,Y)$. Alternative, unbiased test statistics can be obtained by eliminating some of the terms in equation \ref{eq_kern_MMD} . Concerning $T_{n,m}^{\cH}$, Theorem 7 in \cite{MMD-Gretton2012} shows that the test which rejects when
\[ T_{n,m}^{\cH} \geq \kappa \left(\frac{2}{\sqrt{n}} + \frac{2}{\sqrt{m}} + \sqrt{2\left( \frac{1}{n} + \frac{1}{m} \right) \log \left( \frac{2}{\alpha} \right)} \right) \]
has level $\alpha$ and separation rate $\frac{\kappa}{\sqrt{n}}$ in ${\rm MMD}$ distance. In comparison, the permutation test estimates the optimal threshold instead of using a worst-case bound and thus should adapt better to favourable properties of the distributions $P,Q$. Using Theorem \ref{th_test_separation_rate}, we can show that its separation rate is still bounded by $\frac{\kappa}{\sqrt{n}}$ in the worst case.

\begin{corollary} \label{cor_MMD}
    When
    \begin{align*}
       &\left(1 - \frac{2}{\sqrt{n+m-1}} \right) {\rm MMD}(\cH,P,Q) \\ 
       &\quad \geq \kappa \sqrt{\frac{1}{n} + \frac{1}{m}} \left(2\sqrt{2 \log \left( \frac{1}{\alpha_B} \right)} + \sqrt{2 \log \left( \frac{1}{\delta} \right)} \right) + \frac{4\kappa}{\sqrt{n}} + \frac{4\kappa}{\sqrt{m}} + \frac{12\kappa}{n+m} \log \left( \frac{1}{\delta} \right),
    \end{align*}
    the MMD permutation test rejects the null hypothesis with probability at least $1-3\delta$.
\end{corollary}

When $\delta = \alpha$ and $B$ is large enough, the lower bound of Corollary \ref{cor_MMD} is the same as Gretton's proposed threshold, up to a numerical factor of $3$ and the negligible remainder term
\[ \frac{12\kappa}{n+m} \log \left( \frac{1}{\delta} \right). \]

\section{Supplementary material}

\subsection{Around measurability}\label{sec_proof_meas}

\begin{proof}[Proof of Lemma \ref{lemma_meas}]
Let $\cT_0$ be a countable dense subspace of $\cT$ in the product topology. This means that for any integer $n \geq 1$, any $x \in E^n, t \in \cT$ and $\varepsilon > 0$, there exists $t_0 \in \cT_0$ such that
\[ \max_{1 \leq i \leq n} |t(x_i) - t_0(x_i)| < \varepsilon. \]
In particular, for any $x \in E^n,$ the set $\cT_0(x) = \{ (t(x_1),\ldots,t(x_n)) : t \in \cT_0 \}$ is dense in the set $\cT(x) = \{ (t(x_1),\ldots,t(x_n)) : t \in \cT \}$ (as subsets of $[-1,1]^n$). Thus, by continuity of the scalar product, for any $w \in \mathbb{R}^n$,
\[ \sup_{t \in \cT} \left\{ \sum_{i = 1}^n w_i t(x_i) \right\} = \sup_{a \in \cT(x)} \{\langle a,w \rangle \} = \sup_{a \in \cT_0(x)} \{\langle a,w \rangle \} = \sup_{t \in \cT_0} \left\{ \sum_{i = 1}^n w_i t(x_i) \right\}. \]
    It follows that for all $x \in E^n$ and $w \in \mathbb{R}^n$, $g(x,w) = g_0(x,w)$ where
    \[ g_0(x,w) = \sup_{t \in \cT_0} \left\{ \sum_{i = 1}^n w_i t(x_{i}) \right\}. \]
    Let $(t_k)_{k \in \mathbb{N}}$ be an enumeration of $\cT_0$, then $g_0$ is the pointwise limit of the functions 
    \[ h_N : x,w \mapsto \max_{1 \leq k \leq N} \left\{ \sum_{i = 1}^n w_i t_k(x_i)  \right\} \]
    which are measurable on $E^n \times \mathbb{R}^n$.
    This proves that $g_0 = g$ is measurable. 
\end{proof}

\begin{proof}[Proof of Lemma \ref{lem_meas}]
    Let $\tilde{f}$ be the minimal measurable majorant of $f$ with respect to the law of $X$.
    By definition, $\tilde{f}(X) \geq f(X)$ almost surely. Moreover, since $f(X)$ is integrable
    \[ \mathbb{E}[\tilde{f}(X) - f(X) | X] = \tilde{f}(X) - \mathbb{E}[f(X) | X] \geq 0  \]
    almost surely. Since $\mathbb{E}[f(X) | X] $ is of the form $\overline{f}(X)$ for a measurable function $\overline{f}$ we have that $\tilde{f} \geq \overline{f}$ $P_X-$almost surely which implies by definition that $\tilde{f} = \overline{f}$ $P_X-$almost surely. Thus, $\tilde{f}(X) = \mathbb{E}[f(X) | X]$ a.s. which yields
    \[ \mathbb{E}[\tilde{f}(X) - f(X)] = 0 \]
    and hence (since $\tilde{f}(X) \geq f(X)$), $\tilde{f}(X) = f(X)$ a.s.
    Fix some $i \in \{1,\ldots,n\}$.
    Let $\overline{f}_i$ be a version of the minimal measurable majorant of $x \mapsto f_i(x_{(i)})$ on $E^{n}$ with respect to the law of $X$. Since $\tilde{f}$ is measurable, $\overline{f}_i \leq \tilde{f}$ almost surely. Moreover, by Fubini's theorem, for almost all $z \in E$,
    \[ \overline{f}_i(x_1,\ldots,x_{i-1},z,x_{i+1},\ldots,x_n) \geq f_i(x_{(i)}) \text{ a.s.} \]
    By definition of the minimal measurable majorant, for almost all $z \in E$ and $x \in E^n$,
    \[ \overline{f}_i(x_1,\ldots,x_{i-1},z,x_{i+1},\ldots,x_n) \geq \overline{f}_i(x) = \overline{f}_i(x_1,\ldots,x_{i-1},x_i,x_{i+1},\ldots,x_n). \]
    Now, $z$ and $x_i$ play symmetrical roles so almost surely wrt $z,x$,
    \[ \overline{f}_i(x_1,\ldots,x_{i-1},z,x_{i+1},\ldots,x_n) = \overline{f}_i(x). \]
    We may pick any such $z$ and set
    \[ \tilde{f}_i : y \mapsto \overline{f}_i(y_1,\ldots,y_{i-1},z,y_{i+1},\ldots,y_n) \]
    on $E^{n-1}$.  $\tilde{f}_i$ is measurable and such that $f_i(X_{(i)}) \leq \tilde{f}_i(X_{(i)}) \leq \tilde{f}(X) = f(X)$ almost surely, which implies the result.
\end{proof}

\subsection{Proof of the self-bounding property}\label{ssec_self_bouning_SM}

\begin{proof}[Proof of Theorem \ref{thm_self_bounding}]
Let $x = (x_{i})_{1 \leq i \leq n}$.
Fix some $i \in \{1,\ldots,n\},$ let $J$ be uniformly distributed on $\{1,\ldots,n\}$ and let $\tau_{i,J}$ be the transposition of $i$ and $J$. By exchangeability of $\xi, \xi \circ \tau_{i,J} \sim \xi$ and hence
\begin{align*}
   \overline{g}(x) &= \mathbb{E} \left[ \sup_{t \in \cT} \left\{ \sum_{k = 1}^n \xi_{\tau_{i,J}(k)} t(x_{k}) \right\} \right] \\
   &= \mathbb{E} \left[ \sup_{t \in \cT} \left\{ \sum_{k \neq i}^n \xi_{k} t(x_{k}) + (\xi_i - \xi_J) t(x_{J}) + \xi_J t(x_{i}) \right\} \right].
\end{align*}
Let then
\[ \overline{g}_i(x_{(i)}) = \mathbb{E} \left[ \sup_{t \in \cT} \left\{\sum_{k \neq i} \xi_k t(x_{k}) + \frac{1}{n} \sum_{j \neq i}^n (\xi_i - \xi_j) t(x_{j}) \right\} \right]. \]
By Jensen's inequality and since $\mathbb{E}[\xi_J | \xi] = 0,$
\begin{align*}
   \overline{g}(x) &= \mathbb{E} \left[ \sup_{t \in \cT} \left\{ \sum_{k \neq i}^n \xi_{k} t(x_{k}) + (\xi_i - \xi_J) t(x_{J}) + \xi_J t(x_{i}) \right\} \right] \\
   &\geq \mathbb{E} \left[ \sup_{t \in \cT} \mathbb{E} \left[ \sum_{k \neq i}^n \xi_{k} t(x_{k}) + (\xi_i - \xi_J) t(x_{J}) + \xi_J t(x_{i}) \Bigl| \xi \right]  \right] \\
   &= \overline{g}_i(x_{(i)}).
\end{align*}
Define the set
\[ \mathcal{K}(x) = \overline{\bigl\{ (t(x_{i}))_{1 \leq i \leq n} : t \in \cT \bigr\} } \subset [-1,1]^n. \]
$\mathcal{K}(x)$ is a compact set, moreover
\begin{align}
    \overline{g}(x) &= \mathbb{E} \left[ \sup_{a \in \mathcal{K}(x)} \sum_{i = 1}^n \xi_i a_i \right] \\
    \overline{g}_i(x_{(i)}) &= \mathbb{E} \left[ \sup_{a \in \mathcal{K}(x)} \sum_{k \neq i} \xi_k a_k + \frac{1}{n} \sum_{j \neq i}^n (\xi_i - \xi_j) a_j \right] \label{eq_def_gbar-i}
\end{align}
by continuity of the functions within the expectation. Since $\mathcal{K}(x)$ is compact, we can find a measurable $\hat{a}: \mathbb{R}^n \to \mathcal{K}(x)$ so that
\[ \sum_{i = 1}^n y_i \hat{a}_i(y) \geq \sup_{a \in \mathcal{K}(x)} \sum_{i = 1}^n y_i a_i - \varepsilon \]
for any fixed $\varepsilon > 0$. To see this, consider an $\varepsilon$-net $(a_k)_{1 \leq k \leq N}$ in $\mathcal{K}(x)$ and let 
\[ \hat{k}(y) = \min \underset{k \in \{1,\ldots,N\}}{\mathrm{argmax}} \langle a_k,y \rangle, \quad \hat{a}(y) = a_{\hat{k}(y)} .\] 
It follow that
\[ \overline{g}(x) \leq \varepsilon + \mathbb{E}\left[ \sum_{i = 1}^n \xi_i \hat{a}_i(\xi) \right]. \]
By equation \eqref{eq_def_gbar-i} defining $\overline{g}_i$,
\begin{align*}
   \overline{g}(x) - \overline{g}_i(x_{(i)}) &\leq  \varepsilon + \mathbb{E} \left[ \sum_{k = 1}^n \xi_k \hat{a}_k(\xi) - \sum_{k \neq i} \xi_k \hat{a}_k(\xi) - \frac{1}{n} \sum_{j \neq i}^n (\xi_i - \xi_j) \hat{a}_j(\xi) \right] \\
   &\leq \varepsilon + \mathbb{E} \left[ \xi_i \hat{a}_i(\xi) + \frac{1}{n} \sum_{j = 1}^n (\xi_j - \xi_i) \hat{a}_j(\xi) \right]. \numberthis \label{eq_ubd_diff_gbar_gbar-i}
\end{align*}
Since $\hat{a}(\xi) \in [-1,1]^n,$ it follows that
\begin{align*}
   \overline{g}(x) - \overline{g}_i(x_{(i)}) &\leq \varepsilon + \mathbb{E} \left[ |\xi_i| +  \frac{1}{n} \sum_{j \neq i} |\xi_j - \xi_i| \right] \\
   &= \varepsilon + \mathbb{E} \left[ |\xi_1| \right] + \frac{n-1}{n} \mathbb{E}[|\xi_2 - \xi_1|] \text{ by exchangeability } \\
   &\leq \varepsilon + \left(3 - \frac{2}{n} \right) \mathbb{E} \left[ |\xi_1| \right] \text{ by exchangeability and triangle inequality}.
\end{align*}
Moreover, summing over $i$ in equation \eqref{eq_ubd_diff_gbar_gbar-i} yields
\begin{align*}
 \sum_{i=1}^n \left[\overline{g}(x) - \overline{g}_i(x_{(i)}) \right] &\leq n\varepsilon +  \mathbb{E} \left[ \sum_{i = 1}^n \xi_i \hat{a}_i(\xi) + \frac{1}{n} \sum_{j = 1}^n \left( \sum_{i = 1}^n (\xi_j - \xi_i) \right) \hat{a}_j(\xi)  \right] \\
 &\leq n\varepsilon +  \mathbb{E} \left[ \sum_{i = 1}^n \xi_i \hat{a}_i(\xi) + \frac{1}{n} \sum_{j = 1}^n n\xi_j \hat{a}_j(\xi)  \right] \\
 &\leq n\varepsilon +  2 \mathbb{E} \left[ \sum_{i = 1}^n \xi_i \hat{a}_i(\xi)  \right] \\
 &\leq n\varepsilon + 2\overline{g}(x).
\end{align*}
Since $\varepsilon > 0$ was arbitrary, it follows that
\[0 \leq \overline{g}(x) - \overline{g}_i(x_{(i)}) \leq \left(3 - \frac{2}{n} \right) \mathbb{E}[|\xi_1|] \]
and that
\[ \sum_{i = 1}^n \left[ \overline{g}(x) - \overline{g}_i(x_{(i)}) \right] \leq 2 \overline{g}(x). \]
This proves that the function $\frac{\overline{g}}{\left(3 - \frac{2}{n} \right)\mathbb{E}[|\xi_1|]}$ is $(2,0)$-self-bounding.
Now, if $X = (X_{i})_{1 \leq i \leq n}$ are independent random variables such that $\overline{g}(X)$ is measurable (hence integrable), then by lemma \ref{lem_meas} and \cite[Theorem 6.21]{BoucheronLugosiMassart:2013} with $b = 0$ and $a = 2$, with probability at least $1 - e^{-x}$,
 \[ \frac{\overline{g}(X)}{\kappa \left(3 - \frac{2}{n} \right)} \leq \frac{\mathbb{E}\left[ \overline{g}(X) \right]}{\kappa \left(3 - \frac{2}{n} \right)} + \frac{\sqrt{2a x \mathbb{E}[\overline{g}(X)]}}{\sqrt{\kappa \left(3 - \frac{2}{n} \right)}} + 2 \left( \frac{3a-1}{6} \right)_+ x \;,\]
which yields
\[ \overline{g}(X) \leq \mathbb{E}\left[ \overline{g}(X) \right] + \sqrt{2a \kappa \left(3 - \frac{2}{n} \right) x \mathbb{E}[\overline{g}(X)]} + \left(a - \frac{1}{3} \right) \kappa \left(3 - \frac{2}{n} \right) x \;.\]
Ignoring the $\frac{1}{n}$ terms for simplicity yields
\[ \overline{g}(X) \leq \mathbb{E}\left[ \overline{g}(X) \right] + \sqrt{12 \kappa x \mathbb{E}[\overline{g}(X)]} + 5 \kappa x. \]
For the lower tail, \cite[Theorem 6.21]{BoucheronLugosiMassart:2013} yields
\begin{align*}
   \frac{\overline{g}(X)}{\kappa \left(3 - \frac{2}{n} \right)} &\geq \frac{\mathbb{E}\left[ \overline{g}(X) \right]}{\kappa \left(3 - \frac{2}{n} \right)} - \frac{\sqrt{2 a x \mathbb{E}\left[ \overline{g}(X) \right]}}{\sqrt{\kappa \left(3 - \frac{2}{n} \right)}} \;,
\end{align*}
which gives
\begin{equation}
   \overline{g}(X) \geq \mathbb{E}\left[ \overline{g}(X) \right] - \sqrt{12\kappa x \mathbb{E}\left[ \overline{g}(X) \right]}
\end{equation}
with probability at least $1 - e^{-x}$.
\end{proof}

\subsection{Exchangeable pair covariance inequality}\label{ssec_sm_cov_ineq}
\begin{proof}[Proof Lemma \ref{lem_cov_ineq}] We first make appear the function $F$ in the covariance:
\begin{align*}
& \operatorname{cov}\left(g(Z), e^{tg(Z)}\right) \\
= & \mathbb{E}[(g(Z)-\mathbb{E}[g(Z)]) e^{tg(Z)}] \\
= & \mathbb{E}\left[\mathbb{E}\left[F\left(Z, Z^{\prime}\right) \vert Z\right] e^{tg(Z)}\right] \\
= & \mathbb{E}\left[F\left(Z, Z^{\prime}\right) e^{tg(Z)}\right]\;,
\end{align*}
where the last equality follows by using Fubini's theorem. Then, by antisymmetry of $F$ and exchangeability of $\left(Z, Z^{\prime}\right)$, it holds
$$
\operatorname{cov}\left(g(Z), e^{tg(Z)}\right)=\frac{1}{2} \mathbb{E}\left[F\left(Z, Z^{\prime}\right)(e^{tg(Z)} - e^{tg(Z^\prime)})\right]  .
$$
This implies that
$$
\begin{aligned}
& \operatorname{cov}\left(g(Z), e^{tg(Z)}\right) \\
 \leq & \left(\operatorname{cov}\left(g(Z), e^{tg(Z)}\right)\right)_{+} \\
\leq & \frac{1}{2} \mathbb{E}\left[\left(F\left(Z, Z^{\prime}\right)\left(e^{tg(Z)}-e^{tg(Z^\prime)}\right)\right)_{+}\right] \\
\leq & \frac{1}{2}\left(\mathbb{E}\left[\left(F\left(Z, Z^{\prime}\right)\right)_{+}\left(e^{tg(Z)}-e^{tg(Z^\prime)}\right)_{+}\right]\right. \\
& \left.+\mathbb{E}\left[\left(F\left(Z, Z^{\prime}\right)\right)_{-}\left(e^{tg(Z)}-e^{tg(Z^\prime)}\right)_{-}\right]\right) \\
= & \mathbb{E}\left[\left(F\left(Z, Z^{\prime}\right)\right)_{+}\left(e^{tg(Z)}-e^{tg(Z^\prime)}\right)_{+}\right],
\end{aligned}
$$
where the last equality comes again from the exchangeability of $\left(Z, Z^{\prime}\right)$ and antisymmetry of $F$. Now, notice that 
\begin{equation*}
    \left(e^{tg(Z)}-e^{tg(Z^\prime)}\right)_{+}=e^{tg(Z)}\left(1-e^{-t(g(Z)-g(Z^\prime))}\right)_{+}\leq t(g(Z)-g(Z^\prime)_+e^{tg(Z)}\;,
\end{equation*}
which gives 
\begin{align*}
    \operatorname{cov}\left(g(Z), e^{tg(Z)}\right) & \leq t\mathbb{E}\left[\left(F\left(Z, Z^{\prime}\right)\right)_{+}(g(Z)-g(Z^{\prime}))_+ e^{t g(Z)}\right]\\
   & = t\mathbb{E}\left[\EE [\left(F\left(Z, Z^{\prime}\right)\right)_{+}(g(Z)-g(Z^{\prime}))_+ \vert Z] e^{t g(Z)}\right] \;. 
\end{align*}
\end{proof}
\subsection{Decoupling lemma and entropy duality formula}\label{ssec_sm_decoupling}

We provide here a detailed proof of Lemma \ref{theorem duality entropy}, which essentially follows the lines of the proof given in \cite{saumard2019weighted}. We also prove that Lemma \ref{theorem duality entropy} is equivalent to the following duality formula for the entropy : for a non-negative random variable $Z$ such that $\EE[Z\log(Z)]<+\infty$, it holds
$${\rm Ent}(Z)=\EE\left[Z\log\left(Z\right)\right]-\EE[Z]\log(\EE[Z])=\sup_{U\, \text{s.t.}\, \EE[e^{U}]=1}\EE[ZU]\;.$$

\begin{proof}[Proof of Lemma \ref{theorem duality entropy}]
Note that the condition $\EE[Xe^X]<+\infty$ implies that $\EE[e^X]<+\infty$. Hence, if $\EE[e^Y]=+\infty$, then Lemma \ref{theorem duality entropy} holds since in this case $\EE[e^X]<\EE[e^Y]$.

Assume now that $\EE[e^Y]<+\infty$ and set $\beta :=\log(\EE[e^Y])$. As $0\leq \EE[Xe^X] \leq \EE[Ye^X]$, $Y$ is not almost surely equal to $-\infty$ and it holds $\EE[e^Y]>0$, that is $\beta \in \R$. Let us set $U=Y-\beta$. On the one hand, we have $\EE[e^U]=1$, so by the duality formula for the entropy, 
\[ \EE[e^XU]\leq \EE[Xe^X]-\EE[e^X]\log(\EE[e^X])\;.\]
On the other hand, as $\EE[Xe^X]\leq \EE[Ye^X]$, we also have
\[
\EE[e^XU]\geq \EE[Xe^X]-\beta\EE[e^X]\;.
\]
Combining the two inequalities for $\EE[e^XU]$, we get $\beta\leq -\log(\EE[e^X])$, which gives $\EE[e^X]\leq \EE[e^Y]$.

It remains to consider the case where we assume that $\EE[Xe^X]<\EE[Ye^X]$, together with $\EE[e^Y]<+\infty$. In that case, we have $\EE[e^XU]>\EE[Xe^X]-\beta\EE[e^X]\;,$
which implies $\beta < -\log(\EE[e^Y])$ and finally, $\EE[e^X]<\EE[e^Y]$.
\end{proof}

The proof of Lemma \ref{theorem duality entropy} is based on the duality formula for the entropy. Conversely, let us now prove that the duality formula for the entropy is a consequence of Lemma \ref{theorem duality entropy}. 

Consider a non-negative random variable $Z$ such that $\EE[Z\log(Z)]<+\infty$. If $\EE[Z]=0$, then $Z=0$ \textit{a.s.} and the duality formula holds. Now, assume that $\EE[Z]>0$ and set $V=\log(Z/\EE[Z])$. It holds $\EE[e^V]=1$, hence ${\rm Ent}(Z)=\EE[ZV]\leq \sup_{U\, \text{s.t.}\, \EE[e^{U}]=1}\EE[ZU]\;.$ Furthermore, we have the identity ${\rm Ent}(Z)/\EE[Z]=\EE[Ve^V]$, so if we assume that ${\rm Ent}(Z)<\EE[UZ]$ for some random variable $U$ such that $\EE[e^U]=1$, we also have $\EE[Ve^V]={\rm Ent}(Z)/\EE[Z]<\EE[UZ/\EE[Z]]=\EE[Ue^V]$, which gives by Lemma \ref{theorem duality entropy}, $1=\EE[e^V]<\EE[e^U]$. This means that $\sup_{U\, \text{s.t.}\, \EE[e^{U}]=1}\EE[ZU]\leq {\rm Ent}(Z)$. This concludes the proof of the duality formula for the entropy, that consisted in establishing both inequalities.

\subsection{Proofs of the result on concentration of the resampled empirical process conditioned on data}
\subsubsection{Existence of the antisymmetric function}
\begin{proposition} \label{prop_exist_F}
    For any strong stationary time $T$ 
    and any two permutations $\pi,\pi'$, we have that
    \[ \sum_{k = 0}^{+\infty} \left| \mathbb{E} \left[ g(\pi \circ \pi_k) - g(\pi' \circ \pi_k) \right] \right| \leq 2\| g \|_\infty \mathbb{E}[T]. \]
    In particular, the function
    \[ F(\pi,\pi') = \sum_{k = 0}^{+\infty} \mathbb{E} \left[ g(\pi \circ \pi_k) - g(\pi' \circ \pi_k) \right]  \]
    is well-defined. It is antisymmetric and such that
    \[ \mathbb{E}[F(\sigma,\sigma') | \sigma] = g(\sigma) - \mathbb{E}[g(\sigma)] = f(\sigma). \]
\end{proposition}

\begin{proof}
     Let $T$ be a strong stationary time. For any $k \geq 0$,
    \[ g(\pi \circ \pi_k) - g(\pi' \circ \pi_k) = \left( g(\pi \circ \pi_k) - g(\pi' \circ \pi_k) \right) \mathbb{I}\{ k < T\} + \left( g(\pi \circ \pi_k) - g(\pi' \circ \pi_k) \right) \mathbb{I}\{ k \geq T\}. \]
    By definition of a strong stationary time,
    \begin{align*}
       \mathbb{E}\left[ \left( g(\pi \circ \pi_k) - g(\pi' \circ \pi_k) \right) \mathbb{I}\{ k \geq T\} \right] &=  \mathbb{P}(k \geq T) \mathbb{E}\left[g(\pi \circ \pi_k) - g(\pi' \circ \pi_k) | k \geq T  \right] \\
       &= 0
    \end{align*}
    since both $\pi \circ \pi_k$ and $\pi' \circ \pi_k$ are uniformly distributed on $\mathfrak{S}_n$, knowing that $k \geq T$. It follows that
    \begin{align*}
        \left| \mathbb{E} \left[ g(\pi \circ \pi_k) - g(\pi' \circ \pi_k) \right] \right|
        &= \left| \mathbb{E} \left[ \left( g(\pi \circ \pi_k) - g(\pi' \circ \pi_k) \right) \mathbb{I}\{ k < T\} \right] \right| \\
        &\leq \mathbb{E}\left[ 2\| g \|_\infty \mathbb{I}\{ k < T\} \right] 
    \end{align*}
    This yields
    \[ \sum_{k = 0}^{+\infty} \left| \mathbb{E} \left[ g(\pi \circ \pi_k) - g(\pi' \circ \pi_k) \right] \right| \leq \mathbb{E}\left[ 2\| g \|_\infty T \right]. \]
    Thus, the series defining $F(\pi,\pi')$ converges, and is equal to
    \[ \sum_{k = 0}^{+\infty} \mathbb{E} \left[ \left( g(\pi \circ \pi_k) - g(\pi' \circ \pi_k) \right) \mathbb{I}\{ k < T\} \right] = \mathbb{E} \left[ \sum_{k = 0}^{+\infty} \left( g(\pi \circ \pi_k) - g(\pi' \circ \pi_k) \right) \mathbb{I}\{ k < T\} \right]. \]
    Antisymmetry of $F$ is obvious. To conclude, note that by Fubini's theorem,
    \begin{align*}
        \mathbb{E}[F(\sigma,\sigma') | \sigma] &= \sum_{k = 0}^{+\infty} \mathbb{E}[g(\sigma \circ \pi_k) - g(\sigma' \circ \pi_k) | \sigma] \\
        &= \sum_{k = 0}^{+\infty} \mathbb{E}[g(\sigma \circ \pi_k) - g(\sigma \circ \tau_{I,J} \circ \pi_k) | \sigma] \\
        &= \sum_{k = 0}^{+\infty} \mathbb{E}[g(\sigma \circ \pi_k) - g(\sigma \circ \pi_{k+1}) | \sigma].
    \end{align*}
    The existence of a strong stationary time implies in particular that the distribution of $\pi_k$ converges in total variation to the uniform distribution. Hence,
    \[ \lim_{k \to +\infty} \mathbb{E}[g(\sigma \circ \pi_k) | \sigma] = \mathbb{E}[g(\sigma)]. \]
    This finally yields
    \[ \mathbb{E}[F(\sigma,\sigma') | \sigma] = g(\sigma) - \mathbb{E}[g(\sigma)].  \]
\end{proof}

\subsubsection{Proof of Lemma \ref{prop_F_sigma}}\label{ssec_proof_lemma_F_sigma}
\begin{proof}
    Let $\pi$ be a permutation and $\tau_{i,j}$ be a non-trivial transposition.
    For any $k \geq 0$,
    \[ g(\pi \circ \pi_k) - g(\pi \tau_{i,j} \circ \pi_k) = \left( g(\pi \circ \pi_k) - g(\pi \circ \tau_{i,j} \pi_k) \right) \mathbb{I}\{ k < T_{i,j} \} + \left( g(\pi \circ \pi_k) - g(\pi \tau_{i,j} \circ \pi_k) \right) \mathbb{I}\{ k \geq T_{i,j} \}. \]
    By definition of a strong convergence time,
    \begin{align*}
       \mathbb{E}\left[ \left( g(\pi \circ \pi_k) - g(\pi \circ \tau_{i,j} \pi_k) \right) \mathbb{I}\{ k \geq T_{i,j} \} \right] &=  \mathbb{P}(k \geq T_{i,j}) \mathbb{E}\left[g(\pi \circ \pi_k) - g(\pi \circ \tau_{i,j} \pi_k) | k \geq T_{i,j}  \right] \\
       &= 0
    \end{align*}
    since both $\pi \circ \pi_k$ and $\pi' \circ \tau_{i,j} \pi_k$ are identically distributed on $\mathfrak{S}_n$, knowing that $k \geq T_{i,j}$. It follows that
    \[ \mathbb{E} \left[ \sum_{k = 0}^{+\infty} \left( g(\pi \circ \pi_k) - g(\pi \tau_{i,j} \circ \pi_k) \right) \right] = \mathbb{E} \left[ \sum_{k = 0}^{+\infty} \left( g(\pi \circ \pi_k) - g(\pi \tau_{i,j} \circ \pi_k) \right) \mathbb{I}\{ k < T_{i,j} \} \right]. \]
    Let $\sigma' = \sigma \tau_{I,J},$ then knowing $\sigma, \sigma' $ is equivalent to knowing $\sigma, I,J$. These variables are independent from $(\pi_k)_{k \geq 1}$ so that
    \begin{align*}
      F(\sigma, \sigma') &= \mathbb{E} \left[ \sum_{k = 0}^{+\infty} \left( g(\sigma \circ \pi_k) - g(\sigma \tau_{I,J} \circ \pi_k) \right) | \sigma, I, J \right]  \\
      &= \mathbb{E} \left[ \sum_{k = 0}^{+\infty} \left( g(\sigma \circ \pi_k) - g(\sigma \tau_{I,J} \circ \pi_k) \right) \mathbb{I}\{ k < T_{I,J} \} | \sigma, I, J \right]
    \end{align*}
    which yields the lemma.
\end{proof}

\subsubsection{Proof of Theorem \ref{thm_conc_fun_permut}}\label{subsubsection_proof_theorem_concen}
    Let us now analyze the concentration of $f(\sigma)$. Fix some strong stationary time $T$. Let $\psi$ be the moment generating function of $f(\sigma) = g(\sigma) - \mathbb{E}[g(\sigma)]$. By Lemma \ref{lem_cov_ineq},
\begin{align}
   \psi'(\theta)=\EE[f(\sigma)e^{\theta f(\sigma)}] &\leq \theta \mathbb{E}\left[ (f(\sigma) - f(\sigma'))_+ (F(\sigma,\sigma'))_+ e^{\theta f(\sigma)} \right] \nonumber\\
   &= \theta \mathbb{E}\left[\EE[(g(\sigma) - g(\sigma'))_+ (F(\sigma,\sigma'))_+\vert \sigma]  e^{\theta f(\sigma)} \right]\;. \label{eq_Vplus_sigma}
\end{align}
By Lemma \ref{prop_F_sigma} and since $\sigma' = \sigma$ if $I = J$, 
\begin{align*}
    &(g(\sigma) - g(\sigma'))_+ (F(\sigma,\sigma'))_+ \\ \leq & \mathbb{I}\{ I \neq J \} \sum_{k = 0}^{+ \infty} \mathbb{E} \left[ (g(\sigma) - g(\sigma'))_+ (g(\sigma \circ \pi_k) - g(\sigma' \circ \pi_k))_+ \mathbb{I}\{ k < T_{I,J} \} \bigl| \sigma,I,J \right]\;. 
\end{align*}
Let $p_k = \mathbb{P}(k < T_{1,2})$. By the Cauchy-Schwarz inequality,
\begin{align*}
    &(g(\sigma) - g(\sigma'))_+ (F(\sigma,\sigma'))_+  \\ 
    &\quad \leq \sum_{k = 0}^{+ \infty} \sqrt{p_k} (g(\sigma) - g(\sigma'))_+ \mathbb{E} \left[(g(\sigma \circ \pi_k) - g(\sigma' \circ \pi_k))_+^2 \bigl| \sigma,I,J \right]^{1/2} \\
    &\quad \leq \sum_{k = 0}^{+\infty} \frac{\sqrt{p_k}}{2} \left( (g(\sigma) - g(\sigma'))_+^2 + \mathbb{E} \left[(g(\sigma \circ \pi_k) - g(\sigma' \circ \pi_k))_+^2 \bigl| \sigma,I,J \right] \right)\;.
\end{align*}
Define the random variable
\[ W = \sum_{k = 0}^{+\infty} \frac{\sqrt{p_k}}{2} \mathbb{E} \left[ (g(\sigma) - g(\sigma \tau_{I,J}))_+^2 + (g(\sigma \circ \pi_k) - g(\sigma \tau_{I,J} \circ \pi_k))_+^2 \bigl| \sigma \right]\;. \]
It follows from equation \eqref{eq_Vplus_sigma} that
\[ \psi'(\theta) \leq \theta \mathbb{E}[ W e^{\theta f(\sigma)} ] \]
for any $\theta \geq 0$. Lemma \ref{theorem duality entropy} implies therefore that
\[ \psi(\theta) \leq \mathbb{E}[\exp(\theta^2 W)]\;. \]
Let $r = \sum_{k = 0}^{+\infty} \sqrt{p_k}$ (assuming it is finite). Then
\[ W = \sum_{k = 0}^{+\infty} \frac{\sqrt{p_k}}{2r} \left( r\mathbb{E}[(g(\sigma) - g(\sigma'))_+^2|\sigma] + r\mathbb{E} \left[(g(\sigma \circ \pi_k) - g(\sigma' \circ \pi_k))_+^2 \bigl| \sigma \right] \right)\;. \]
By convexity of the exponential function,
\begin{align*}
   &\exp(\theta^2 W) \leq  \frac{1}{2} \exp \left(r\theta^2 \mathbb{E}[(g(\sigma) - g(\sigma'))_+^2|\sigma]  \right)+\ldots\\
    \ldots + & \frac{1}{2} \sum_{k = 0}^{+\infty} \frac{\sqrt{p_k}}{r} \exp \left(r \theta^2 \mathbb{E} \left[(g(\sigma \circ \pi_k) - g(\sigma' \circ \pi_k))_+^2 \bigl| \sigma \right] \right)\;.
\end{align*}
Moreover, by Jensen's inequality,
\begin{align*}
    & \exp \left(r \theta^2 \mathbb{E} \left[(g(\sigma \circ \pi_k) - g(\sigma' \circ \pi_k))_+^2 \bigl| \sigma \right] \right) \\ \leq & \mathbb{E}\left[ \exp \left(r \theta^2 \mathbb{E} \left[(g(\sigma \circ \pi_k) - g(\sigma' \circ \pi_k))_+^2 \bigl| \sigma, \pi_k \right] \right)  | \sigma \right]\;,
\end{align*}
which implies that
\begin{align}
    &\mathbb{E} \left[ \exp(\theta^2 W)  \right] \leq \frac{1}{2} \mathbb{E} \left[ \exp \left(r\theta^2 \mathbb{E}[(g(\sigma) - g(\sigma'))_+^2|\sigma]  \right) \right] \nonumber \\ 
&\quad + \frac{1}{2} \sum_{k = 0}^{+\infty} \frac{\sqrt{p_k}}{r} \mathbb{E}\left[ \exp \left(r \theta^2 \mathbb{E} \left[(g(\sigma \circ \pi_k) - g(\sigma' \circ \pi_k))_+^2 \bigl| \sigma, \pi_k \right] \right) \right]\;. \label{eq_Jensen_permut}
\end{align}
Now, remark that for any permutation $\pi$ and any transposition $\tau = \tau_{i,j}$,
\[ \tau_{i,j} \circ \pi = \pi \circ \tau_{\pi^{-1}(i),\pi^{-1}(j)}\;.  \]
In particular,
\[ \sigma' \circ \pi_k = \sigma \circ \tau_{I,J} \circ \pi_k = (\sigma \circ \pi_k) \circ \tau_{\pi_k^{-1}(I),\pi_k^{-1}(J)}\;. \]
Conditional on $\sigma,\pi_k$, the pair $\pi_k^{-1}(I),\pi_k^{-1}(J)$ follows the same distribution as $I,J$. Hence
\[ \mathbb{E} \left[(g(\sigma \circ \pi_k) - g(\sigma' \circ \pi_k))_+^2 \bigl| \sigma, \pi_k \right] = \mathbb{E} \left[(g(\sigma \circ \pi_k) - g(\sigma \circ \pi_k \circ \tau_{I,J}))_+^2 \bigl| \sigma, \pi_k \right]\;. \]
By independence of $\pi_k$ from $\sigma$, $\sigma \circ \pi_k$ is also a uniform random permutation and is independent from $\tau_{I,J},$ which means that the pair $(\sigma \circ \pi_k, \sigma \circ \pi_k \circ \tau_{I,J})$ has the same distribution as the pair $(\sigma, \sigma')$. In particular,
\[ \mathbb{E}\left[ \exp \left(r \theta^2 \mathbb{E} \left[(g(\sigma \circ \pi_k) - g(\sigma' \circ \pi_k))_+^2 \bigl| \sigma, \pi_k \right] \right) \right] = \mathbb{E} \left[ \exp \left(r\theta^2 \mathbb{E}[(g(\sigma) - g(\sigma'))_+^2|\sigma]  \right) \right]. \]
Injecting this upper bound into equation \eqref{eq_Jensen_permut} yields
\[ \mathbb{E} \left[ \exp(\theta^2 W)  \right] \leq \mathbb{E} \left[ \exp \left(r\theta^2 \mathbb{E}[(g(\sigma) - g(\sigma'))_+^2|\sigma]  \right) \right] \]
and it follows that
\[ \psi(\theta) \leq \mathbb{E} \left[ \exp \left(r\theta^2 \mathbb{E}[(g(\sigma) - g(\sigma'))_+^2|\sigma]  \right) \right]  \]
for any $\theta \geq 0$. Now, conclude by observing that
\begin{align*}
    \mathbb{E}[(g(\sigma) - g(\sigma'))_+^2|\sigma] &=  \mathbb{E}[(g(\sigma) - g(\sigma \tau_{I,J}))_+^2| \sigma ] \\
    &= \frac{2(1-\alpha_0)}{n(n-1)} \sum_{i = 1}^{n-1} \sum_{j = i+1}^n (g(\sigma) - g(\sigma \tau_{i,j}))_+^2 \\
    &= \frac{1-\alpha_0}{n-1} V_+(g,\sigma)\;.
\end{align*}

\subsubsection{Proof of Proposition \ref{prop_strong_stat_time}} \label{proof_strong_stat_time} 
 Let $T = T_{1,2}$ for short. 
    Define the following times
    \begin{align*}
        T_1 &= \min \left\{t \geq 0 : |P_m(t)| \geq \frac{n}{3}  \right\} \\
        T_2 &= \min \left\{t \geq 0 : |P_m(T_1 + t)| > \frac{n}{2}  \right\} \\
        T_3 &= \min \left\{t \geq 0 : \{i,j\} \cap P_m(T_1 + T_2 + t) \neq \emptyset   \right\} \\
        T_4 &= \min \left\{t \geq 0 : \{i,j\} \subset P_m(T_1 + T_2 + T_3 + t) \neq \emptyset   \right\} \\
        &= T - T_1 - T_2 - T_3.
    \end{align*}
    By construction, $T = T_1 + T_2 + T_3 + T_4$. $(T_i)_{1 \leq i \leq 4}$ are stopping times; let $(\cF_i)_{1 \leq i \leq 4}$ be the corresponding stopped $\sigma$-algebras. Let
    \[ s = \min \left(1.01, 1 + \frac{1}{3n} \right). \]
    The proof of \cite[Proposition 13]{White2019} shows that $T_1$ is stochastically dominated by $S_{n-1}$, the hitting time of $n-1$ by a biased random walk $(X_t)_{t \geq 0}$ started at $X_0 = 0$ and with increment distribution
    \[ \mathbb{P}(X_1 = 1) = \frac{1}{3}, \mathbb{P}(X_1 = - 1) = \frac{1}{6}, \mathbb{P}(X_1 = 0) = \frac{1}{2}. \]
    By the strong Markov property and translation equivariance, the distribution of $S_{n-1}$ is that of an i.i.d. sum of $n-1$ variables distributed as $S_1$, the hitting time of $1$. Thus,
    \[ \mathbb{E}[s^{T_1}] \leq \mathbb{E}[s^{S_1}]^{n-1} = G_1(s)^{n-1} \]
    where
    \[ G_1(s) = \mathbb{E}[s^{S_1}]. \]
    Now, conditioning on the first step of the walk, we see that
    \begin{itemize}
        \item If $X_1 = 1$ then $S_1 = 1$ by definition
        \item If $X_1 = 0$, we have wasted one step, the conditional distribution of $S_1$ is the (unconditional) distribution of $S_1 + 1$.
        \item If $X_1 = -1$, two upward steps need to be taken, so the conditional distribution of $S_1$ is that of $1 + S_1 + S_1^*$, where $S_1^*$ is an i.i.d. copy of $S_1$.
    \end{itemize}
    This leads to the equation
    \[ G_1(s) = \frac{s}{3} + \frac{s}{2} G_1(s) + \frac{s}{6} G_1(s)^2 \]
    \textit{i.e.}
    \[  \frac{s}{6} G_1(s)^2 - \left(1 - \frac{s}{2} \right) G_1(s) + \frac{s}{3} = 0.  \]
    This quadratic equation has two solutions, only one of which belongs to $[0,1]$ when $s \in [0,1]$. This yields
    \begin{equation} \label{eq_calc_G1}
       G_1(s) = \frac{3}{s} \left(1 - \frac{s}{2} - \sqrt{1 - s + \frac{s^2}{36}} \right) 
    \end{equation}
    at least for $s\in [0,1]$.
    
    The right-hand side can be analytically continued to the open disk $B(0,s_*)$, where \[ s_* = \frac{1}{\frac{1}{2} + \frac{\sqrt{2}}{3}} > 1.02. \] 
    To see this, remark that for any complex $z$ such that $|z| < s_*$, 
    \begin{align*}
        \mathrm{Re}\left(1 - z + \frac{z^2}{36} \right) &= 1 - \mathrm{Re}(z) + \frac{\mathrm{Re}(z)^2 - \mathrm{Im}(z)^2}{36} \\
        &= 1 - \mathrm{Re}(z) + \frac{2\mathrm{Re}(z)^2 - |z|^2}{36} \\
        &\geq 1 - \mathrm{Re}(z) + \frac{\mathrm{Re}(z)^2}{18} - \frac{(s_*)^2}{36} \\
        &> 1 - s_* + \frac{(s_*)^2}{36} \text{ since } |\mathrm{Re}(z)| \leq |z| < s_* \\
        &> 0,
    \end{align*}
   so that the function
   \[ \psi: z \mapsto \frac{3}{z} \left(1 - \frac{z}{2} - \sqrt{1 - z + \frac{z^2}{36}} \right) \]
   is well defined and analytic on $B(0,s_*) \backslash \{0\}$, where $\sqrt{\cdot}$ denotes the analytic continuation of the ordinary square root function to the open half-plane $\{z : \mathrm{Re}(z) > 0\}$. $0$ is a removable singularity as can easily be checked using Taylor expansion. Since $\psi$ is analytic on $B(0,s_*)$, the radius of convergence of its Taylor series at $0$ is at least $s_*$. By equation \eqref{eq_calc_G1}, the Taylor series of $\psi$ at $0$ is 
   \[ \sum_{k = 0}^{+\infty} z^k \mathbb{P}(S_1 = k) = G_1(z), \]
   which proves that $G_1$ is well-defined and equal to $\psi$ on $B(0,s_*)$ and in particular on $(0,s_*)$.

    Since $S_1$ is integer-valued, $G_1$ is convex on $[0;s_*)$, in particular
    \[ G_1(s) \leq 1 + \frac{G_1(1.01)-1}{0.01}(s-1) \leq 1 + 6.81(s-1) \leq 1 + \frac{6.81}{3n} \]
    and hence
    \[ \mathbb{E}[s^{T_1}] \leq \left(1 + \frac{6.81}{3n} \right)^{n-1} \leq \exp \left( \frac{6.81}{3} \right) \leq 9.68. \]
    Consider now $T_2$. As explained in the proof of \cite[Proposition 14]{White2019}, at each time step $t$ such that $\frac{n}{3} \leq k = |P_m(t)| \leq \frac{n}{2}$, there is a probability at least
    \[ \frac{k(n-k)}{n(n-1)} \geq \frac{1}{3} \frac{2}{3} = \frac{2}{9} \]
    of increasing the size of $P_m(t)$ by (at least) one, conditional on the past. Thus, conditionally on the past up to time $T_1$, the distribution of $T_2$ is dominated by the sum of $\frac{n}{2} - \frac{n}{3} + 1 = \frac{n}{6}+1$ independent geometric variables with parameter $\frac{2}{9}$, so that
    \[ \mathbb{E}\left[s^{T_2} | \cF_1 \right] \leq \left(\frac{\frac{2}{9}s}{1 - \left(1 - \frac{2}{9} \right)s} \right)^{\frac{n}{6}+1} = \left( \frac{2s}{9 - 7s} \right)^{\frac{n}{6}+1} \leq \frac{2\times 1.01}{9 - 7\times1.01} \left( \frac{2s}{9 - 7s} \right)^{\frac{n}{6}} \]
    which is well-defined and finite for all $s < \frac{9}{7}$. By convexity of the function $u \mapsto \frac{2u}{9-7u}$, we have that
    \[ \frac{2s}{9 - 7s} \leq 1 + \left(\frac{2\times 1.01}{9 - 7\times 1.01} - 1 \right) \frac{s-1}{0.01} \leq 1 + \frac{4.67}{3n} \]
    which yields
    \[ \mathbb{E}\left[s^{T_2} | \cF_1 \right] \leq 1.047 \left( 1 + \frac{4.67}{3n} \right)^{\frac{n}{6}} \leq 1.047 \exp \left( \frac{4.67}{18} \right) \leq 1.36. \]
    Consider now $T_3$ and $T_4$. Once $t \geq T_1 + T_2$, that is, once $|P_m(t)|$ becomes strictly larger than $\frac{n}{2}$, \cite[Marking scheme D]{White2019} is such that the block that was maximal at time $T_1 + T_2$ continues to grow and hence remains maximal at all later times. 
    
    At each time step $t \geq T_1 + T_2$, assuming $i,j \notin P_m(t)$, \cite[Marking scheme D]{White2019} yields $i \in P_m(t+1)$ or $j \in P_m(t+1)$ as soon as $(I_{t+1},J_{t+1}) \in \{i,j\} \times P_m(t) \cup P_m(t) \times \{i,j\}$, which happens with probability
    \[ \frac{2|P_m(t)|}{n(n-1)} \geq \frac{1}{n-1}. \]
    Thus, conditionally on the past up to time $T_1 + T_2,$ $T_3$ is stochastically dominated by a geometric distribution with parameter $\frac{1}{n-1}$, which yields
    \[ \mathbb{E}\left[ s^{T_3} | \cF_2 \right] \leq \frac{\frac{s}{n-1}}{1 - \frac{n-2}{n-1}s}. \]
    Now, $s \leq \frac{3n+1}{3n}$ hence 
    \begin{align*}
        \frac{n-2}{n-1} s &\leq \frac{3n^2 - 5n - 2}{3n^2 - 3n} \\ 
        &\leq 1 - \frac{2}{3(n-1)} \\
        1 - \frac{n-2}{n-1} \frac{3n+1}{3n} &\geq \frac{2}{3(n-1)} \text{hence} \\
        \frac{\frac{s}{n-1}}{1 - \frac{n-2}{n-1}s} &\leq \frac{3}{2}s \\
       &\leq 1.5 \times 1.01 \\
       &= 1.515.
    \end{align*}
    Similarly, conditionally on the past up to time $T_1 + T_2 + T_3,$ $T_4$ is stochastically dominated by a geometric distribution with parameter $\frac{1}{2(n-1)}$, so that
    \[ \mathbb{E}\left[ s^{T_4} | \cF_3 \right] \leq \frac{\frac{s}{2(n-1)}}{1 - \frac{2n-3}{2n-2}s}. \]
   Since $s \leq \frac{3n+1}{3n}$,
   \begin{align*}
        \frac{2n-3}{2n-2} s &\leq \frac{6n^2 - 7n - 3}{6n^2 - 6n} \\ 
        &\leq 1 - \frac{1}{6(n-1)} \\
        \frac{\frac{s}{2(n-1)}}{1 - \frac{2n-3}{2n-2}s} &\leq 3s \\
        &\leq 3\times 1.01
   \end{align*}
   It follows that 
   \[ \mathbb{E}\left[ s^{T_3} | \cF_2 \right] \leq 1.515, \quad \mathbb{E}\left[ s^{T_4} | \cF_3 \right] \leq 3.03. \]
   This finally yields
   \[ \mathbb{E}[s^T] = \mathbb{E} \left[ s^{T_1} \mathbb{E}\left[ s^{T_2} | \cF_1 \right] \mathbb{E}\left[ s^{T_3} | \cF_2 \right] \mathbb{E}\left[ s^{T_4} | \cF_3 \right]  \right] \leq 9.68\times 1.36 \times 1.515 \times 3.03 
   \leq 60.5 
   \]
   Now, for any $t \geq 1$, by Markov's inequality,
   \[ \mathbb{P} \left( T \geq t \right) = \mathbb{P} \left( s^T \geq s^t \right) \leq \frac{\mathbb{E}\left[ s^T \right]}{s^t} \leq \frac{60.5}{s^t}. \]
   Let
   \[ t_* = \inf \left\{t > 0 : s^t \geq 60.5 \right\} = \frac{\log(60.5)}{\log(s)}, \]
   it follows that
   \begin{align*}
       r(T) &= \sum_{t = 0}^{+\infty} \sqrt{\mathbb{P}(T > t)} \\
       &= \sum_{t = 1}^{+\infty} \sqrt{\mathbb{P}(T \geq t)} \\
       &\leq \lfloor t_* \rfloor + \sum_{t = \lfloor t_* \rfloor + 1}^{+\infty} \sqrt{\frac{60.5}{s^t}} \\
       &\leq \lfloor t_* \rfloor + \sum_{t = 0}^{+\infty} s^{- \frac{t}{2}} \\
       &\leq t_* + \left(1 - \frac{1}{\sqrt{s}} \right)^{-1}.
   \end{align*}
   Since $s = \min \left(1.01, 1 + \frac{1}{3n} \right)$, by convexity of $u \mapsto \frac{1}{\sqrt{u}}$,
   \[ \frac{1}{\sqrt{s}} \leq 1 - 100*\left(1 - \frac{1}{\sqrt{1.01}} \right) (s-1) \leq 1 - 0.496 \min \left(0.01, \frac{1}{3n} \right) \]
   and thus
   \[ \left(1 - \frac{1}{\sqrt{s}} \right)^{-1} \leq \frac{\max(100,3n)}{0.496} \leq \max(202, 6.05 n). \]
   Moreover, by concavity of the logarithm,
   \[ \log(s) \geq 100\times \log(1.01) \min \left(0.01, \frac{1}{3n} \right) \geq 0.995 \min \left(0.01, \frac{1}{3n} \right) \]
   and so
   \[ t_* \leq \frac{\log(60.5)}{0.995} \max(100,3n) \leq \max\left(413, 12.37 n \right). \]
   This finally yields
   \[ r(T) \leq \max \left(615, 18.42 n \right) \]
   as claimed.

\subsubsection{Proof of Proposition \ref{prop_bound_Vplus_perm}}\label{ssec_proof_prop_bound_Vplus_perm}
Consider now the specific case of the function 
\[ g_x: \sigma \mapsto \sup_{t \in \cT} \left\{ \sum_{i = 1}^n w_{\sigma(i)} t(x_i) \right\}. \]
Define the set 
\[ \mathcal{K}(x) = \overline{\left\{ (t(x_i))_{1 \leq i \leq n} : t \in \cT \right\}} \subset \mathbb{R}^n \]
which is closed and bounded (by assumption) and such that
\[ g_x(\sigma) = \sup_{z \in \mathcal{K}(x)} \sum_{i = 1}^n w_{\sigma(i)} z_i\;. \]
As $\mathcal{K}(x)$ is compact, we can choose for each $\pi \in \mathfrak{S}_n$, 
\[ \overline{z}(\pi) \in \underset{z \in \mathcal{K}(x)}{\mathrm{argmax}} \sum_{i = 1}^n w_{\pi(i)} z_i \;. \]
Since $\mathfrak{S}_n$ is finite, this does not raise measurability issues. Let $\hat{z}$ denote the random variable $\overline{z}(\sigma)$. Let $I,J$ be independent and uniformly distributed on $\{1,\ldots,n\}$, then
\begin{align*}
    g_x \left( \sigma \right) - g_x(\sigma \circ \tau_{I,J}) &\leq \sum_{i = 1}^n w_{\sigma(i)} \hat{z}_i - \sum_{i = 1}^n w_{(\sigma \circ \tau_{I,J})(i)} \hat{z}_i  \\
    &= w_{\sigma(I)} \hat{z}_{I} + w_{\sigma(J)} \hat{z}_J - w_{\sigma(J)} \hat{z}_I - w_{\sigma(I)} \hat{z}_J \\
    &= \left(w_{\sigma(I)} - w_{\sigma(J)} \right) \left(\hat{z}_I - \hat{z}_J \right).
\end{align*}
It follows that
\[ \left( g_x \left( \sigma \right) - g_x(\sigma \circ \tau_{I,J}) \right)_+^2 \leq \left[ \left(w_{\sigma(I)} - w_{\sigma(J)} \right) \left(\hat{z}_I - \hat{z}_J \right) \right]_+^2 \leq \left(w_{\sigma(I)} - w_{\sigma(J)} \right)^2 \left(\hat{z}_I - \hat{z}_J \right)^2. \]
As a consequence, since $\hat{z} =\overline{z}(\sigma)$ depends on $\sigma$ but not on $I,J$ and since  $V_+(g_x,\sigma)=n\EE[ \left( g_x \left( \sigma \right) - g_x(\sigma \circ \tau_{I,J}) \right)_+^2\vert \sigma]$, we have
\[ V_+(g_x,\sigma) \leq n\sup_{t \in \cT} \left\{ \mathbb{E} \left[ \left(w_{\sigma(I)} - w_{\sigma(J)} \right)^2 \left(t(x_I) - t(x_J) \right)^2 \bigl| \sigma \right] \right\} := W\;. \]
If the quantities $t(x_i)$ all belong to $[-1,1]$, then
\[ V_+(g_x,\sigma) \leq W \leq 4n \mathbb{E} \left[ \left(w_{\sigma(I)} - w_{\sigma(J)} \right)^2 | \sigma \right] = 8 \sum_{i = 1}^n w_i^2\;. \]
In general, since $w_i \in [a;b]$ for all $i\in \left\{ 1,\ldots, n\right\}$ by definition, 
\begin{align*}
    V_+(g_x,\sigma) &\leq n(b-a)^2 \sup_{t \in \cT} \left\{ \mathbb{E} \left[ \left(t(x_I) - t(x_J) \right)^2 \bigl| \sigma \right] \right\} \\
    &= n(b-a)^2 \sup_{t \in \cT} \left\{ \frac{1}{n^2} \sum_{i = 1}^n \sum_{j = 1}^n (t(x_i) - t(x_j))^2 \right\} \\
    &=  2(b-a)^2 \sup_{t \in \cT} \left\{ \sum_{i = 1}^n (t(x_i) - \overline{t}_{x})^2 \right\}\;,
\end{align*}
where $\overline{t}_{x}$ denotes the empirical mean,
\[ \overline{t}_{x} = \frac{1}{n} \sum_{i = 1}^n t(x_i)\;. \]
This concludes the proof of Proposition \ref{prop_bound_Vplus_perm}.

\subsection{A version of Tolstikhin's theorem using the method of exchangeable pairs}
\label{proof_Tolstikhin_exch_pair}
Let us recall first Tolstikhin's concentration inequality (\cite{Tolstikhin2017}, see also \cite{tolstikhin2014localized}), which is a clever modification of a previous result by Bobkov (\cite{Bobkov2004}).

The result describes the concentration of some functions defined on the symmetric group $\mathfrak{S}_n$, $n\geq 3$, and satisfying some invariance conditions. More precisely, let us state the definition of $(k,n)$-symmetric functions.

\begin{definition} Let $n\geq 3$ be a positive integer and let $g$ be a function defined on $\mathfrak{S}_n$. Take a positive integer $k$ such that $k< n$ and let $\mathfrak{S}_{k,n-k}$ be the subgroup of permutations that leave the set $\{1,\ldots,k\}$ invariant. The function $g$ is said to be $(k,n)$-symmetric if and only if $g(\sigma \circ \pi) = g(\sigma)$ for all $\sigma \in \mathfrak{S}_n$ and $\pi \in \mathfrak{S}_{k,n-k}$.
\end{definition}

Tostikhin's result can expressed as follows.

\begin{theorem}[Tolstikhin \cite{Tolstikhin2017}, Bobkov \cite{Bobkov2004}]\label{Th_Tolsti}
    Let $n\geq 3$, $\sigma\in \mathfrak{S}_n$ a uniform ramdom permutation and $g:\mathfrak{S}_n \rightarrow \R$, a $(k,n)$-symmetric function. Denote 
    $$\vert \nabla g(\sigma) \vert^2_+ = \sum_{i=1}^k\sum_{j=k+1}^n(g(\sigma)-g(\sigma \circ \tau_{i,j}))_+^2$$ and assume that there exists a positive constant $\Sigma^2$ such that, almost surely, $\vert \nabla g(\sigma) \vert^2_+ \leq \Sigma^2$. Then, for any $t>0$ it holds
    \begin{equation}\label{ineq_Tolsti}
        \PP\left(g(\sigma)-\EE[g(\sigma)]\geq t\right) \leq \exp\left(-\frac{(n+2)t^2}{8\Sigma^2}\right)\;.
    \end{equation}
\end{theorem}

The point of this section is to prove a version of Tolstikhin's result using the method of exchangeable pairs. More precisely, we will prove the following theorem.

\begin{theorem}\label{th_tol_exch}
   Grant the assumptions and notations of Theorem \ref{Th_Tolsti}. Then, for all $\theta \geq 0$,
 \[ \mathbb{E}\left[ \exp(\theta (g(\sigma)-\EE[g(\sigma)])) \right] \leq \mathbb{E} \left[ \exp \left( \frac{2}{n} \frac{2n-2}{2n-5}  \theta^2 \left| \nabla g(\sigma) \right|_+^2 \right) \right] \;.\]
 If we have $\vert \nabla g(\sigma) \vert^2_+ \leq \Sigma^2$ $a.s.$, then, for any $t>0$,
    \begin{equation}\label{ineq_Tolsti_mod}
        \PP\left(g(\sigma)-\EE[g(\sigma)]\geq t\right) \leq \exp\left(-\frac{2n-5}{2n-2}\frac{nt^2}{8\Sigma^2}\right)\;.
    \end{equation}
\end{theorem}

\begin{proof}
Let $(I,J),(I_j,J_j)_{j \in \mathbb{N}}$ be i.i.d. pairs uniformly distributed on $\{1,\ldots,n\}^2$ and define a sequence $(\pi_t)_{t\in \NN}$ of random permutations such that: $\pi_0=Id$ and, for any $t\geq 1$,
\[ \pi_t = \tau_{I_1,J_1} \circ \tau_{I_2,J_2} \circ ... \circ \tau_{I_t,J_t}\;. \]
For $\sigma \in \mathfrak{S}_n$, we set $\sigma_t = \sigma \circ \pi_t$. For $\sigma,\sigma'\in \mathfrak{S}_n$, define the random time
\[ T_{\sigma,\sigma'} = \min \left\{ t \in \mathbb{N} : \pi_t^{-1} \sigma^{-1} \sigma' \pi_t \in \mathfrak{S}_{k,n-k} \right\}. \]
Remark that $T$ is a symmetric function of $(\sigma,\sigma')$.
Since $\pi_t$ converges exponentially fast to the uniform distribution on $\mathfrak{S}_n$ in total variation distance, $T$ is almost surely finite and of finite expectation, provided that $\sigma^{-1} \sigma'$ is of the form $\pi \tau \pi^{-1}$ for some $\tau \in \mathfrak{S}_{k,n-k}$, \textit{i.e.} $(\sigma,\sigma') \in \Delta$ where
\[ \Delta = \left\{ (\sigma_1,\sigma_2) : \sigma_1^{-1} \sigma_2 \in H_n \right\} \text{ with } H_n = \left\{ \pi \tau \pi^{-1} : \pi \in \mathfrak{S}_n, \tau \in \mathfrak{S}_{k,n-k} \right\}.  \]
Note that $\Delta$ is a symmetric set, \textit{i.e.} $(\sigma,\sigma') \in \Delta \iff (\sigma',\sigma) \in \Delta$, and that it contains the pairs of the form $(\sigma, \sigma \circ \tau_{i,j})$ for any $\sigma \in \mathfrak{S}_n$ and $(i,j) \in \{1,\ldots,n\}^2$. To see this, remark that
\[ \pi^{-1} (\sigma \tau_{i,j})^{-1} \sigma \pi = \pi^{-1} \tau_{i,j} \pi = \tau_{\pi^{-1}(i), \pi^{-1}(j)}\;, \]
which belongs to $\mathfrak{S}_{k,n-k}$ as long as $(\pi^{-1}(i), \pi^{-1}(j)) \in \{1,\ldots,k\}^2 \cup \{k+1,\ldots,n\}^2$.
For $n > 2$, such a permutation $\pi$ exists for any $i\neq j$.
For $(\sigma,\sigma')\in \mathfrak{S}^2_n$, let us now define $\sigma_t'$ as follows:
\[ \sigma_t' = \begin{cases}
                &\sigma' \circ \pi_t \text{ if } t \leq T\;, \\
                &\sigma_t \pi_T^{-1} \sigma^{-1} \sigma' \pi_T \text{ if } t > T\;.
               \end{cases} 
\]
Note that the sequence $(\sigma'_t)_{t\geq 0}$ depends on $\sigma$ and $T$, but we do not explicitely mention these quantities, that are clear from context, in the aim to lighten the notation. 

\noindent Define a function $F: \Delta \to \mathbb{R}$ by
\[ F(\sigma,\sigma') = \mathbb{E} \left[ \sum_{t = 0}^{+\infty} g(\sigma_t) - g(\sigma_t') | \sigma,\sigma' \right]. \]
We need to prove that
\begin{itemize}
    \item $F$ is antisymmetric on $\Delta$,
    \item $\mathbb{E}[F(\sigma,\sigma') | \sigma] = g(\sigma) - \mathbb{E}[g(\sigma)] = f(\sigma)$ where $\sigma$ is uniformly distributed on $\mathfrak{S}_n$ and $\sigma' = \sigma \circ \tau_{I,J}$.
\end{itemize}
The function $F$ is antisymmetric since $T_{\sigma',\sigma} = T_{\sigma,\sigma'} = T$ and by symmetry of $g$,
\begin{align*}
    \sum_{t = 0}^{+\infty} g(\sigma_t) - g(\sigma'_t) &= \sum_{t = 0}^{T-1} g(\sigma \pi_t) - g(\sigma' \pi_t) \\
    &= - \sum_{t = 0}^{T-1} g(\sigma' \pi_t) - g(\sigma \pi_t)\;,
\end{align*}
which yields $F(\sigma,\sigma') = - F(\sigma',\sigma)$ by taking expectations conditional on $(\sigma,\sigma')$. To prove the second property, we will show that for any $(\sigma,\sigma')\in \mathfrak{S}_n^2$, $(\sigma_t)_{t \geq 0}$ and $(\sigma_t')_{t \geq 0}$ are equal in distribution to a right transposition walk started respectively at $\sigma$ and $\sigma'$. For the sequence $(\sigma_t)_{t \geq 0}$, this is clearly true by definition. Let us show that $(\sigma_t')_{t \geq 0}$ is also a right transposition walk. If $0 < t \leq T$, then
\[ \sigma_t' = \sigma' \pi_t = \sigma'  \pi_{t-1} \tau_{I_t,J_t} = \sigma_{t-1}' \tau_{I_t,J_t} \]
and if $t \geq T+1$, then
\[ \sigma_t' = \sigma_t \eta = \sigma_{t-1} \tau_{I_t,J_t} \eta = \sigma_{t-1} \eta \tau_{\eta^{-1}(I_t), \eta^{-1}(J_t)} = \sigma_{t-1}' \tau_{\eta^{-1}(I_t), \eta^{-1}(J_t)} \;, \]
where $\eta = \pi_T^{-1} \sigma^{-1} \sigma' \pi_T$. Thus, the sequence $(\sigma_t')_{t \in \mathbb{N}}$ satisfies the recursive formula
\[ \sigma_t' = \sigma_{t-1}' \tau_{I_t',J_t'} \]
for all $t \geq 1$, where
\[ (I_t',J_t') =  (I_t,J_t) \mathbb{I}\{t-1 < T\} + (\eta^{-1}(I_t), \eta^{-1}(J_t)) \mathbb{I}\{T \leq t-1 \}. \]
To complete the argument, it thus suffices to show that $(I_t',J_t')_{t \geq 1}$ is an i.i.d. sequence uniformly distributed on $\{1,\ldots,n\}^2$. To see this, introduce the filtration
\[ \cF_t = \sigma ((I,J),(I_s,J_s)_{s \leq t}) \]
and remark that $T$ is a stopping time with respect to this filtration, which implies that $\{ T \leq t-1 \} \in \cF_{t-1}$ and that $\eta \mathbb{I}\{ T \leq t-1 \}$ is $\cF_{t-1}-$measurable. Since $(I_t,J_t)$ is independent from $\cF_{t-1}$ and uniform, this implies that $(I_t',J_t')$ is uniformly distributed conditionally on $\cF_{t-1}$. It follows that $(I_t',J_t')$ is independent from $\cF_{t-1}$ and in particular from $(I_s',J_s')_{1 \leq s \leq t-1}$. This proves that $(I_t',J_t')_{t \geq 1}$ is an i.i.d. uniform sequence. 

Let $\psi=\EE[e^{\cdot f(\sigma)}]$ be the moment generating function of $f(\sigma) = g(\sigma) - \mathbb{E}[g(\sigma)]$. It follows from Lemma \ref{lem_cov_ineq} that
\[ \psi'(\theta) \leq \theta \mathbb{E}\left[ (g(\sigma) - g(\sigma'))_+ (F(\sigma,\sigma'))_+ e^{\theta f(\sigma)} \right]. \]
Moreover, by $(k,n)$-symmetry of $g$,
\begin{align*}
   (F(\sigma,\sigma'))_+ &\leq \sum_{k = 0}^{+\infty} \mathbb{I}\{k < T\} (g(\sigma_k) - g(\sigma'_k))_+ \\
   &= \sum_{k = 0}^{+\infty} \mathbb{I}\{k < T\} (g(\sigma \pi_k) - g(\sigma' \pi_k))_+\;.
\end{align*}
It follows that
\[ (g(\sigma) - g(\sigma'))_+ (F(\sigma,\sigma'))_+ \leq \sum_{k = 0}^{+ \infty} \mathbb{E} \left[ (g(\sigma) - g(\sigma'))_+ (g(\sigma \circ \pi_k) - g(\sigma' \circ \pi_k))_+ \mathbb{I}\{ k < T \} \bigl| \sigma,\sigma' \right]\;. \]
Let $p_k = \mathbb{P}(k < T)$. By the Cauchy-Schwarz inequality,
\begin{align*}
    &(g(\sigma) - g(\sigma'))_+ (F(\sigma,\sigma'))_+  \\ 
    &\quad \leq \sum_{k = 0}^{+ \infty} \mathbb{E} \left[ (g(\sigma) - g(\sigma'))_+^2 \mathbb{I}\{ k < T\} | \sigma,\sigma' \right]^{1/2} \mathbb{E} \left[(g(\sigma \circ \pi_k) - g(\sigma' \circ \pi_k))_+^2 \bigl| \sigma,\sigma' \right]^{1/2}. \numberthis \label{eq_decoup_Tolstikhin}
\end{align*}
Let us now analyze the random variable $T$.  First, remark that
\begin{align*}
   T &= \min \{ t \in \mathbb{N} : \pi_t^{-1} \tau_{I,J} \pi_t \in \mathfrak{S}_{k,n-k} \} \\ 
   &= \min \{ t \in \mathbb{N} : \tau_{\pi_t^{-1}(I),\pi_t^{-1}(J)} \in \mathfrak{S}_{k,n-k} \} \\
   &= \min \left\{ t \in \mathbb{N} : (\pi_t^{-1}(I),\pi_t^{-1}(J)) \in \{1,\ldots,k\}^2 \cup \{k+1,\ldots,n\}^2 \right\}\;.
\end{align*}
In particular, $T = 0$ if and only if $(I,J) \in \{1,\ldots,k\}^2 \cup \{k+1,\ldots,n\}^2$. We now assume that this is not the case: $T>0$. Let $(X_t,Y_t)=(\pi^{-1}_t(I),\pi^{-1}_t(J))$ for any $t\geq 0$. For all $t \geq 1$, $\pi_t^{-1} = (\pi_{t-1} \tau_{I_t,J_t})^{-1} = \tau_{I_t,J_t} \pi_{t-1}^{-1}$, hence
\[ (X_t,Y_t) = \left(\tau_{I_t,J_t}(X_{t-1}), \tau_{I_t,J_t}(Y_{t-1}) \right). \]
Since $(I_t,J_t)$ is independent from $(X_s,Y_s)_{1 \leq s \leq t-1}$, this implies that $(X_t,Y_t)$ is a Markov chain started at $(I,J)$ (which we condition on). Define the sets
\[ E = \{1,\ldots,k\}, F = \{k+1,\ldots,n\}, A = E^2, B = E\times F \cup F \times E, C = F^2. \]
For any $t \in \mathbb{N}$, let
\[ Z_t = \begin{cases}
          &A \text{ if } (X_t,Y_t) \in A\;, \\
          &B \text{ if } (X_t,Y_t) \in B\;, \\
          &C \text{ if } (X_t,Y_t) \in C\;.
 \end{cases} 
 \]
 We show that $(Z_t)_{t \in \mathbb{N}}$ is a Markov chain and compute its transition matrix. Define the filtration
 \[ \cF_t = \sigma((X_s,Y_s)_{1 \leq s \leq t})\;. \]
 Since $(X_t,Y_t)_{t\geq 0}$ is a Markov chain, for any $R \in \{A,B,C\}$,
 \[  \mathbb{P}(Z_t = R | \cF_{t-1}) = \mathbb{P} ((X_t,Y_t) \in R | (X_{t-1},Y_{t-1})) = \mathbb{P} \left((\tau_{I_t,J_t}(X_{t-1}),\tau_{I_t,J_t}(Y_{t-1})) \in R | (X_{t-1},Y_{t-1}) \right). \]
Depending on $Z_{t-1}$, the following cases are possible:
 \begin{itemize}
     \item If $Z_{t-1} = A$, that is $(X_{t-1},Y_{t-1}) \in E^2$, then $(X_t,Y_t) \in E \times F$ if and only if 
     \[ (I_t,J_t) \in \{X_{t-1}\} \times F \cup F \times \{ X_{t-1} \}  \]
     and $(X_t,Y_t) \in F \times E$ if and only if
     \[ (I_t,J_t) \in \{Y_{t-1}\} \times F \cup F \times \{ Y_{t-1} \}\;. \]
     Thus, $(X_t,Y_t) \in B$ if and only if
     \[ (I_t,J_t) \in \{X_{t-1}\} \times F \cup F \times \{ X_{t-1} \} \cup \{Y_{t-1}\} \times F \cup F \times \{ Y_{t-1} \}\;. \]
     This yields
     \[ \mathbb{P}(Z_t = B | X_{t-1},Y_{t-1}) = \mathbb{P}(Z_t = B | Z_{t-1} = A) = 4 \frac{n-k}{n^2}\;. \]
     Moreover, $\mathbb{P}(Z_t = C | X_{t-1},Y_{t-1}) = 0$ since a single transposition can only ``move'' one of $X$ and $Y$ at a time. It follows that
     \[ \mathbb{P}(Z_t = C | X_{t-1},Y_{t-1}) = 0 \text{ and } \mathbb{P}(Z_t = A | X_{t-1},Y_{t-1}) = 1 - 4\frac{n-k}{n^2}\;. \]
     \item Symetrically, by exchanging $k$ and $n-k$, if $Z_{t-1} = C$, then
     \[ \mathbb{P}(Z_t = B | X_{t-1},Y_{t-1}) = \frac{4k}{n^2}, \mathbb{P}(Z_t = A | X_{t-1},Y_{t-1}) = 0 \text{ and } \mathbb{P}(Z_t = C | X_{t-1},Y_{t-1}) = 1 - \frac{4k}{n^2}\;. \]
     \item Finally, consider the case $Z_{t-1} = B$ \textit{i.e.} $(X_{t-1},Y_{t-1}) \in B = E \times F \cup F \times E$. If $X_{t-1} \in E$ and $Y_{t-1} \in F$, then $(X_t,Y_t) \in A = E^2$ if and only if
     \[ (I_t,J_t) \in \{ Y_{t-1} \} \times  \left(E \backslash \{ X_{t-1} \} \right) \cup \left(E \backslash \{ X_{t-1} \} \right) \times \{ Y_{t-1} \} \;, \]
     which occurs with probability 
     \[ \mathbb{P} \left( Z_t = A | (X_{t-1},Y_{t-1}) \right) =  2 \frac{k-1}{n^2}\;.\] 
     The same holds if $X_{t-1} \in F$ and $Y_{t-1} \in E$. By symmetry, we also have that
     \[ \mathbb{P} \left( Z_t = C | (X_{t-1},Y_{t-1}) \right) =  2 \frac{n-k-1}{n^2} \]
     on the event $(X_{t-1},Y_{t-1}) \in B$. Finally,
     \[ \mathbb{P} \left( Z_t = B | (X_{t-1},Y_{t-1}) \right) = 1 - 2 \frac{k-1}{n^2} - 2\frac{n-k-1}{n^2} = 1 - 2 \frac{n-2}{n^2} \]
     on the event $(X_{t-1},Y_{t-1}) \in B$.
 \end{itemize}
 Thus, $(Z_t)_{t\geq 0}$ is a Markov chain with transition matrix: \\
 \begin{tabular}{c|c|c|c}
      & A & B & C  \\
      \hline 
    A & $1 - 4 \frac{n-k}{n^2}$ & $4 \frac{n-k}{n^2}$ & 0 \\
    B & $2 \frac{k-1}{n^2}$ & $1 - 2 \frac{n-2}{n^2}$ & $2 \frac{n-k-1}{n^2}$ \\
    C & $0$ & $4 \frac{k}{n^2}$ & $1 - 4 \frac{k}{n^2}$ 
    \vspace{0.2cm}
 \end{tabular}
 
 \noindent The random variable $T$ is non-zero if and only if $(I,J) = (X_0,Y_0) \in B$ that is, if the chain $Z_t$ starts in state $B$. $T$ is then equal to the exit time from state $B$. Thus, $T$ follows a geometric distribution with parameter $q = 2\frac{n-2}{n^2}$ conditionally on $(I,J) \in B$. Moreover, $T$ is independent of $\sigma$. Thus, the distribution of $T$ conditional on $(\sigma,\sigma')$, that is, conditional on $\sigma,I,J$, is geometric with parameter $q$ if $(I,J) \in B$, \textit{i.e.} if $T > 0$. This yields
 \[ \mathbb{E} \left[ (g(\sigma) - g(\sigma'))_+^2 \mathbb{I}\{ k < T\} | \sigma,\sigma' \right] = \mathbb{I}\{T > 0\} (1-q)^k (g(\sigma) - g(\sigma'))_+^2 = (1-q)^k (g(\sigma) - g(\sigma'))_+^2 \]
 since $g(\sigma) = g(\sigma')$ when $(I,J) \notin B$. It follows from equation \eqref{eq_decoup_Tolstikhin} that
 \begin{align*}
  &(g(\sigma) - g(\sigma'))_+ (F(\sigma,\sigma'))_+ \\ 
  &\quad \leq   \sum_{k = 0}^{+ \infty} (1-q)^{k/2} (g(\sigma) - g(\sigma'))_+ \mathbb{E} \left[(g(\sigma \pi_k) - g(\sigma' \pi_k))_+^2 \bigl| \sigma,\sigma' \right]^{1/2} \\
  &\quad \leq \frac{1}{2} (g(\sigma) - g(\sigma'))_+^2 \sum_{k = 0}^{+ \infty} (1-q)^{k/2} + \frac{1}{2} \sum_{k = 0}^{+ \infty} (1-q)^{k/2}  \mathbb{E} \left[(g(\sigma \pi_k) - g(\sigma' \pi_k))_+^2 \bigl| \sigma,\sigma' \right].
 \end{align*}
 Define the random variable
\[ W = \sum_{k = 0}^{+\infty} \frac{(1-q)^{k/2}}{2} \mathbb{E} \left[ (g(\sigma) - g(\sigma'))_+^2 + (g(\sigma \circ \pi_k) - g(\sigma' \circ \pi_k))_+^2 \bigl| \sigma \right]. \]
By reasoning exactly as in the proof of Theorem \ref{thm_conc_fun_permut}, we find that
\[ \psi(\theta) \leq \mathbb{E} \left[ \exp \left( r \theta^2 \mathbb{E} \left[ (g(\sigma) - g(\sigma'))_+^2 | \sigma  \right] \right) \right], \]
where
\[ r = \sum_{k = 0}^{+ \infty} (1-q)^{k/2} = \frac{1}{1 - \sqrt{1-q}}\;. \]
The quantity $r$ can further be bounded as follows (recall that $q = 2(n-2)/n^2$):
 \begin{align*}
     1-q &= 1 - 2\frac{n-2}{n^2} = \frac{n^2 - 2n + 4}{n^2} = \frac{(n-1)^2 + 3}{n^2} \\
     \sqrt{1-q} &= \frac{n-1}{n} \sqrt{1 + \frac{3}{(n-1)^2}} \\
     &\leq \frac{n-1}{n} \left(1 + \frac{3}{2(n-1)^2} \right) \\
     &\leq \frac{n-1}{n} + \frac{3}{2n(n-1)} \\
    1 - \sqrt{1-q} &\geq \frac{1}{n} - \frac{3}{2n(n-1)} = \frac{1}{n} \left(1 - \frac{3}{2(n-1)} \right) \\
                &= \frac{1}{n} \frac{2n-5}{2n-2} \\
      \frac{1}{1-\sqrt{1-q}} & \leq n \frac{2n-2}{2n-5} \;,        
 \end{align*}
 yielding $r \leq n \frac{2n-2}{2n-5}$. Moreover, since $g$ is $(k,n)-$invariant,
 \begin{align*}
     \mathbb{E} \left[ (g(\sigma) - g(\sigma'))_+^2 | \sigma  \right] &= \frac{1}{n^2} \sum_{i = 1}^n \sum_{j = 1}^n (g(\sigma) - g(\sigma \circ \tau_{i,j}))_+^2 \\
     &= \frac{2}{n^2} \sum_{i = 1}^k \sum_{j = k+1}^n (g(\sigma) - g(\sigma \circ \tau_{i,j}))_+^2 \\
     &= \frac{2}{n^2} \left| \nabla g(\sigma) \right|_+^2\;,
 \end{align*}
 which finally yields, for all $\theta \geq 0$,
 \[ \mathbb{E}\left[ \exp(\theta f(\sigma)) \right] \leq \mathbb{E} \left[ \exp \left( \frac{2}{n} \frac{2n-2}{2n-5}  \theta^2 \left| \nabla g(\sigma) \right|_+^2 \right) \right] \;,\]
which is the first part of Theorem \ref{th_tol_exch}. Futhermore, if $\left| \nabla g(\sigma) \right|_+^2 \leq \Sigma^2$ for some constant $\Sigma^2$ then by the Chernoff bound, for all $t \geq 0$,
 \[ \mathbb{P} \left( f(\sigma) \geq t \right) \leq  \exp\left( \frac{2}{n} \frac{2n-2}{2n-5} \theta^2 \Sigma^2 - \theta t \right) \]
 which yields
\begin{equation*} \label{eq_th_Tolsti_paires_echs}
    \mathbb{P} \left( f(\sigma) \geq t \right) \leq  \exp\left( - \frac{2n-5}{2n-2} \frac{n t^2}{8 \Sigma^2} \right)
\end{equation*}
 by optimizing over $\theta > 0$. 
\end{proof}

\subsection{Proofs related to Section \ref{sec_gen_bound}}\label{ssec_proofs_general_bound}

\begin{proof}[Proof of Theorem \ref{thm_hp_ubd_boot}]
Introduce a uniform random permutation $\sigma$ and let
\begin{equation}
    \overline{g}_\xi(x) = \mathbb{E} \left[ \sup_{t \in \cT} \xi_{\sigma(i)} t(x_{i}) | \xi \right]\;.
\end{equation}
Now, by Theorem \ref{thm_intro_cond_exp} applied to $2\overline{g}_\xi$ -- since for $t$ a function in $\cT$, 2$t$ takes values in $[-1,1]$ -- conditionnally on $\xi,$ with probability at least $1 - e^{-x}$,
\begin{align*}
   2\overline{g}_\xi(X) &\leq 2\mathbb{E}\left[ \overline{g}_\xi(X) | \xi \right] + \sqrt{12 \kappa_\xi 2\mathbb{E}\left[ \overline{g}_\xi(X) | \xi \right]} + 5 \kappa_\xi x
\end{align*}
where
\[ \kappa_\xi = \frac{1}{n} \sum_{i = 1}^n |\xi_i|\;. \]
Moreover, by definition
\[ \mathbb{E}\left[ \overline{g}_\xi(X) | \xi \right] = \mathbb{E}\left[\sup_{t \in \cT} \sum_{i=1}^{n}\xi_{\sigma(i)} t(X_{i})  | \xi \right] = M_n(\xi)  \]
which yields
\[ \overline{g}_\xi(X) \leq M_n(\xi) + \sqrt{6 \kappa_\xi x M_n(\xi)} + \frac{5}{2} \kappa_\xi x \]
with probability at least $1 -e^{-x}$, knowing $\xi$. 
Furthermore, by Theorem \ref{Th_concen_exchange}, for any $\theta \geq 0,$
\[ \mathbb{E} \left[\exp \left( \theta (g(X,\xi_\sigma) - \overline{g}_\xi(X)) \right) \bigr| \xi,X \right] \leq \exp \left(76 \theta^2 \norm{\xi}_\infty^2 v_+(X) \right). \]
Let
\[ \overline{v}_+(X) = \sup_{t \in \cT} \left\{ \sum_{i = 1}^n \left(t(X_i) - \mu_t \right)^2 \right\} \geq v_+(X). \]
As the supremum of sums of independent random variables valued in $[0,1],$ $\overline{v}_+$ is a self-bounding function of $X$. By \cite[Theorem 6.12]{BoucheronLugosiMassart:2013}, for any $\lambda \geq 0$,
\[ \mathbb{E} \left[ \exp(\lambda v_+(X)) \right] \leq \mathbb{E} \left[ \exp(\lambda \overline{v}_+(X)) \right] \leq \exp \left( (e^\lambda - 1) \mathbb{E}[\overline{v}_+(X)] \right). \]
It follows that
\[ \mathbb{E} \left[\exp \left( \theta (g(X,\xi_\sigma) - \overline{g}_\xi(X)) \right) \bigr| \xi \right] \leq \exp \left( \left(e^{76\norm{\xi}_\infty^2 \theta^2} - 1 \right) \mathbb{E}[\overline{v}_+(X)] \right). \]
Fix some $t > 0$. Let $v = \mathbb{E}[\overline{v}_+(X)]$. If $t \leq 2v \norm{\xi}_\infty \sqrt{19},$ let $\theta = \frac{t}{152 \norm{\xi}_\infty^2 v}$, then by the Chernoff bound,
\begin{align*}
    \log \mathbb{P} (g(X,\xi_\sigma) - \overline{g}_\xi(X)) \geq t | \xi) &\leq  v \left( \exp \left( \frac{t^2}{304 \norm{\xi}_\infty^2 v^2} \right) - 1 \right) - \frac{t^2}{152 \norm{\xi}_\infty^2 v} \\
    &\leq 1.14 \frac{t^2}{304 \norm{\xi}_\infty^2 v} - \frac{t^2}{152 \norm{\xi}_\infty^2 v} \text{ since } \frac{t^2}{304 \norm{\xi}_\infty^2 v^2} \leq \frac{1}{4} \\
    &\leq - \frac{0.43 t^2}{152 \norm{\xi}_\infty^2 v}.
\end{align*}
If $t > 2v \norm{\xi}_\infty \sqrt{19},$ then by the Chernoff bound with $\theta = \frac{1}{4 \sqrt{19} \norm{\xi}_\infty}$,
\begin{align*}
  \log \mathbb{P} (g(X,\xi_\sigma) - \overline{g}_\xi(X)) \geq t | \xi) &\leq v \left(e^{\frac{1}{4}} - 1 \right) - \frac{t}{4 \sqrt{19} \norm{\xi}_\infty} \\
  &\leq \frac{t}{2 \sqrt{19} \norm{\xi}_\infty} \left(e^{\frac{1}{4}} - 1 \right) - \frac{t}{4\sqrt{19} \norm{\xi}_\infty} \\
  &\leq - \frac{0.43 t}{4\sqrt{19} \norm{\xi}_\infty}.
\end{align*}
Hence,
\[ \log \mathbb{P} (g(X,\xi_\sigma) - \overline{g}_\xi(X)) \geq t | \xi) \leq -0.43 \min \left( \frac{t^2}{152 \norm{\xi}_\infty^2 v}, \frac{t}{4\sqrt{19} \norm{\xi}_\infty} \right). \]
As a consequence, 
\[ \mathbb{P} \left(g(X,\xi_\sigma) - \overline{g}_\xi(X) \geq \max \Bigl( \frac{4\norm{\xi}_\infty \sqrt{19 v x}}{\sqrt{0.86}}, \frac{8x\sqrt{19} \norm{\xi}_\infty}{0.86} \Bigr) \Bigr| \xi \right) \leq e^{-x}. \]
Now, let us bound $v$ using symmetrization. Let $X^*$ be an independent copy of $X$ and let $\varepsilon \in \{-1;1\}^n$ be an i.i.d. vector of Rademacher random variables independent from $X,X^*$. It holds
\begin{align*}
    v &= \mathbb{E} \left[ \sup_{t \in \cT} \left\{ \sum_{i = 1}^n (t(X_i)-\mu_t)^2 \right\} \right] \\
    &\leq \sigma^2 + \mathbb{E} \left[ \sup_{t \in \cT} \left\{ \sum_{i = 1}^n (t(X_i)-\mu_t)^2 - \mathbb{E} \bigl[ (t(X_i)-\mu_t)^2 \bigr] \right\} \right] \\
    &\leq \sigma^2 + \mathbb{E} \left[ \sup_{t \in \cT} \left\{ \sum_{i = 1}^n (t(X_i)-\mu_t)^2 - (t(X_i^*)-\mu_t)^2 \right\} \right] \text{ by Jensen's inequality} \\
    &= \sigma^2 + \mathbb{E} \left[ \sup_{t \in \cT} \left\{ \sum_{i = 1}^n \varepsilon_i \left((t(X_i)-\mu_t)^2 - (t(X_i^*)-\mu_t)^2 \right) \right\} \right] \\
    &= \sigma^2 + \mathbb{E} \left[ \sup_{t \in \cT} \left\{ \sum_{i = 1}^n \varepsilon_i (t(X_i) - t(X_i^*)) \left(t(X_i) + t(X_i^*) - 2\mu_t \right) \right\} \right] \\
    &\leq \sigma^2 + 2\mathbb{E} \left[ \sup_{t \in \cT} \left\{ \sum_{i = 1}^n \varepsilon_i (t(X_i) - t(X_i^*)) \right\} \right] \text{ by contractivity } \\
    &= \sigma^2 + 2\mathbb{E} \left[ \sup_{t \in \cT} \left\{ \sum_{i = 1}^n (t(X_i) - t(X_i^*)) \right\} \right] \\
    &\leq \sigma^2 + 4 M_n\;.
\end{align*}
It follows that
\[ g(X,\xi_\sigma) - \overline{g}_\xi(X) \leq \norm{\xi}_\infty \max \left( 19 \sqrt{x(4M_n +\sigma^2)}, 41 x \right) \]
with probability at least $1 - e^{-x}$ (knowing $\xi$).

It follows that conditionally on $\xi$, with probability at least $1 - 2e^{-x}$,
\[ g(X,\xi_\sigma) \leq M_n(\xi) + \sqrt{6 \kappa_\xi x M_n(\xi)} + \frac{5}{2} \kappa_\xi x + \norm{\xi}_\infty \max \left( 19 \sqrt{x(4M_n +\sigma^2)}, 41 x \right). \]
Since $\xi_\sigma$ is equal in distribution to $\xi$ and $\norm{\xi_\sigma}_\infty = \norm{\xi}_\infty,$ $\kappa_\xi = \kappa_{\xi_\sigma}$, $M_n(\xi_\sigma) = M_n(\xi),$ it follows that with probability at least $1 - 2e^{-x}$,
\[ g(X,\xi) \leq M_n(\xi) + \sqrt{6\kappa_\xi x M_n(\xi)} + \frac{5}{2} \kappa_\xi x + \norm{\xi}_\infty \max \left( 19 \sqrt{x(4M_n +\sigma^2)}, 41 x \right).  \]
(note that this probability is \emph{not} conditional on $\xi$).
\end{proof}

\begin{proof}[Proof of Lemma \ref{lem_quantile_condi}]
    Let $t\in \R$. Since
    \[ \mathbb{I}\{ Y > t\} \geq \mathbb{I}\{ Y \geq q_\alpha(Y | X) > t \} = \mathbb{I}\{ Y \geq q_\alpha(Y | X) \} \mathbb{I}\{ q_\alpha(Y | X) > t \}, \]
    it follows by taking conditional expectations that, almost surely,
    \[ \mathbb{P}(Y > t | X) \geq \mathbb{P}(Y \geq q_\alpha(Y | X) | X) \mathbb{I}\{ q_\alpha(Y | X) > t \} \geq \alpha \mathbb{I}\{ q_\alpha(Y | X) > t \} \;. \]
    Taking expectations yields
    \[ \mathbb{P}(Y > t) \geq \alpha \mathbb{P} \left( q_\alpha(Y | X) > t \right)\;. \]
    If $t = q_{\gamma \alpha}(Y)$ then by definition,
    \[ \alpha \gamma \geq \mathbb{P}(Y > t) \geq \alpha \mathbb{P} \left( q_\alpha(Y | X) > t \right)\;, \]
    which proves that
    \[ \mathbb{P} \left( q_\alpha(Y | X) > t \right) \leq \gamma \]
    and hence that
    \[ q_\gamma \left( q_\alpha(Y | X) \right) \leq t = q_{\gamma \alpha}(Y). \]
\end{proof}

\begin{proof}[Proof of Corollary \ref{cor_ubd_quant_boot}]
    Let $J$ be uniformly distributed on $\{1,\ldots,B\}$ and independent from $X$ and $(\xi^{b})_{b = 1,\ldots,B}$, then by definition
    \[ \hat{q}_\alpha^B = q_\alpha \left( g(X,\xi^{(J)}) | X,\xi \right). \]
    It follows from Lemma \ref{lem_quantile_condi} above that
    \[ q_\gamma \left( \hat{q}_\alpha^B \right) \leq q_{\gamma \alpha} \left( g(X,\xi^{(J)}) \right) = q_{\gamma \alpha}(g(X,\xi))\;. \]
    Let $x = \log 2 - \log (\gamma \alpha_1)$ and let $Y$ be the upper bound given by Theorem \ref{thm_hp_ubd_boot} (with this value of $x$) and with $\kappa_\xi$ replaced by its upper bound $2$. Theorem \ref{thm_hp_ubd_boot} states that with probability at least $1 - \gamma \alpha_1,$ $g(X,\xi) \leq Y$. By the union bound, this implies that with probability at least $1 - \gamma \alpha,$ $g(X,\xi) \leq q_{\gamma(\alpha_2 + \alpha_3)}(Y)$. Thus, 
    \[ q_\gamma \left( \hat{q}_\alpha^B \right) \leq q_{\gamma(\alpha_2 + \alpha_3)}(Y)\;. \]
    We can write $Y$ in the form $Y = \phi(M_n(\xi),\norm{\xi}_\infty)$ where the (deterministic) function $\phi$ is non-decreasing in both of its variables. By the union bound, with probability at least $\gamma(\alpha_2 + \alpha_3)$,
    \[ \norm{\xi}_\infty \leq q_{\gamma \alpha_3}(\norm{\xi}_\infty) \text{ and } M_n(\xi) \leq q_{\gamma \alpha_2}(M_n(\xi))\;, \]
    in which case
    \[ Y = \phi(M_n(\xi),\norm{\xi}_\infty) \leq \phi \left(q_{\gamma \alpha_2}(M_n(\xi)),  q_{\gamma \alpha_3}(\norm{\xi}_\infty) \right)\;. \]
    This proves that
    \[ q_\gamma \left( \hat{q}_\alpha^B \right) \leq q_{\gamma(\alpha_2 + \alpha_3)}(Y) \leq \phi \left(q_{\gamma \alpha_2}(M_n(\xi)),  q_{\gamma \alpha_3}(\norm{\xi}_\infty) \right)\;, \]
    which is the result.
\end{proof}

\subsection{Confidence regions for the mean: optimized version and proof}\label{ssec_sm_conf_reg}
The following version of Theorem \ref{thm_conf_reg_mean} handles the special case of symmetric random variables.

\begin{theorem} \label{thm_conf_reg_mean_gen}
    Let $\kappa = \mathbb{E}[|\xi_1|]$ and $b = \norm{\xi_1}_\infty$. Assume that $\norm{X_1 - \mu} \leq M$ almost surely for some constant $M$. Let $\eta = 1$ if $X_1$ is symmetrically distributed and $\eta = 2$ otherwise. 
    With probability at least $1 - 2e^{-x},$ for all $\theta > 0,$
    \[ \norm{\overline{X}_n - \mu} \leq (1+\theta)^2 \frac{\eta}{\kappa} \hat{R}_n + \sigma_\cB \sqrt{\frac{2x}{n}} + \left(\frac{3\eta + 1}{\theta} + 9\eta + 1 + 9 \eta \theta + 3 \eta \theta^2 \right) \frac{xM}{n}. \]
    where 
    \[ \sigma_\cB = \sup_{l \in \cB^*} \sqrt{\mathrm{Var}(l(X_1))}.  \]
    Moreover, with probability at least $1 - 2e^{-x},$ for all $\theta > 0,$ 
    \[ \norm{\overline{X}_n - \mu} \geq (1-\theta)^2 \frac{\hat{R}_n}{\eta b} - \sigma_\cB \sqrt{\frac{2x}{n}} - \left(\frac{3\kappa}{\eta b} + 1 \right) \frac{xM}{\theta n} - \left(2 - 4\frac{\kappa}{\eta b} + \theta^2 \frac{\kappa}{\eta b} \right) \frac{xM}{n}. \]
\end{theorem}

Theorem \ref{thm_conf_reg_mean_gen} yields Theorem \ref{thm_conf_reg_mean} by setting $\eta = 2$.
Let us prove Theorem \ref{thm_conf_reg_mean_gen}.

\begin{lemma} \label{lem_ineq_quad}
    Let $y,t,a$ be non-negative reals. Then
    \[ y \geq t - a\sqrt{t} \implies t \leq y + a\sqrt{y + \frac{a^2}{4}} + \frac{a^2}{2}. \]
    Let also $b \leq y$ be non-negative, then for all $\theta \in [0,1]$,
    \[ y \leq t + a\sqrt{t} +b \implies t \geq (1-\theta)(y-b) - \frac{a^2}{4} \left(\frac{1}{\theta} + \theta - 2 \right). \]
\end{lemma}
\begin{proof}
    The conclusion is obvious when $t \leq y$. Assume now that $t > y$, then
    \begin{align*}
        y \geq t - a \sqrt{t} &\iff a\sqrt{t} \geq t - y \\
        &\iff a^2 t \geq (t-y)^2 = t^2 -2y t + y^2 \\
        &\iff t^2 - (2y+a^2)t + y^2 \leq 0
    \end{align*}
    The discriminant of this quadratic inequality is $\Delta = (2y+a^2)^2 - 4y^2 = 4ya^2 + a^4$ and the leading coefficient is positive, so
    \[ t^2 - (2y+a^2)t + y^2 \leq 0 \implies t \leq \frac{2y+a^2 + \sqrt{4ya^2 + a^4}}{2} = y + a\sqrt{y + \frac{a^2}{4}} + \frac{a^2}{2} \]
    Consider now the second inequality. The conclusion is obvious when $t \geq y - b$. Assume now that $t < y - b$, then
    \begin{align*}
        y \leq t + a\sqrt{t} +b &\implies a \sqrt{t} \geq y - t -b \\
        &\implies a^2 t \geq (y-t-b)^2 \\
        &\implies a^2 t \geq (y-b)^2 - 2(y-b)t + t^2 \\
        &\implies (y-b)^2 - (2(y-b)+a^2)t + t^2 \leq 0.
    \end{align*}
    The discriminant is $\Delta = (2(y-b)+a^2)^2 - 4(y-b)^2 = 4(y-b) a^2 + a^4$ and the leading coefficient is positive, so
    \begin{align*}
        (y-b)^2 - (2(y-b)+a^2)t + t^2 \leq 0 \implies t &\geq \frac{2(y-b)+a^2 - \sqrt{4(y-b)a^2 + a^4}}{2} \\ 
        &= y-b + \frac{a^2}{2} - a \sqrt{y - b + \frac{a^2}{4}}. 
    \end{align*}
    If $y > b$, then the right hand side is at most $y-b$. Now, for any $\theta \in [0,1]$,
    \[ t \geq (1-\theta)(y - b) + \left(1 - \frac{\theta}{2} \right) \frac{a^2}{2} - \frac{a^2}{4\theta} = (1-\theta)(y-b) - \frac{a^2}{4} \left(\frac{1}{\theta} + \theta - 2 \right) \]
    which proves the result.
\end{proof}  

Assume that $\norm{X_1 - \mu} \leq M,$ then since $\sum_{i = 1}^n \xi_i = 0$,
\begin{align*}
    \frac{n \hat{R}_n}{M} &= \mathbb{E} \left[\norm{\frac{1}{M} \sum_{i = 1}^n \xi_i(X_i-\mu)} | X_1,\ldots,X_n \right] \\
    &= \mathbb{E} \left[ \sup_{l \in \cB^*} \left\{ \sum_{i = 1}^n \xi_i l\left( \frac{X_i-\mu}{M} \right)  \right\} \Bigr| X_1,\ldots,X_n \right] \\
    &= \overline{g}(X),
\end{align*}
setting $\cT = \{x \mapsto l \left( \frac{x-\mu}{M} \right) : l \in \cB^* \}$. Let $\kappa =\mathbb{E}[|\xi_1|]$. By Theorem \ref{thm_self_bounding}, with probability at least $1 - e^{-x}$,
\[ \frac{n\hat{R}_n}{M} \geq \frac{n\mathbb{E}[\hat{R}_n]}{M} - \sqrt{12\kappa x \mathbb{E}[\hat{R}_n]} \sqrt{\frac{n}{M}}  \]
and 
\[ \frac{n\hat{R}_n}{M} \leq \frac{n\mathbb{E}[\hat{R}_n]}{M} + \sqrt{12\kappa x \mathbb{E}[\hat{R}_n]} \sqrt{\frac{n}{M}} + 5 \kappa x  \]
which yields
\begin{align*}
    \hat{R}_n &\geq \mathbb{E}\left[\hat{R}_n \right] - \sqrt{12 \kappa M \mathbb{E}[\hat{R}_n]} \sqrt{\frac{x}{n}}  \\
   \hat{R}_n &\leq \mathbb{E}\left[\hat{R}_n \right] + \sqrt{12 \kappa M \mathbb{E}[\hat{R}_n]} \sqrt{\frac{x}{n}} + 5 \kappa M \frac{x}{n}
\end{align*}
on events $E_x^1, E_x^2$ each with probability $\geq 1 - e^{-x}$. 
By lemma \ref{lem_ineq_quad}, on $E_x^1,$
\begin{equation}
 \mathbb{E}\left[\hat{R}_n \right] \leq \hat{R}_n + \sqrt{12 \kappa M} \sqrt{\frac{x}{n}} \sqrt{\hat{R}_n + 3\kappa M \frac{x}{n}} + 6\kappa M \frac{x}{n}.    
\end{equation}
Hence, for any $\theta > 0,$
\begin{align*}
  \mathbb{E}\left[\hat{R}_n \right] &\leq (1+\theta) \hat{R}_n + \left(3\theta + \frac{3}{\theta} \right) \kappa M \frac{x}{n} + 6 \kappa M \frac{x}{n}  \\
  &\leq (1+\theta) \hat{R}_n + \left(\theta + \frac{1}{\theta} + 2 \right) 3\kappa M \frac{x}{n}. \numberthis \label{eq_conf_ubd_boot}
\end{align*}
Moreover, by lemma \ref{lem_ineq_quad}, for any $\theta \in [0,1],$ on $E_x^2$
\begin{align*}
  \mathbb{E}\left[\hat{R}_n \right] &\geq (1-\theta) \left( \hat{R}_n - 5 \kappa M \frac{x}{n} \right) - 3\kappa M \frac{x}{n} \left(\frac{1}{\theta} + \theta - 2 \right) \\
  &= (1-\theta) \hat{R}_n - \kappa M \frac{x}{n} \left( \frac{3}{\theta} - 1 - 2\theta \right) \numberthis \label{eq_conf_lbd_boot}
\end{align*}

On the other hand, by Bousquet's inequality \cite[Theorem 12.5]{BoucheronLugosiMassart:2013} for the centered empirical process
\begin{align*}
    \frac{n}{M} \norm{\overline{X}_n - \mu} &= \sup_{l \in \cB^*} \left\{ \sum_{i = 1}^n l \left( \frac{X_i - \mu}{M} \right)  \right\} 
\end{align*}
for any $x > 0$, with probability at least $1 - e^{-x},$
\[ \frac{n}{M} \norm{\overline{X}_n - \mu} \leq \frac{n}{M} \mathbb{E} \left[ \norm{\overline{X}_n - \mu} \right] + \sqrt{2x\left(2\frac{n}{M} \mathbb{E} \left[ \norm{\overline{X}_n - \mu} \right] + \frac{n \sigma_\cB^2}{M^2} \right)} + \frac{2x}{3} \]
which yields
\[ \norm{\overline{X}_n - \mu} \leq \mathbb{E} \left[ \norm{\overline{X}_n - \mu} \right] + \sqrt{\frac{2x}{n}} \sqrt{2M \mathbb{E} \left[ \norm{\overline{X}_n - \mu} \right] + \sigma_\cB^2} + \frac{2x M}{3 n}. \]
where $\sigma_\cB^2$ is the ``wimpy variance'',
\[ \sigma_\cB^2 = \sup_{l \in \cB^*} \left\{ \mathrm{Var}(l(X_1)) \right\}. \]
Thus, for any $\theta > 0$,
\begin{align*}
    \norm{\overline{X}_n - \mu} &\leq (1+\theta) \mathbb{E} \left[ \norm{\overline{X}_n - \mu} \right] + \frac{M x}{\theta n} + \sigma_\cB \sqrt{\frac{2x}{n}} + \frac{2x M}{3 n} \\
    &\leq (1+\theta) \mathbb{E} \left[ \norm{\overline{X}_n - \mu} \right] + \sigma_\cB \sqrt{\frac{2x}{n}} + \left(\frac{1}{\theta} + \frac{2}{3} \right) \frac{M x}{n}
\end{align*}

Let $\eta = 1$ if $X_1$ is symmetric about $\mu$ and $\eta = 2$ otherwise.
Now, by 
Proposition \ref{prop_lower_mean_boot} and since $\mathbb{E}[(\xi_1)_+] = \frac{1}{2} \mathbb{E}[|\xi_1|] = \frac{\kappa}{2},$  
\[\mathbb{E} \left[ \norm{\overline{X}_n - \mu} \right] \leq \frac{\eta}{\kappa} \mathbb{E}[\hat{R}_n],\] so combining the above bound with equation \ref{eq_conf_ubd_boot} yields
\[ \norm{\overline{X}_n - \mu} \leq (1+\theta)^2 \frac{\eta}{\kappa} \hat{R}_n + \sigma_\cB \sqrt{\frac{2x}{n}} + \left( 3\eta \left(\frac{1}{\theta} + 3 + 3\theta + \theta^2 \right) + \frac{1}{\theta} + \frac{2}{3} \right) \frac{xM}{n} \]
for any $\theta > 0$, with probability at least $1 - 2e^{-x}$. 
Finally,
\[ \norm{\overline{X}_n - \mu} \leq (1+\theta)^2 \frac{\eta}{\kappa} \hat{R}_n + \sigma_\cB \sqrt{\frac{2x}{n}} + \left(\frac{3\eta + 1}{\theta} + 9\eta + 1 + 9 \eta \theta + 3 \eta \theta^2 \right) \frac{xM}{n}. \]
Consider now the lower confidence bound. By the Klein-Rio lower bound \cite[Theorem 1.2]{KleinRio2005}, with probability at least $1 - e^{-x}$,
\[ \frac{n}{M} \norm{\overline{X}_n - \mu} \geq \frac{n}{M} \mathbb{E} \left[ \norm{\overline{X}_n - \mu} \right] - \sqrt{2x\left(2\frac{n}{M} \mathbb{E} \left[ \norm{\overline{X}_n - \mu} \right] + \frac{n \sigma_\cB^2}{M^2} \right)} - 2x \]
or equivalently,
\[ \norm{\overline{X}_n - \mu} \geq \mathbb{E} \left[ \norm{\overline{X}_n - \mu} \right] - \sqrt{\frac{2x}{n}} \sqrt{2M \mathbb{E} \left[ \norm{\overline{X}_n - \mu} \right] + \sigma_\cB^2} - 2M \frac{x}{n}. \]
Thus, for any $\theta \in [0,1],$
\begin{align*}
    \norm{\overline{X}_n - \mu} &\geq (1-\theta) \mathbb{E} \left[ \norm{\overline{X}_n - \mu} \right] - \frac{Mx}{\theta n} - \sigma_\cB \sqrt{\frac{2x}{n}} - 2 \frac{Mx}{n} \\
    &= (1-\theta) \mathbb{E} \left[ \norm{\overline{X}_n - \mu} \right] - \sigma_\cB \sqrt{\frac{2x}{n}} - \left(\frac{1}{\theta} + 2 \right) \frac{Mx}{n}
\end{align*}
By Proposition \ref{prop_upper_mean_boot}, $\eta b \mathbb{E} \left[ \norm{\overline{X}_n - \mu} \right] \geq \mathbb{E}[\hat{R}_n]$. It follows from equation \eqref{eq_conf_lbd_boot} that on an event with probability at least $1 - 2e^{-x}$,
\begin{align*}
    \norm{\overline{X}_n - \mu} &\geq (1-\theta)^2 \frac{\hat{R}_n}{\eta b} - M \frac{\kappa x}{\eta b n} \left(\frac{3}{\theta} - 4 - \theta + 2\theta^2 \right) - \sigma_\cB \sqrt{\frac{2x}{n}} - \left(\frac{1}{\theta} + 2 \right) \frac{Mx}{n} \\
    &\geq (1-\theta)^2 \frac{\hat{R}_n}{\eta b} - \sigma_\cB \sqrt{\frac{2x}{n}} - \left(\frac{3\kappa}{\eta b} + 1 \right) \frac{xM}{\theta n} - \left(2 - 4 \frac{\kappa}{\eta b} + \theta^2 \frac{\kappa}{\eta b} \right) \frac{xM}{n}
\end{align*}
which proves the lower bound.

\begin{proof}[Proof of Lemma \ref{lemma_Lp_var}]
Let $\Sigma$ be the variance-covariance matrix of $X$ and let $A$ be a square root of $\Sigma$. For any $k \in \{1,\ldots,d\}$, let $c_k$ be the k-th column of $A$. By definition of $A$,
\[ \sigma_k^2 = \Sigma_{k,k} = \sum_{j = 1}^d A_{k,j} A_{j,k} = \norm{c_k}_2^2.  \]
Let $l \in \cB^*$, there exists $\theta \in \mathbb{R}^d$ such that $\norm{\theta}_q \leq 1$ and $\langle \theta,u \rangle = l(u)$ for all $u \in \mathbb{R}^d$. Then
    \[ \mathrm{Var}(l(X)) = \mathrm{Var}(\langle \theta,X\rangle) = \langle \theta, \Sigma \theta \rangle = \norm{A \theta}^2. \]
Now,
\[ A \theta = \sum_{k = 1}^d \theta_k c_k \]
so for any vector $u$ with $\norm{u}_2 \leq 1$, by Hölder's inequality and the Cauchy-Schwarz inequality,
\begin{align*}
    \langle u, A\theta \rangle &= \sum_{k = 1}^d \theta_k \langle u, c_k \rangle \\
    &\leq \norm{\theta}_q \left( \sum_{k = 1}^d |\langle u, c_k \rangle|^p \right)^{1/p} \\
    &\leq \left( \sum_{k = 1}^d \norm{c_k}_2^p \right)^{1/p} \\
    &= \norm{\sigma}_p.
\end{align*}
Since this is true for any $u$ such that $\norm{u}_2 \leq 1$, it follows that
\[ \sqrt{\mathrm{Var}(l(X))} = \norm{A \theta} \leq \norm{\sigma}_p. \]
Since this is true for any $l \in \cB^*$, it follows that $\sigma_\cB \leq \norm{\sigma}_p$.
\end{proof} 

\subsection{Two-sample test: proof of Theorem \ref{th_test_separation_rate}}\label{ssec_two_sample_proof}


\subsubsection{Main steps of the proof}\label{sssec_two_sample_main_steps}

Grant the notations of Section \ref{sec_two_sample_test}. Recall that we consider being under the alternative, where $P\neq Q$, and that the goal is to control the test power, measured by the quantity $\PP(T_{n,m}(\cF)>\hat{q}_B(\alpha))$, with respect to the value of the distance $d_\cF(P,Q)$.

\begin{itemize}
    \item \underline{Step 1: Relating the empirical quantile $\hat{q}_B(\alpha)$ to the true distribution:}
\end{itemize}

\noindent Denote $T_{\xi}(Z)=\sup_{f \in \cF} \left\{ \sum_{i = 1}^{n+m} \xi_i f(Z_i) \right\}$. Recall also that the quantity $\hat{q}_B(\alpha)$ is the empirical $(1-\alpha)$-quantile, for the left-continuous cumulative distribution function $\hat{F}^{(\ell)}_B$, of the sample $(T_{\xi^{(b)}}(Z))_{b=0}^B$. Recall that $\xi^{(0)}=w$ and $(\xi^{(b)})_{b=1}^B$ is an i.i.d. sample, with the same distribution as $\xi=(w_{\sigma(i)})_{i=1}^{n+m}$ and independent of $Z$ and $\xi$. Hence $\hat{q}_B(\alpha)$ should be close to the $(1-\alpha)$-quantile of $T_{\xi}$. The following lemma, the proof of which can be found in Section \ref{sssec_two_sample_tech_lem}, quantifies the latter assertion.
\begin{lemma}\label{lem_F_hat_q}
Let $\delta\in(0,1)$ and denote $F^{(\ell)}_{T_\xi}$ the left-continuous cumulative distribution function of $T_\xi(Z)$ conditionally on $Z$. 
More explicitely, for any $q\in \R$,
\[F^{(\ell)}_{T_\xi(Z)}(q)=\PP( T_\xi(Z) < q \vert Z) \; .\]
With probability at least $1-\delta$, it holds $\hat{q}_B(\alpha)\leq q_\delta(Z)$, where $q_\delta(Z)$ is such that 
\begin{equation}\label{lower_F_q}
F^{(\ell)}_{T_\xi(Z)}(q_\delta(Z))\geq\left (1+\frac{1}{B}\right)\left(1-\alpha + \sqrt{\frac{3\alpha\log(1/\delta)}{B}}\right)\; .
\end{equation}
\end{lemma}
Let us assume that $\alpha, n$ and $B$ are such that the right-hand side of Inequality (\ref{lower_F_q}) is smaller than one, and denote $\alpha_B(\delta)\in (0,1)$ such that
\[1-\alpha_B(\delta) = \left (1+\frac{1}{B}\right)\left(1-\alpha + \sqrt{\frac{3\alpha\log(1/\delta)}{B}}\right)\; .\]

\begin{itemize}
    \item \underline{Step 2: Assessing the ``true quantile'' $q_\delta(Z)$ by concentration of $T_\xi(Z)$:}
\end{itemize}
Let us now compute an admissible value for $q_\delta(Z)$. To do so, we control the concentration of $T_\xi(Z)$ in two steps. First, we apply Theorem \ref{Th_Tolsti} to obtain the concentration conditionally on $Z$. Second, we apply Theorem \ref{thm_intro_cond_exp}, giving access to the concentration of $\EE[T_\xi(Z)\vert Z]$ around $\EE[T_\xi(Z)]$.

Note that the vector $w$ has coordinates with two possible values, $1/n$ and $-1/m$. Hence, according to Theorem \ref{Th_Tolsti}, we have, for any $t\geq 0$, 
\begin{align*}
    \PP(T_\xi(Z)-\EE[T_\xi(Z)\, \vert \, Z]\geq t \, \vert \, Z) &\leq \exp\left( 
 -\frac{t^2}{8\left(\frac{1}{n}+\frac{1}{m}\right)^2} \max \left( \frac{1}{v_+(Z)}, \frac{n+m+2}{nm} \right) \right) \label{ineq_concen_g_x} \\
 &\leq \exp \left(- \frac{t^2}{8 \Delta_{n,m}^2} \right) \;,
\end{align*}
where \[ v_+(Z) = \sup_{f \in \cF} \left\{ \sum_{i = 1}^{n+m} \left(f(Z_i) - \frac{1}{n+m} \sum_{j = 1}^{n+m} f(Z_j) \right)^2 \right\} \]
and
\begin{equation} \label{eq_def_Delta-nm}
   \Delta_{n,m} = \min \left(2\left( \frac{1}{n} + \frac{1}{m} \right) \sqrt{2v_+(Z)}, 2\sqrt{2\left( \frac{1}{n} + \frac{1}{m} \right)} \right).  
\end{equation}
By reparametrizing the inequality, we get, for any $\beta\in (0,1)$,
\begin{equation*}\label{concen_T_xi}
    \PP\left(T_\xi(Z) \geq \EE[T_\xi(Z)\, \vert \, Z]+ \Delta_{n,m} \sqrt{\log\left(\frac{1}{\beta}\right)}\Biggl|Z\right) \leq \beta\;.
\end{equation*}
Hence, by taking $\beta=\alpha_B$, we obtain
\[
1-F^{(\ell)}_{T_\xi(Z)}\left(\EE[T_\xi(Z)\, \vert \, Z]+ \Delta_{n,m} \sqrt{\log\left(\frac{1}{\alpha_B}\right)}\right)\leq \alpha_B\;,
\]
which gives that the value
\begin{equation}\label{value_q_delta_Z}
    q_\delta(Z)= \EE[T_\xi(Z)\, \vert \, Z]+ \Delta_{n,m} \sqrt{\log\left(\frac{1}{\alpha_B}\right)}
\end{equation}
works. In order to compare the values of $T_{n,m}(\cF)$ and $q_\delta(Z)$ with high probability, we will control each random variable separately, and in particular, relate each of those quantities to the distance $d_\cF(P,Q)$.

\begin{itemize}
    \item \underline{Step 3: upper deviations of $q_\delta(Z)$:}
\end{itemize}

\noindent Using Theorem \ref{thm_intro_cond_exp}, we have, for any $t>0$,
\begin{equation*}
    \PP\left(\EE[T_\xi(Z)\, \vert \, Z]\geq \EE[T_\xi(Z)]+2\sqrt{3\EE[\vert \xi_1\vert]\EE[T_\xi(Z)]t}+\frac{5\EE[\vert \xi_1\vert]t}{2}\right) \leq \exp(-t).
\end{equation*}
Using the fact that for any $a,b, \kappa>0$, $\sqrt{ab}<\kappa a+b(4\kappa)^{-1}$, together with the identity $\EE[\vert \xi_1\vert]=2/(n+m)$, we obtain,  for any $t,\kappa>0$,
\begin{equation}\label{concen_cond_exp_T_xi}
    \PP\left(\EE[T_\xi(Z)\, \vert \, Z]\geq (1+\kappa)\EE[T_\xi(Z)]+\frac{6}{n+m}\left(1+\frac{1}{\kappa}\right)t\right) \leq \exp(-t).
\end{equation}
Taking $\kappa = 1$ and combining with equation \eqref{value_q_delta_Z} yields
\begin{equation} \label{eq_ubd_q-delta_gen}
   q_\delta(Z) \leq 2\EE[T_\xi(Z)] + \Delta_{n,m} \sqrt{ \log \left( \frac{1}{\alpha_B} \right)} + \frac{12}{n+m} \log \left( \frac{1}{\beta} \right) 
\end{equation}
with probability no less than $1 - \beta$. To bound $\Delta_{n,m}$ with high probability, let us bound $v_+(Z)$.
By noticing that 
\[
v_+(Z)=\frac{1}{2(n+m)}\sup_{f\in \cF}\left\{ \sum_{i,j = 1}^{n+m} \left(f(Z_i) - f(Z_j) \right)^2  \right\}\;,
\]
we prove in Lemma \ref{lem_v_+} below that the function $v_+/4$ is $(2,0)$-self-bounding.
Hence, by \cite[Theorem 6.21]{BoucheronLugosiMassart:2013}, we have for any $t>0$,
\begin{equation*}
    \PP\left(v_{+}(Z) \geq \EE[v_{+}(Z)]+4\sqrt{\EE[v_{+}(Z)]t}+\frac{10t}{3} \right) \leq \exp(-t)\;,
\end{equation*}
which implies, for any $t,\kappa>0$,
\begin{equation}\label{concen_v++}
    \PP\left(v_{+}(Z) \geq (1+\kappa)\EE[v_{+}(Z)]+4\left(1+\frac{1}{\kappa}\right)t\right) \leq \exp(-t)\;.
\end{equation}
Combining Inequalities (\ref{eq_ubd_q-delta_gen}) and (\ref{concen_v++}) with $t=\log(1/\beta)$ and $\kappa=1$ together with Identity \eqref{eq_def_Delta-nm}, we get: for any $\beta \in (0,1)$, with probability larger than $1-2\beta$,
\begin{align*}
     &q_\delta(Z)\leq 2\EE[T_\xi(Z)]+ \frac{12}{n+m}\log\left(\frac{1}{\beta}\right)+4\left(\frac{1}{n}+\frac{1}{m}\right)\sqrt{\left(\EE[v_{+}(Z)]+4\log\left(\frac{1}{\beta}\right)\right) \log\left(\frac{1}{\alpha_B}\right)}\;.
\end{align*}
\begin{itemize}
    \item \underline{Step 4: Control of the expectations:}
\end{itemize}
By denoting $Z^{(2)}_{i,j,f}=(f(Z_i)-f(Z_j))^2$, we have
\begin{equation*}
    \EE[v_{+}(Z)]\leq \frac{1}{2(n+m)} \EE\left[\sup_{f\in \cF}\left\{\sum_{i,j=1}^{n+m} Z^{(2)}_{i,j,f}-\EE\left[Z^{(2)}_{i,j,f}\right]\right\}\right] +  \frac{1}{2(n+m)}\sup_{f\in \cF} \left\{ \sum_{i,j=1}^{n+m} \EE\left[Z^{(2)}_{i,j,f}\right] \right\}\;.
\end{equation*}
Furthermore, $$\EE\left[Z^{(2)}_{i,j,f}\right]=\var(f(Z_i))+\var(f(Z_j))+(\EE[f(Z_i)]-\EE[f(Z_j)])^2\;.$$ 
Consequently, 
\begin{align*}
     \frac{1}{2(n+m)}\sum_{i,j=1}^{n+m}\EE\left[Z^{(2)}_{i,j,f}\right] & \leq \frac{n+m-1}{n+m} (n \var(f(X_1))+ m \var(f(Y_1)))+\frac{nm}{n+m}(Pf-Qf)^2\\
     &\leq n \var(f(X_1))+ m \var(f(Y_1))+\frac{nm}{n+m}(Pf-Qf)^2\;.
\end{align*}   
In addition, let us denote $\bar{f}_i=f(Z_i)-\EE[f(Z_i)]$. Note that for both $i<j\leq n$ and $n<i<j$, \[ Z^{(2)}_{i,j,f}=(\bar{f}_i-\bar{f}_j)^2\;.\] Also, for $i\leq n<j$, \[ Z^{(2)}_{i,j,f}-\EE[Z^{(2)}_{i,j,f}]=(\bar{f}_i-\bar{f}_j)^2-\EE[(\bar{f}_i-\bar{f}_j)^2]-2(Pf-Qf)(\bar{f}_i-\bar{f}_j)\;.\] Using these identities, we obtain
\begin{align}
  &\EE\left[\sup_{f\in \cF}\left\{\sum_{i,j=1}^{n+m} Z^{(2)}_{i,j,f}-\EE\left[Z^{(2)}_{i,j,f}\right]\right\}\right] \nonumber\\  
  \leq& \EE\left[\sup_{f\in \cF}\left\{\sum_{\substack{i,j=1}}^{n+m} (\bar{f}_i-\bar{f}_j)^2-\EE\left[(\bar{f}_i-\bar{f}_j)^2\right]\right\}\right] + 
  4d_\cF(P,Q)\EE\left[ \sup_{f\in \cF}\left\vert \sum_{i=1}^{n} \sum_{j=n+1}^{n+m} (\bar{f}_i-\bar{f}_j)\right\vert \right] \;.\nonumber \label{upper_square_1}\\
\end{align}
Now, for the first term at the right-hand side of Inequality (\ref{upper_square_1}), by expanding
the quantities $(\bar{f}_i-\bar{f}_j)^2$, we obtain
\begin{align*}
    &\EE\left[\sup_{f\in \cF}\left\{\sum_{\substack{i,j=1}}^{n+m} (\bar{f}_i-\bar{f}_j)^2-\EE\left[(\bar{f}_i-\bar{f}_j)^2\right]\right\}\right]\\
    = & \EE\left[\sup_{f\in \cF}\left\{\sum_{\substack{i,j=1}}^{n+m} (\bar{f}^2_i-\EE[\bar{f}_i^2]+ \bar{f}^2_j-\EE[\bar{f}_j^2])-2\sum_{\substack{i,j=1}}^{n+m}\bar{f}_i\bar{f}_j+2\sum_{\substack{i,j=1}}^{n+m}\EE[\bar{f}_i\bar{f}_j]\right\}\right]\\
    = & \EE\left[\sup_{f\in \cF}\left\{\sum_{\substack{i,j=1}}^{n+m} (\bar{f}^2_i-\EE[\bar{f}_i^2]+ \bar{f}^2_j-\EE[\bar{f}_j^2])-2\left(\sum_{\substack{i=1}}^{n+m}\bar{f}_i\right)^2+2\sum_{\substack{i=1}}^{n+m}\EE[\bar{f}^2_i]\right\}\right]\\ 
    \leq & 2(n+m) \EE \left[ \sup_{f\in \cF}\left\{\sum_{i=1}^{n+m} \bar{f}_i^2-\EE[\bar{f}_i^2]\right\} \right]+2(n\var(f(X_1))+m\var(f(Y_1)))\\
    \leq & 4(n+m)\EE \left[ \sup_{f\in \cF}\left\{\sum_{i=1}^{n+m} \varepsilon_i \bar{f}_i^2\right\} \right]+2(n\var(f(X_1))+m\var(f(Y_1)))\;,
\end{align*}
where $(\varepsilon_i)_{i=1}^n$ is a collection of independent Rademacher variables, independent also of the $Z_i$'s. Note that, almost surely, $\bar{f}_i \in [-2,2]$ for any $i\in \left\{ 1,\ldots, n+m\right\}$, which gives by the contraction principle (conditionally to the $Z_i$'s and using that the function $x\rightarrow x^2/4$ is a contraction on $[-2,2]$), 
\begin{align*}
    &\EE\left[\sup_{f\in \cF}\left\{\sum_{\substack{i,j=1}}^{n+m} (\bar{f}_i-\bar{f}_j)^2-\EE\left[(\bar{f}_i-\bar{f}_j)^2\right]\right\}\right]\\
    \leq &  16(n+m)\EE \left[ \sup_{f\in \cF}\left\{\sum_{i=1}^{n+m} \varepsilon_i \bar{f}_i\right\} \right]+2(n\var(f(X_1))+m\var(f(Y_1)))\\
    \leq & 32 (n+m)\EE \left[ \sup_{f\in \cF}\left\vert\sum_{i=1}^{n+m} \bar{f}_i \right\vert \right] +2(n\var(f(X_1))+m\var(f(Y_1)))\\
    \leq & 32 (n+m)(M_n(P)+M_m(Q))+2(n\var(f(X_1))+m\var(f(Y_1)))\;.\\
\end{align*}
As for the second term on the right-hand side of Inequality \eqref{upper_square_1}, it holds
\begin{align*}
    \EE\left[ \sup_{f\in \cF}\left\vert \sum_{i=1}^{n} \sum_{j=n+1}^{n+m} (\bar{f}_i-\bar{f}_j)\right\vert \right] &
\leq m \EE\left[ \sup_{f\in \cF}\left\vert \sum_{i=1}^{n} \bar{f}_i\right\vert \right]+n\EE\left[ \sup_{f\in \cF}\left\vert \sum_{j=n+1}^{n+m} \bar{f}_j\right\vert \right]\\
&=mM_n(P)+nM_m(Q)\;.
\end{align*}
Putting things together and setting $\alpha_k=k/(n+m)$ for $k\in \{ n, m\}$ and $V=\sup_{f\in \cF} \left\{ n \var(f(X_1)) + m\var(f(Y_1)) \right\}$, we get
\begin{align}\label{upper_v_+}
    &\EE\left[v_+(Z)\right]\nonumber\\
\leq & 16 (M_n(P)+M_m(Q))+2d_\cF(P,Q)(\alpha_m M_n(P)+\alpha_n M_m(Q))+\left(1+\frac{1}{n+m}\right)V+2\frac{nm}{n+m}d_\cF^2(P,Q)
\nonumber\\
\leq & 16 (M_n(P)+M_m(Q))+\left(\frac{1}{n}+\frac{1}{m}\right)(\alpha_m M_n(P)+\alpha_n M_m(Q))^2+\frac{3V}{2}+3\frac{nm}{n+m}d_\cF^2(P,Q)\;,
\end{align}
where in the last inequality, we used the fact that for any $a,b,\eta>0,$ $2ab\leq \eta a^2+b^2/\eta$, by choosing $\eta=(nm)/(n+m)$ -- which gives $1/\eta=1/n+1/m$  --, $a=d_\cF(P,Q)$ and $b=\alpha_m M_n(P)+\alpha_n M_m(Q)$. Furthermore, one can note that 
\begin{align*}
     & \left(\frac{1}{n}+\frac{1}{m}\right)(\alpha_m M_n(P)+\alpha_n M_m(Q))^2 \\
     \leq & 2 \left(\frac{1}{n}+\frac{1}{m}\right)(\alpha^2_mM^2_n(P)+\alpha^2_nM^2_m(Q)) \\
     = & \frac{2m}{n(n+m)} M^2_n(P)+ \frac{2n}{m(n+m)} M^2_m(Q)\;.
\end{align*}
Plugging the latter upper-bound in the right-hand side of Inequality \eqref{upper_v_+}, we get
\begin{equation}
    \EE[v_+(Z)] \leq 16 (M_n(P)+M_m(Q))+\frac{2m}{n(n+m)} M^2_n(P)+ \frac{2n}{m(n+m)} M^2_m(Q)+
    \frac{3V}{2}+\frac{3nm}{n+m}d_\cF^2(P,Q)\;.
\end{equation}
We turn now to the control of the expectation of $T_\xi(Z)$ under the alternative, where $P\neq Q$. Define for any distribution $R$ the quantity

\begin{align*}
    \delta_{n,m}(R) & =  \mathbb{E}_{Z \sim R^{\otimes (n+m)}} \left[ \sup_{f \in \cF} \left\{ \sum_{i = 1}^{n+m} \xi_i f(Z_i) \right\} \right]\\
                   & = \mathbb{E}_{Z \sim R^{\otimes (n+m)}} \left[ \sup_{f \in \cF} \left\{ \sum_{i = 1}^{n+m} \xi_i (f(Z_i)-R(f)) \right\} \right] \;,
\end{align*}
where the latter equality comes from the condition $\sum_{i = 1}^{n+m} \xi_i=0$. Define also
\[ \delta_{n,m}(R,S) = \mathbb{E}_{X \sim R^{\otimes n}} \mathbb{E}_{Y \sim S^{\otimes m}} \left[ \sup_{f \in \cF} \left\{ \sum_{i = 1}^{n} \xi_i f(X_i) + \sum_{j = 1}^{m} \xi_{j+n} f(Y_j) \right\} \right], \]
the expected value of the threshold when $X$ follows the distribution $R$ and $Y$ follows a different distribution $S$. We have the following result, proved in Section \ref{sssec_two_sample_tech_lem} below.
\begin{lemma}\label{lemma_expect_delta_alt} It holds
\begin{align*}
    \delta_{n,m}(R,S) &\leq \delta_{n,m}(R) + \delta_{n,m}(S) + \mathbb{E}\left| \sum_{i = 1}^{n+m} \xi_i \right| d_{\cF}(R,S) \\
    &\leq \delta_{n,m}(R) + \delta_{n,m}(S) + \frac{d_{\cF}(R,S)}{\sqrt{n+m-1}}.
\end{align*}
    
\end{lemma}

\noindent Note that, under the alternative, we have $\EE[T_\xi(Z)]=\delta_{n,m}(P,Q)$. Hence, Lemma \ref{lemma_expect_delta_alt} gives
\begin{equation} \label{ubd_Esp_T-xi-Z}
    \EE[T_\xi(Z)] \leq \delta_{n,m}(P) + \delta_{n,m}(Q) + \frac{d_{\cF}(P,Q)}{\sqrt{n+m-1}} .
\end{equation}
We also have, 
\begin{align*}
    \delta_{n,m}(R) & = \mathbb{E}_{Z \sim R^{\otimes (n+m)}} \left[ \sup_{f \in \cF} \left\{ \sum_{i = 1}^{n+m} w_{\sigma(i)} (f(Z_i)-R(f)) \right\} \right]\\
    & = \mathbb{E}_{Z \sim R^{\otimes (n+m)}} \left[ \sup_{f \in \cF} \left\{ \sum_{i = 1}^{n+m} w_{i} (f(Z_{\sigma^{-1}(i)})-R(f)) \right\} \right]\\
    & = \mathbb{E}_{R^{\otimes (n+m)}} \left[ \sup_{f \in \cF} \left\{ \sum_{i = 1}^{n+m} w_{i} (f(Z_{i})-R(f))\right\} \right]\; ,
\end{align*}
where the last identity comes from the exchangeability of the i.i.d. vector $(Z_1,\ldots,Z_n)$. 
Hence, by the sub-addivity of the supremum and the symmetry of $\mathcal F$, we get
\begin{align*}
     \delta_{n,m}(R) &\leq \frac{1}{n} \mathbb{E}_{ R^{\otimes (n)}} \left[ \sup_{f \in \cF} \left\{ \sum_{i = 1}^{n} (f(Z_{i})-R(f)) \right\} \right] +  \frac{1}{m} \mathbb{E}_{ R^{\otimes (m)}} \left[ \sup_{f \in \cF} \left\{ \sum_{j = 1}^{m} (f(Z_{i})-R(f)) \right\} \right]\\
    &= \frac{M_n(R)}{n}+\frac{M_m(R)}{m} \;.
\end{align*}

\begin{itemize}
    \item \underline{Step 5: Putting things together to bound $q_\delta(Z)$ from above:}
\end{itemize}

\noindent Putting the previous computations together, we get that, with probability at least $1-2\beta$,
\begin{align}
    &q_\delta(Z) \nonumber\\
    &\leq \frac{2}{n}(M_n(P)+M_n(Q))+\frac{2}{m}(M_m(P)+M_m(Q))+\frac{2d_{\cF}(P,Q)}{\sqrt{n+m-1}}+ \frac{12}{n+m}\log\left(\frac{1}{\beta}\right)+ \nonumber \\
    &+ 4\left(\frac{1}{n}+\frac{1}{m}\right) \sqrt{\log\left(\frac{1}{\alpha_B}\right)} \times\nonumber \\
     &\sqrt{16 (M_n(P)+M_m(Q))+\frac{2m}{n(n+m)} M^2_n(P)+ \frac{2n}{m(n+m)} M^2_m(Q)+\frac{3V}{2}+\frac{3nm}{n+m}d_\cF^2(P,Q)+4\log\left(\frac{1}{\beta}\right)}\;. \label{upper_q_delta}
\end{align}
Moreover, from equations \eqref{eq_ubd_q-delta_gen}, \eqref{ubd_Esp_T-xi-Z} and \eqref{eq_def_Delta-nm}, it follows that
\begin{align}
 q_\delta(Z) &\leq \frac{2}{n}(M_n(P)+M_n(Q))+\frac{2}{m}(M_m(P)+M_m(Q))+\frac{2d_{\cF}(P,Q)}{\sqrt{n+m-1}} \nonumber \\  
 &\quad + 2\sqrt{2\left(\frac{1}{n} + \frac{1}{m} \right) \log \left( \frac{1}{\alpha_B} \right)} + \frac{12}{n+m}\log\left(\frac{1}{\beta}\right) \label{ubd_qdelta_hoeff}
\end{align}

\begin{itemize}
    \item \underline{Step 6: deviations from below for $T_{n,m}(\cF)$:}
\end{itemize}

\noindent To control the power of the test, it remains to bound from below the statistic $T=T_{n,m}(\cF)$. 
Fix some $\varepsilon > 0$. Let $f \in \cF$ be such that $Pf - Qf \geq d_{\cF}(P,Q) - \varepsilon$. By definition of $T,$
\[ T \geq \frac{1}{n} \sum_{i = 1}^n f(X_i) - \frac{1}{m} \sum_{j = 1}^n f(Y_j) \geq 
 \frac{1}{n} \sum_{i = 1}^n (f(X_i) - Pf) - \frac{1}{m} \sum_{j = 1}^m (f(Y_j) - Qf) + d_{\mathcal{F}}(P,Q) - \varepsilon. \]
 By Hoeffding's inequality, with probability greater than $1 - e^{-t}$,
 \[ T \geq d_{\mathcal{F}}(P,Q) - \varepsilon - \sqrt{2t\left(\frac{1}{n} + \frac{1}{m} \right)}. \]
 Since this bound holds for any $\varepsilon > 0$, it follows that with probability at least $1 - \delta$,
 \begin{equation} \label{lbd_test_stat_hoeff}
    T \geq d_{\mathcal{F}}(P,Q) - \sqrt{2\left(\frac{1}{n} + \frac{1}{m} \right)\log\left( \frac{1}{\delta} \right)}. 
 \end{equation}
 Thus, if the right-hand side of Inequality \eqref{lbd_test_stat_hoeff} is greater than the right-hand side of \eqref{ubd_qdelta_hoeff} with $\beta=\delta$, then 
\[
\PP (T>\hat{q}_B(\alpha)) \geq 1-3\delta.
\]
In other words, the type 2 error of the test is at most $3\delta$. A few calculations show that this is ensured if
\begin{align*}
  &\left(1 - \frac{2}{\sqrt{n+m-1}} \right)  d_{\cF}(P,Q) \\
  &\quad \geq  \frac{2}{n}(M_n(P)+M_n(Q))+\frac{2}{m}(M_m(P)+M_m(Q))+ \frac{12}{n+m}\log\left(\frac{1}{\delta}\right) \\
  &\quad + \sqrt{\frac{1}{n} + \frac{1}{m}} \left(2\sqrt{2\log \left( \frac{1}{\alpha_B} \right)} + \sqrt{2\log \left( \frac{1}{\delta} \right)} \right).
\end{align*}
 Let us now control the type 2 error by taking into account the variance terms. By Bernstein's inequality, with probability greater than $1 - e^{-t}$,
 \begin{align*}
   T &\geq d_{\mathcal{F}}(P,Q) - \varepsilon - \sqrt{2t \left( \frac{\var(f(X_1))}{n} + \frac{\var(f(Y_1))}{m} \right)} - t \left( \frac{1}{n} \vee \frac{1}{m} \right) \\
   &\geq d_{\mathcal{F}}(P,Q) - \varepsilon - \sqrt{2t \left( \frac{\sigma^2_P(\cF)}{n} + \frac{\sigma^2_Q(\cF)}{m} \right)} - t \left( \frac{1}{n} \vee \frac{1}{m} \right).
 \end{align*}
Again, this bound holds for any $\varepsilon > 0$, which gives that, with probability at least $1 - \delta$,
\begin{equation} \label{lbd_test_stat_alternative}
   T \geq d_{\mathcal{F}}(P,Q) - \sqrt{2\log\left(\frac{1}{\delta}\right) \left( \frac{\sigma^2_P(\cF)}{n} + \frac{\sigma^2_Q(\cF)}{m} \right)} - \log\left(\frac{1}{\delta}\right)\left( \frac{1}{n} \vee \frac{1}{m} \right). 
\end{equation}
Thus, if the right-hand side of Inequality \eqref{lbd_test_stat_alternative} is greater than the right-hand side of \eqref{upper_q_delta} with $\beta=\delta$, then 
\[
\PP (T>\hat{q}_B(\alpha)) \geq 1-3\delta.
\]
A few basic calculations show that this is ensured if
\begin{align*}
    &\left(1 - \frac{2}{\sqrt{n+m-1}}-4\sqrt{3\left(\frac{1}{n}+\frac{1}{m}\right)\log\left(\frac{1}{\alpha_B}\right)} \right) d_{\cF}(P,Q)\\
    &\quad \geq \left(\frac{1}{n}\vee \frac{1}{m}\right)\log\left( \frac{1}{\delta}\right)+  \frac{2}{n}(M_n(P)+M_n(Q))+\frac{2}{m}(M_m(P)+M_m(Q))+ \frac{12}{n+m}\log\left(\frac{1}{\delta}\right) \\
    &+ \sqrt{2\left( \frac{\sigma^2_P(\cF)}{n} + \frac{\sigma^2_Q(\cF)}{m} \right)\log\left(\frac{1}{\delta}\right)} +2\left(\frac{1}{n}+\frac{1}{m}\right)\sqrt{\log\left(\frac{1}{\alpha_B}\right)} \times  \\ 
    &\times\sqrt{\left(34 (M_n(P)+M_m(Q))+\frac{2m}{n(n+m)} M^2_n(P)+ \frac{2n}{m(n+m)} M^2_m(Q)+V+4\log\left(\frac{1}{\beta}\right)\right)}\;.
\end{align*}
This complete the proof of Theorem \ref{th_test_separation_rate}.
\subsubsection{Technical Lemmas}\label{sssec_two_sample_tech_lem}

\begin{proof}[Proof of Lemma \ref{lem_F_hat_q}]
First note that, for any $q\in \R$,
\begin{equation}
    \hat{F}^{(\ell)}_B(q)=\frac{1}{B+1}\sum_{i=0}^B \mathbb{I}_{\left\{ T_{\xi^{(b)}}(Z) <q \right\}}
\geq \frac{B}{B+1} \tilde{F}_B^{(\ell)}(q)\; ,
\end{equation}
where \[\tilde{F}_B^{(\ell)}(q)=\frac{1}{B} \sum_{i=1}^B \mathbb{I}_{\left\{ T_{\xi^{(b)}}(Z)<q\right\}}\] is the left-continuous version of the empirical cumulative distribution function of the variables $(T_{\xi^{(1)}}(Z),\ldots,T_{\xi^{(B)}}(Z))$.

Now, conditionally on $Z$, $(T_{\xi^{(1)}}(Z),\ldots,T_{\xi^{(B)}}(Z))$ is an i.i.d. sample with the same marginal distribution as $T_{\xi}(Z)$. Note also that 
\[ \tilde{F}_B^{(\ell)}(q)-F_{T_\xi(Z)}^{(\ell)}(q) = (1-F_{T_\xi(Z)}^{(\ell)}(q))-(1-\tilde{F}_B^{(\ell)}(q)) = \PP(T_\xi(Z)\geq q\vert Z)-\frac{1}{B} \sum_{b=0}^B \mathbb{I}_{\left\{ T_{\xi^{(b)}}(Z)\geq q \right\}} \; .\]
Hence, by Chernoff's multiplicative bound (\cite{hagerup1990guided}) applied to the right-hand side of the latter identity conditionally to $Z$, it holds 
\begin{equation}
    \PP\left(\PP(T_\xi(Z)\geq q\vert Z)-\frac{1}{B} \sum_{b=0}^B \mathbb{I}_{\left\{ T_{\xi^{(b)}}(Z)\geq q \right\}} \geq -\sqrt{\frac{3\PP(T_\xi(Z)\geq q\vert Z)\log(1/\delta)}{B}}\Biggl| Z\right)\geq 1-\delta \;.
\end{equation}
Equivalently, with probability at least $1-\delta$,
\[
\frac{1}{B} \sum_{b=0}^B \mathbb{I}_{\left\{ T_{\xi^{(b)}}(Z)\geq q \right\}} \leq
\PP(T_\xi(Z)\geq q\vert Z)+\sqrt{\frac{3\PP(T_\xi(Z)\geq q\vert Z)\log(1/\delta)}{B}} \;.
\]
We look for a value of $q$ such that the right-hand side of the latter inequality is smaller than  $\alpha$. Thus, a sufficient condition is given by
\[
\sqrt{\PP(T_\xi(Z)\geq q\vert Z)} \leq \frac{1}{2}\left( -\sqrt{\frac{3\log(1/\delta)}{B}}+\sqrt{\frac{3\log(1/\delta)}{B}+4\alpha}\right) .
\]
Taking the previous inequality to the square and using the fact that $\PP(T_\xi(Z)\geq q\vert Z)=1-F_{T_\xi(Z)}^{(\ell)}(q)$, we obtain
\[
F_{T_\xi(Z)}^{(\ell)}(q) \geq 1-\alpha-\frac{3\log(1/\delta)}{2B}+\frac{1}{2}\sqrt{\frac{3\log(1/\delta)}{B}}\sqrt{\frac{3\log(1/\delta)}{B}+4\alpha}\;.
\]
By using that, for any $a,b>0$, $\sqrt{a+b}\leq \sqrt{a}+\sqrt{b}$, we finally get the following sufficient condition,
\[
F_{T_\xi(Z)}^{(\ell)}(q) \geq 1-\alpha+\sqrt{\frac{3\alpha\log(1/\delta)}{B}}\;,
\]
which concludes the proof.
\end{proof}



\begin{proof}[Proof of Lemma \ref{lemma_expect_delta_alt}]
Write
\begin{align*}
    &\sum_{i = 1}^n \xi_i f(X_i) + \sum_{j = 1}^m \xi_{j+n} f(Y_j) \\
    & =  \sum_{i = 1}^n \xi_i f(X_i) + \sum_{i = n+1}^{n+m} \xi_i Rf
     + \sum_{j = 1}^{m} \xi_{j+n} f(Y_j) + \sum_{j = 1}^{n} \xi_j Sf \\
     &\quad \quad - \left( \sum_{i = n+1}^{n+m} \xi_i \right) Rf - \left( \sum_{i = 1}^{n} \xi_i \right) Sf.        
\end{align*}
Since the weights $\xi$ sum to $0$, the last term is
\[  - \left( \sum_{i = n+1}^{n+m} \xi_i \right) Rf - \left( \sum_{i = 1}^{n} \xi_i \right) Sf = \left( \sum_{i = 1}^{n} \xi_i \right)(Rf - Sf). \]
By sub-additivity of the supremum, it follows that
\begin{align*}
    &\sup_{f \in \cF} \left\{ \sum_{i = 1}^n \xi_i f(X_i) + \sum_{j = 1}^m \xi_{j+n} f(Y_j)\right\} \\
    &\quad \leq \sup_{f \in \cF} \left\{ \sum_{i = 1}^n \xi_i f(X_i) + \sum_{i = n+1}^{n+m} \xi_i Rf \right\} \\
    &\quad \quad + \sup_{f \in \cF} \left\{ \sum_{j = 1}^{m} \xi_{j+n} f(Y_j) + \sum_{j = 1}^{n} \xi_j Sf \right\} \\
    &\quad \quad + \left| \sum_{i = 1}^{n} \xi_i \right| d_{\cF}(R,S).
\end{align*}
We assume for notational convenience that $X$ and $Y$ extend to i.i.d. sequences of length $n+m$, so that we may write $X_i$ for $i \geq n+1$.  By Jensen's inequality, it follows that
\begin{align*}
    &\mathbb{E} \left[ \sup_{f \in \cF} \left\{ \sum_{i = 1}^n \xi_i f(X_i) + \sum_{j = 1}^m \xi_{j+n} f(Y_j) \right\} \Bigl| \xi \right] \\
    &\quad \leq \mathbb{E} \left[ \sup_{f \in \cF} \left\{ \sum_{i = 1}^n \xi_i f(X_i) + \sum_{i = n+1}^{n+m} \xi_i f(X_i)  \right\} \Bigl| \xi \right] \\ 
     &\quad \quad +\mathbb{E}\left[\sup_{f \in \cF} \left\{ \sum_{j = 1}^{m} \xi_{j+n} f(Y_j) + \sum_{j = 1}^{n} \xi_j f(Y_j) \right\} \Bigl| \xi \right] \\
     &\quad \quad + \left| \sum_{i = 1}^{n} \xi_i \right| d_{\cF}(R,S).
\end{align*}
Taking expectations gives the first inequality. To prove the second inequality, it remains to control the quantity $\EE  [\left| \sum_{i = 1}^{n} \xi_i \right|]$. By the Cauchy-Schwarz inequality,
\[ \mathbb{E} \left[ \left| \sum_{i = 1}^{n} \xi_i \right| \right] \leq \mathbb{E} \left[ \Bigl( \sum_{i = 1}^{n} \xi_i \Bigr)^2 \right]^{1/2}. \]
By exchangeability,
\begin{align*}
  \mathbb{E} \left[ \Bigl( \sum_{i = 1}^{n} \xi_i \Bigr)^2 \right] &= n \mathbb{E}\left[ \xi_1^2 \right] + n(n-1) \mathbb{E}\left[ \xi_1 \xi_2 \right] \\
  &= n \mathbb{E}\left[ w_{\sigma(1)}^2 \right] + n(n-1) \mathbb{E}\left[ w_{\sigma(1)} w_{\sigma(2)} \right] \\
  &= \frac{n}{n+m} \sum_{i = 1}^{n+m} w_i^2 + \frac{n(n-1)}{(n+m)(n+m-1)} \sum_{i = 1}^{n+m} w_i \sum_{j \neq i} w_j \\
  &= \frac{n}{n+m} \sum_{i = 1}^{n+m} w_i^2 - \frac{n(n-1)}{(n+m)(n+m-1)} \sum_{i = 1}^{n+m} w_i^2\\
  &= \frac{n}{n+m} \left(1 - \frac{n-1}{n+m-1} \right) \sum_{i = 1}^n w_i^2 \\
  &= \frac{n}{n+m} \frac{m}{n+m-1} \left( \frac{1}{n} + \frac{1}{m} \right) \\
  &= \frac{1}{n+m-1}\;,
\end{align*}
where, for equality between the third and the fourth line, we used the fact that $\sum_{j\neq i} w_j=-w_i$. The proof is now complete.
\end{proof}

\begin{lemma}\label{lem_v_+}
Recall the following notation:
   \[v_+(Z) = \sup_{f \in \cF} \left\{ \sum_{i = 1}^{n+m} \left(f(Z_i) - \frac{1}{n+m} \sum_{j = 1}^{n+m} f(Z_j) \right)^2 \right\}\;.\]
   The function $v_+/4$ is $(2,0)$-self-bounding.
\end{lemma}

\begin{proof}
    Define, for all $i\in \left\{ 1,\ldots, n+m \right\}$, $Z^{(i)}=(Z_1,\ldots,Z_{i-1},Z_{i+1},\ldots,Z_{n+m})$ and
    \[
    v_i(Z^{(i)})=\frac{n+m-1}{n+m}\sup_{f \in \cF} \left\{ \sum_{k \neq i} \left(f(Z_k) - \frac{1}{n+m-1} \sum_{l\neq i} f(Z_l) \right)^2 \right\}\;.
    \]
Note that
\[
 v_+(Z) = \frac{1}{2(n+m)}\sup_{f \in \cF} \left\{ \sum_{i,j = 1}^{n+m} \left(f(Z_i) - f(Z_j) \right)^2 \right\}
\]
and, for any $i\in \left\{ 1,\ldots, n+m \right\}$,
\[
 v_i(Z^{(i)})= \frac{1}{2(n+m)}\sup_{f \in \cF} \left\{ {\sum_{\substack{k,\ell = 1\\k,\ell\neq i}}^{n+m}} \left(f(Z_k) - f(Z_\ell) \right)^2 \right\}
\]
Then, we can assume without loss of generality that there exists $\hat{f}\in \cF$ such that
\[
 v_+(Z) = \frac{1}{2(n+m)}\sum_{i,j = 1}^{n+m} \left(\hat{f}(Z_i) - \hat{f}(Z_j) \right)^2 \; .
\]    
This gives that, for any $i\in \left\{ 1,\ldots, n+m \right\}$,
    $$0\leq v_+(Z)-v_i(Z^{(i)}) \leq \frac{1}{n+m} \sum_{j = 1}^{n+m} \left(\hat{f}(Z_i) - \hat{f}(Z_j) \right)^2\leq 4$$ and $$\sum_{i=1}^n (v_+(Z)-v_i(Z^{(i)})) \leq 2 v_+(Z)\;.$$ Finally, the conclusion follows by dividing the functions $v_+$ and $v_i$, $i\in\left\{1,\ldots, n+m \right\}$, by $4$.
\end{proof}

\begin{lemma}
 The quantity $M_k(R)/k$ is non-increasing in $k\in \NN_*$.
\end{lemma}

\begin{proof}
    Consider an integer $k\geq 1$. Take $J\subset \left\{ 1,\ldots,k+1\right\}$ a random set of indices, with distribution that is uniform among the subsets of $\left\{ 1,\ldots,k+1\right\}$ of length $k$, and independent of the sample $(Z_1,\ldots,Z_{k+1})\sim R^{\otimes (k+1)}$. Then, it holds
    \begin{align*}
        M_k(R)=& \EE_{R^{\otimes k} }\left[\sup_{f\in \cF}\sum_{i=1}^{k}(f(Z_i)-\EE[f(Z_i)])\right]\\
        =& \EE_{J,R^{\otimes (k+1)} }\left[\sup_{f\in \cF}\sum_{i\in J}(f(Z_i)-\EE[f(Z_i)])\right]\\
        =& \EE_{J,R^{\otimes (k+1)} }\left[\sup_{f\in \cF}\sum_{i=1}^{k+1}\mathbb{I}_{i\in J}(f(Z_i)-\EE[f(Z_i)])\right]\\
         \geq & \EE_{R^{\otimes (k+1)} }\left[\sup_{f\in \cF}\EE_{J}\left[\sum_{i=1}^{k+1}\mathbb{I}_{i\in J} (f(Z_i)-\EE[f(Z_i)])\right]\right]\\
        = & \EE_{R^{\otimes (k+1)} }\left[\sup_{f\in \cF}\sum_{i=1}^{k+1}\PP(i\in J)(f(Z_i)-\EE[f(Z_i)])\right]\\
        = & \frac{k}{k+1}\EE_{R^{\otimes (k+1)} }\left[\sup_{f\in \cF}\sum_{i=1}^{k+1}(f(Z_i)-\EE[f(Z_i)])\right]\;,
    \end{align*}
    which gives the result.
\end{proof}

\subsection{Proofs of corollaries of Theorem \ref{th_test_separation_rate}}
\subsubsection{Proof of Corollary \ref{cor_wass_test}}
    It follows from the work of Fournier and Guillin (\cite[Theorem 1]{Fournier2015} with $p = 1$ and $q = r > 2$) that
    \[ \frac{1}{l} M_l(R) \leq C_0 V_r(R) \frac{L_l}{l^\gamma} \]
    for all $l \in \mathbb{N}$, any $r > 2$ and any dimension $d \geq 1$. 
    Moreover, Theorem \ref{th_test_separation_rate} applies since $\cF$ is a pointwise closed class of measurable functions valued in $[-1,1]$. We bound the terms of Equation \eqref{eq_thm_sep_rates} as follows. By Inequalities \eqref{ineq_W_dim2} and \eqref{ineq_W_dimsup3},
    \[ \frac{2}{n}(M_n(P)+M_n(Q))+\frac{2}{m}(M_m(P)+M_m(Q)) \leq 2C_0 \left( \frac{L_n}{n^\gamma} + \frac{L_m}{m^\gamma} \right) (V_r(P) + V_r(Q))\;. \]
    Moreover, we have seen that for two independent random vectors with distribution $R$,
    \[ \sigma_R^2(\cF) \leq \frac{1}{2} \mathbb{E}[d(X,X')^2] \leq \frac{1}{2} \mathbb{E} \left[ (d(X,\mu) + d(\mu,X'))^2 \right] \leq \mathbb{E}\left[ d(X,\mu)^2 + d(X',\mu)^2 \right] \]
    from which it follows that
    \[ \sigma^2_R(\cF) \leq 2V_2(R) \]
   and
    \[ \sqrt{2\left( \frac{\sigma^2_P(\cF)}{n} + \frac{\sigma^2_Q(\cF)}{m} \right)\log\left(\frac{1}{\delta}\right)} \leq 2 \sqrt{ \left(\frac{V_2(P)}{n} + \frac{V_2(Q)}{m} \right) \log\left(\frac{1}{\delta}\right)}\;. \]
    Consider now the sum $S$ under the square root on the last line of Equation \eqref{eq_thm_sep_rates}. Since
    \[ V \leq n\sigma^2_P(\cF) + m\sigma_Q^2(\cF) \leq 2(nV_2(P) + mV_2(Q)) \leq 2n(V_2(P) + \rho V_2(Q)) \]
    and
    \begin{align*}
       \frac{2m}{n(n+m)} M_n^2(P) + \frac{2n}{m(n+m)} M_m^2(Q) &\leq \frac{2m}{n(n+m)} 2n M_n(P) + \frac{2n}{m(n+m)} 2m M_m(Q) \\ 
       &\leq 4(M_n(P) + M_m(Q))\;, 
    \end{align*}
    it follows that
    \begin{align*}
        2\left(\frac{1}{n} + \frac{1}{m}\right)\sqrt{S} &\leq \frac{4}{n} \sqrt{38(M_n(P) + M_m(Q)) + 2n(V_2(P) + \rho V_2(Q)) + 4\log\left( \frac{1}{\delta} \right)} \\
        &\leq \frac{4}{\sqrt{n}} \sqrt{2(V_2(P) + \rho V_2(Q))} + \frac{4}{n} \sqrt{38C_0 \left( L_n n^{1-\gamma} V_r(P) + L_m (m)^{1-\gamma} V_r(Q) \right)} \\ 
        &\quad + \frac{8}{n} \sqrt{\log \left( \frac{1}{\delta} \right)} \\
        &\leq \frac{4}{\sqrt{n}} \sqrt{2(V_2(P) + \rho V_2(Q))} + \frac{C\sqrt{L_n}}{n^{\frac{1+\gamma}{2}}} + \frac{8}{n} \sqrt{\log \left( \frac{1}{\delta} \right)}\;,
    \end{align*}
    where $C$ depends on $C_0,\rho,V_r(P)$ and $V_r(Q)$.
    The remaining terms on the right-hand side of Equation \eqref{eq_thm_sep_rates} are of order
    \[ \frac{1}{n} \log \left( \frac{1}{\delta} \right). \]

\subsubsection{Proof of Corollary \ref{cor_MMD}}
\begin{proof}
Let $k \geq 2$ be an integer, $R$ be a probability distribution on $\cX$ and $Z_1,\ldots,Z_k$ be i.i.d random variables with common distribution $R$. Let $\mu_R$ be the kernel mean embedding of $R$. We apply Theorem \ref{th_test_separation_rate} to the class $\frac{1}{\kappa} \cF$. The result will follow by multiplying both sides of equation \eqref{eq_thm_sep_rates} by $\kappa$.
    A classical computation shows that
    \begin{align*}
        \kappa M_k(R) &= \mathbb{E} \left[ \norm{\sum_{i = 1}^k k(Z_i,\cdot) - \mu_R}_{\cH} \right] \\
        &\leq \mathbb{E} \left[ \norm{\sum_{i = 1}^k k(Z_i,\cdot) - \mu_R}_{\cH}^2 \right]^{1/2} \\
        &= \sqrt{k \mathbb{E} \left[ \norm{k(Z_1,\cdot) - \mu_R}_{\cH}^2 \right]} \\
        &\leq \sqrt{k \mathbb{E} \left[ \norm{k(Z_1,\cdot)}_{\cH}^2 \right]} \\
        &\leq \kappa \sqrt{k} 
    \end{align*}
    Thus,
    \[ \frac{2}{n} \left(M_n(P) + M_n(Q) \right) + \frac{2}{m} \left(M_m(P) + M_m(Q) \right) \leq 4\left( \frac{1}{\sqrt{n}} + \frac{1}{\sqrt{m}} \right). \]
\end{proof}

\subsection{Bounds for the expectation of the supremum of the bootstrap empirical process}\label{ssec_mean}

The following upper and lower bounds for the mean of the supremum of the exchangeable bootstrap empirical process are instrumental in our proofs related to sections \ref{sec_conf_reg} and \ref{sec_two_sample_test}. We believe that they are folklore results. Nonetheless, for the sake of completeness, we provide complete proofs. 

In particular, they are consequences of the most advanced results, that are due to Han and Wellner (\cite{HanWell:19}).


Let us start with a universal lower bound.
\begin{proposition} \label{prop_lower_mean_boot}
    Let $(\xi_i)_{1 \leq i \leq n}$ be a collection of weights of sum $0$, independent of the i.i.d. sample $X = (X_1,\ldots,X_n)$. Assume also that the weights have a common moment of order one, equal to $\EE[\vert \xi_1\vert]$, and a common moment of order one for their positive parts as well, that is $\mathbb{E}\left[ (\xi_1)_+ \right]=\ldots=\mathbb{E}\left[ (\xi_n)_+ \right]$. For any $t \in \cT$, let $\mu_t$ denote the common mean of the random variables $t(X_{i})$. We have that
    \begin{equation}\label{ineq_lower_bound}
         \mathbb{E}\left[ (\xi_1)_+ \right] \mathbb{E}\left[ \sup_{t \in \cT} \sum_{i = 1}^n (t(X_{i}) - \mu_t) \right] \leq  \mathbb{E}\left[ \sup_{t \in \cT} \sum_{i = 1}^n \xi_i t(X_{i}) \right]\;. 
    \end{equation}
    If moreover the process is symmetric in the sense that $(\mu_t - t(X_{1}))_{t \in \cT}$ is equal in distribution to $(t(X_{1}) - \mu_t)_{t \in \cT}$, then
    \begin{equation}\label{ineq_lower_bound_sym}
        \mathbb{E}\left[ |\xi_1| \right] \mathbb{E}\left[ \sup_{t \in \cT} \sum_{i = 1}^n (t(X_{i}) - \mu_t) \right] \leq  \mathbb{E}\left[ \sup_{t \in \cT} \sum_{i = 1}^n \xi_i t(X_{i}) \right] \;.
    \end{equation}
      
\end{proposition}

Proposition \ref{prop_lower_mean_boot} can be compared, for instance, with the lower bound of Lemma 2.9.1 in \cite{vandervaartWellner:23}. One may indeed notice that the proof of the latter lemma, which proceeds with similar arguments as our proof below, does not require that the weights be independent. One can also notice that exchangeability of the weights is not required in Proposition \ref{prop_lower_mean_boot}. 

\begin{proof}
First remark that, since $\sum_{i = 1}^n \xi_i = 0$, it holds for each $t \in \cT$, 
\[ \sum_{i = 1}^n \xi_i t(X_{i}) = \sum_{i = 1}^n \xi_i (t(X_{i}) - \mu_t), \]
which implies that
\[ \sup_{t \in \cT} \sum_{i = 1}^n \xi_i t(X_{i}) = \sup_{t \in \cT} \sum_{i = 1}^n \xi_i (t(X_{i}) - \mu_t).  \]
Let us now consider the case where the process $(t(X_{1}) - \mu_t)_{t \in \cT}$ is symmetric. Conditionally on $\xi$, for any signs $(\varepsilon_i)_{1 \leq i \leq n}$, the quantity
\[ \sup_{t \in \cT} \sum_{i = 1}^n \xi_i (t(X_{i}) - \mu_t)  \]
is equal in distribution to
\[ \sup_{t \in \cT} \sum_{i = 1}^n \xi_i  \varepsilon_i (t(X_{i}) - \mu_t) \]
and in particular to
\[ \sup_{t \in \cT} \sum_{i = 1}^n |\xi_i| (t(X_{i}) - \mu_t) \; .\]
Consequently, it holds
\[ \mathbb{E} \left[ \sup_{t \in \cT} \sum_{i = 1}^n \xi_i t(X_{i}) \right] = \mathbb{E}\left[ \sup_{t \in \cT} \sum_{i = 1}^n |\xi_i| (t(X_{i}) - \mu_t)  \right]. \]
By Jensen's inequality, it follows that
\begin{align*}
    \mathbb{E} \left[ \sup_{t \in \cT} \sum_{i = 1}^n \xi_i t(X_{i}) \right] &\geq \mathbb{E}\left[ \sup_{t \in \cT} \sum_{i = 1}^n \mathbb{E}[|\xi_i|] (t(X_{i}) - \mu_t)  \right] \\
    &= \mathbb{E}[|\xi_1|] \mathbb{E}\left[ \sup_{t \in \cT} \sum_{i = 1}^n (t(X_{i}) - \mu_t)  \right] \; ,
\end{align*}
thus proving Inequality (\ref{ineq_lower_bound_sym}).
    It remains to consider now the general case.
    Let $A = \{i : \xi_i > 0 \}$. This set depends only on $\xi$,  it is therefore independent of $X$. Conditionally on $\xi$, $ (t(X_{i}))_{i \in A, t\in \cT}$ and $(t(X_{i}))_{i \in A^c}$ are independent collections of i.i.d. random processes with mean $(\mu_t)_{t \in \cT}$. Thus, by Jensen's inequality,
    \begin{align*}
        &\mathbb{E} \left[ \sup_{t \in \cT} \sum_{i = 1}^n \xi_i (t(X_{i}) - \mu_t) \bigl| \xi, (t(X_{i}))_{i \in A, t \in \cT} \right] \\
        &\quad \geq \sup_{t \in \cT} \mathbb{E} \left[  \sum_{i = 1}^n \xi_i (t(X_{i}) - \mu_t) \bigl| \xi, (t(X_{i}))_{i \in A, t \in \cT} \right] \\
        &\quad = \sup_{t \in \cT} \sum_{i \in A} \xi_i (t(X_{i}) - \mu_t) \\
        &\quad = \sup_{t \in \cT} \sum_{i = 1}^n (\xi_i)_+ (t(X_{i}) - \mu_t).
    \end{align*}
 Taking expectations, it follows that
 \[ \mathbb{E}\left[ \sup_{t \in \cT} \sum_{i = 1}^n \xi_i t(X_{i}) \right] \geq \mathbb{E}\left[ \sup_{t \in \cT} \sum_{i = 1}^n (\xi_i)_+ (t(X_{i}) - \mu_t) \right].  \]
 Another application of the conditional version of Jensen's inequality yields
 \begin{align*}
    &\mathbb{E}\left[ \sup_{t \in \cT} \sum_{i = 1}^n (\xi_i)_+ (t(X_{i}) - \mu_t) \bigl| X \right] \\
    &\quad \geq \sup_{t \in \cT}  \mathbb{E}\left[ \sum_{i = 1}^n (\xi_i)_+ (t(X_{i}) - \mu_t) \bigl| X \right] \\
    &\quad = \sup_{t \in \cT} \sum_{i = 1}^n \mathbb{E}[(\xi_i)_+] (t(X_{i}) - \mu_t) \\
    &\quad = \mathbb{E}\left[ (\xi_1)_+ \right] \sup_{t \in \cT} \sum_{i = 1}^n (t(X_{i}) - \mu_t) ,
 \end{align*}
 which gives Inequality (\ref{ineq_lower_bound}) and concludes the proof.
\end{proof}

When the weights are bounded, it is also possible to give simple upper bounds by a symmetrization argument. Such results will be sufficient for our needs. Further results related to unbounded weights can be found in \cite{HanWell:19}.
\begin{proposition}\label{prop_upper_mean_boot}
   Let $(\xi_i)_{1 \leq i \leq n}$ be weights of sum $0$, independent from the i.i.d. sample $X = (X_{1},\ldots,X_{n})$ and bounded by a constant $b$ ($|\xi_i| \leq b$ $a.s$ for each $i$). The weights need not be exchangeable. For any $t \in \cT$, let $\mu_t$ denote the common mean of the random variables $t(X_{i})$. If the process $(t(X_{1}))_{t \in \cT}$ is symmetric in the sense that $(\mu_t - t(X_{1}))_{t \in \cT}$ is equal in distribution to $(t(X_{1}) - \mu_t)_{t \in \cT}$, then
   \begin{equation}\label{ineq_upper_bound_sym}
       \mathbb{E}\left[ \sup_{t \in \cT} \sum_{i = 1}^n \xi_i t(X_{i}) \right] \leq b \mathbb{E}\left[ \sup_{t \in \cT} \sum_{i = 1}^n (t(X_{i}) - \mu_t) \right] \;. 
   \end{equation}
    In the general case, 
   \begin{equation}\label{ineq_upper_bound}
       \mathbb{E}\left[ \sup_{t \in \cT} \sum_{i = 1}^n \xi_i t(X_{i}) \right] \leq 2b \mathbb{E}\left[ \sup_{t \in \cT} \sum_{i = 1}^n (t(X_{i}) - \mu_t) \right]\;,. 
   \end{equation}
    
\end{proposition}

\begin{proof}
   As in the proof of Proposition \ref{prop_lower_mean_boot}, we note that the relation $\sum_{i = 1}^n \xi_i = 0$, for each $t \in \cT$, gives
\[ \sum_{i = 1}^n \xi_i t(X_{i}) = \sum_{i = 1}^n \xi_i (t(X_{i}) - \mu_t)\;, \]
which implies that
\[ \sup_{t \in \cT} \sum_{i = 1}^n \xi_i t(X_{i}) = \sup_{t \in \cT} \sum_{i = 1}^n \xi_i (t(X_{i}) - \mu_t)\;.  \] 
Again, we start with the symmetric case. Let $\varepsilon = (\varepsilon_i)_{1 \leq i \leq n}$ be a vector of i.i.d. Rademacher weights. By independence of $\xi$ and $X$, the processes
\[ (\xi_i (t(X_{i}) - \mu_t) )_{1 \leq i \leq n, t \in \cT}, \ (\xi_i \varepsilon_i (t(X_{i}) - \mu_t) )_{1 \leq i \leq n, t \in \cT}  \]
are equal in distribution. This yields
\[ \mathbb{E}\left[ \sup_{t \in \cT} \sum_{i = 1}^n \xi_i t(X_{i}) \right]=\mathbb{E}\left[ \sup_{t \in \cT} \sum_{i = 1}^n \xi_i (t(X_{i})-\mu_t) \right] = \mathbb{E}\left[ \sup_{t \in \cT} \sum_{i = 1}^n \xi_i \varepsilon_i (t(X_{i}) - \mu_t) \right].  \]
Then, by the contractivity property of the Rademacher complexity and the symmetry assumption again,
\begin{align*}
    \mathbb{E}\left[ \sup_{t \in \cT} \sum_{i = 1}^n \xi_i t(X_{i}) \right] &\leq b \mathbb{E}\left[ \sup_{t \in \cT} \sum_{i = 1}^n \varepsilon_i (t(X_{i}) - \mu_t) \right] \\
    &= b \mathbb{E}\left[ \sup_{t \in \cT} \sum_{i = 1}^n (t(X_{i}) - \mu_t) \right] \;,
\end{align*}
which gives Inequality (\ref{ineq_upper_bound_sym}). Consider now the general case.
Let $X^*$ be an independent copy of $X$, also independent of $\xi$. By Jensen's inequality,
\begin{align*}
   \sup_{t \in \cT} \sum_{i = 1}^n \xi_i (t(X_{i}) - \mu_t) &= \sup_{t \in \cT} \mathbb{E}\left[ \sum_{i = 1}^n \xi_i (t(X_{i}) - t(X^*_{i})) \bigr| X \right] \\
   &\leq \mathbb{E}\left[ \sup_{t \in \cT} \sum_{i = 1}^n \xi_i (t(X_{i}) - t(X^*_{i})) \bigr| X \right]\;.
\end{align*}
The process $(t(X_{1}) - t(X^*_{1}))_{t \in \cT}$ is symmetric and zero-mean, hence by the previous case,
\begin{align*}
   \mathbb{E}\left[ \sup_{t \in \cT} \sum_{i = 1}^n \xi_i t(X_{i}) \right] 
   &\leq  b \mathbb{E}\left[ \sup_{t \in \cT} \sum_{i = 1}^n (t(X_{i}) - t(X^*_{i})) \right] \\
   &\leq 2b\mathbb{E}\left[ \sup_{t \in \cT} \sum_{i = 1}^n (t(X_{i}) - \mu_t) \right],
\end{align*}
using the sub-additivity of the supremum. 
\end{proof}

\bibliography{chern}
\end{document}